\UseRawInputEncoding
\documentclass{amsart}
\usepackage{mathrsfs}
\usepackage{graphicx}
\usepackage{amsmath}
\usepackage{amscd}
\usepackage{cite}
\usepackage{hyperref}
\usepackage{epsfig}
\usepackage{subfigure}
\usepackage{caption}
\usepackage{tikz}
  \usepackage{float}
  \usepackage{amssymb}
\usepackage{amsfonts}
\usepackage[all,cmtip]{xy}
\usepackage{appendix}

\newtheorem{Example}{Example}[section]
\newtheorem{Definition}{Definition}[section]
\newtheorem{Theorem}{Theorem}[section]
\newtheorem{Theorem/Definition}{Theorem/Definition}[section]
\newtheorem{Proposition}{Proposition}[section]
\newtheorem{Lemma}{Lemma}[section]
\newtheorem{Corollary}{Corollary}[section]

\newcommand{\pd}{\partial}
\newcommand{\bC}{{\mathbb C}}
\newcommand{\bE}{{\mathbb E}}

\newcommand{\bQ}{{\mathbb Q}}
\newcommand{\bR}{{\mathbb R}}

\newcommand{\cD}{{\mathcal D}}

\newcommand{\cF}{{\mathcal F}}
\newcommand{\cG}{{\mathcal G}}

\newcommand{\cM}{{\mathcal M}}

\newcommand{\half}{\frac{1}{2}}
\newcommand{\cV}{{\mathcal V}}

\newcommand{\Mbar}{\overline{\cM}}

\newcommand{\wF}{{\widehat F}}
\newcommand{\wcF}{{\widehat{\mathcal{F}}}}
\newcommand{\tF}{{\widetilde F}}

\newcommand{\hpd}{{\hat{\pd}}}

\newcommand{\brac}[2]{\begin{bmatrix} {#1} \\ {#2} \end{bmatrix}}

\newcommand{\be}{\begin{equation}}
\newcommand{\ee}{\end{equation}}
\newcommand{\bea}{\begin{eqnarray}}
\newcommand{\ben}{\begin{eqnarray*}}
\newcommand{\een}{\end{eqnarray*}}
\newcommand{\eea}{\end{eqnarray}}

\DeclareMathOperator{\Aut}{Aut} \DeclareMathOperator{\Id}{id}

\DeclareMathOperator{\val}{val}

\usepackage{color}

\definecolor{yellow}{rgb}{1,1,0}
\definecolor{orange}{rgb}{1,.7,0}
\definecolor{red}{rgb}{1,0,0} \definecolor{green}{rgb}{0,1,1}
\definecolor{white}{rgb}{1,1,1}

\definecolor{A}{rgb}{.75,1,.75}

\theoremstyle{remark}
\newtheorem{Remark}{Remark}[section]

\begin{document}

\newtheorem{myDef}{Definition}
\newtheorem{thm}{Theorem}
\newtheorem{eqn}{equation}

\title[Orbifold Euler Characteristics of $\Mbar_{g,n}$]
{Orbifold Euler Characteristics of $\Mbar_{g,n}$}

\author{Zhiyuan Wang}
\address{School of Mathematical Sciences\\
Peking University\\Beijing, 100871, China}
\email{zhiyuan19@math.pku.edu.cn}

\author{Jian Zhou}
\address{Department of Mathematical Sciences\\
Tsinghua University\\Beijing, 100084, China}
\email{jianzhou@mail.tsinghua.edu.cn}

\begin{abstract}
We solve the problem of the computation of the orbifold Euler characteristics of $\Mbar_{g,n}$.
We take the works of Harer-Zagier \cite{hz} and Bini-Harer \cite{bh}
as our starting point,
and apply the formalisms developed in \cite{wz} and \cite{zhou1}
to this problem.
These formalisms are typical examples of mathematical methods inspired by
quantum field theories.
We also present many closed formulas and some numerical data.
In genus zero the results are related to Ramanujan polynomials,
and in higher genera we get recursion relations almost identical
to the recursion relations for Ramanujan polynomials but with different initial values.
We also show that the generating series given by the orbifold Euler characteristics of $\Mbar_{g,n}$
is the logarithm of the KP tau-function of the topological 1D gravity evaluated at the times
given by the orbifold Euler characteristics of $\cM_{g,n}$.
Conversely, the logarithm of this tau-function evaluated at the times
given by certain generating series of the orbifold Euler characteristics of $\Mbar_{g,n}$ is
a generating series of the orbifold Euler characteristics of $\cM_{g,n}$.
This is a new example of open-closed duality.
\end{abstract}

\maketitle

\tableofcontents

\section{Introduction}

\subsection{The computations of $\chi(\Mbar_{g,n})$}

The problem of the computation of the orbifold Euler characteristics $\chi(\Mbar_{g,n})$
of the Deligne-Mumford moduli spaces $\Mbar_{g,n}$ of stable curves
is a long standing problem in algebraic geometry.
In this work we will solve this problem using some methods we have developed in
an earlier work \cite{wz} inspired by the quantum field theory techniques developed by string
theorists in the study of holomorphic anomaly equations.
Such equations have played a crucial role in the theory of mirror symmetry
in higher genera.
By extracting the salient features of the theory of holomorphic
anomaly equations,
we formulate a notion of abstract quantum field theory in \cite{wz}
and develop some recursion relations in this formalism.
It turns out that this general formalism includes various problems and related
recursion relations as special cases.
We will show that this formalism is also applicable to the problem of computing
$\chi(\Mbar_{g,n})$ and it provides an effective algorithm for their computations.
We will present various closed formulas which are not accessible by other methods in the literature.

The idea of using quantum field theory techniques to solve problems in algebraic geometry is
certainly not new.
There are many well-known examples in the literature,
and some of them are famous work on the problem we deal with in this work.
Let us briefly review some of them.

The moduli space $\cM_{g,n}$ of smooth Riemann surfaces
of genus $g$ with $n$ marked points $(2g-2+n>0)$
is an orbifold of dimension $3g-3+n$,
and the orbifold Euler characteristic of $\cM_{g,n}$
is given by the famous Harer-Zagier formula:
\ben
\chi({\cM}_{g,n})=(-1)^{n}\cdot\frac{(2g-1)B_{2g}}{(2g)!}(2g+n-3)!,
\een
which is first proved by Harer and Zagier \cite{hz},
and by Penner \cite{pe}.
These authors have introduced the methods of Hermitian matrix models
into the study of moduli spaces of algebraic curves.
A different proof is given by Kontsevich \cite[Appendix D]{kon1},
in an appendix of his famous work on the proof of Witten's Conjecture \cite{wit1}.
That work can be considered as a development of \cite{hz} and \cite{pe}
in the sense that
quantum field theory techniques such as summations over Feynman graphs also
play an important role.

In \cite{dm},
Deligne and Mumford construct a natural compactification
$\overline\cM_{g,n}$ of $\cM_{g,n}$
by allowing double points on curves.
The compactification $\overline\cM_{g,n}$ is the moduli space
of stable curves of genus $g$ with $n$ marked points,
and it carries a natural stratification
which can be described by the dual graphs of stable curves:
\ben
\overline\cM_{g,n}=\bigsqcup_{\Gamma}\cM_\Gamma,
\een
where the disjoint union is indexed by the set of
all connected stable graphs of genus $g$ with $n$ external edges,
and $\cM_\Gamma$ is the moduli space of stable curves
whose dual graph is $\Gamma$.
A natural question is the computations of the orbifold Euler characteristic
and ordinary Euler characteristic of $\overline\cM_{g,n}$.
Some results about the ordinary Euler characteristics for small $g$
have already been given by different authors.
For $g=0$,
see \cite{get1, gls, ke, ma2}.
For $g=1$,
see \cite{get2}.
For $g=2$,
see \cite{bgp, fv, get3}.
For $g=3$,
see \cite{gl}.
There are also some excellent results on the
orbifold Euler characteristics $\chi(\Mbar_{g,n})$ in literatures.
In Manin's computation of $\chi(\Mbar_{0,n})$,
quantum field theory techniques have been applied to get the following result:
let
\ben
\chi(t) := t + \sum_{n=2}^\infty \chi(\Mbar_{0,n+1}) \frac{t^n}{n!}
\een
be the generating series at genus zero,
then $\chi(t)$ is the unique solution of either one of the following two equations:
\ben
&& (1+ \chi) \log (1+\chi) = 2\chi - t, \\
&& (1+t-\chi) \chi_t = 1+ \chi.
\een
See \cite{gls} for some combinatorial results
related to the solution of the first equation.

The common features of the works of
Harer-Zagier, Penner and Manin are the following.
They all formulated the geometric problem as finding
a summation over Feynman graphs.
In the case of Harer-Zagier and Penner,
the relevant graphs are the fat graphs in the sense
of 't Hooft \cite{tH} and these graphs lead them naturally to
generalized Gaussian integrals over the space
of Hermitian matrices;
and in the case of Manin the relevant graphs are marked trees
and this leads to a Gaussian integral
on a one-dimensional space.
The next step of all these works is to apply
the Laplace method of finding asymptotic expansions
of these formal integrals.

The orbifold Euler characteristic of $\overline\cM_{g,n}/S_n$
has been studied by this strategy
by Bini-Harer \cite{bh}.
First, they notice that  $\overline\cM_{g,n}/S_n$
can also be written in terms a sum over stable graphs
\ben
\chi(\overline{\cM}_{g,n}/S_n)=\sum_{\Gamma\in \cG^c_{g,n}}
\frac{1}{|\Aut(\Gamma)|}\prod_{v\in V(\Gamma)}\chi(\cM_{g_v,\val_v}).
\een
It follows that the generating series of $\chi(\Mbar_{g,n})$
has a formal integral representation
for which one can apply the Laplace method to
obtain the asymptotic behaviors
(see Bini-Harer \cite[Theorem 3.2, Theorem 3.3]{bh}).

Unfortunately very few numerical data and
no closed formula for $\chi(\overline{\cM}_{g,n})$
have been obtained by the above methods.
The two ways to compute $\chi(\Mbar_{g,n})$ in that work,
i.e.,
summing over all stable graphs or expanding formal Gaussian integrals,
are both too complicated to carry out specific computations.
In fact,
the complexity of finding all the relevant Feynman graphs
and their automorphism groups
makes it impossible to compute $\chi(\Mbar_{g,n}/S_n)$
using the above graph sum formula directly
unless $g$ and $n$ are extremely small
(see eg. \cite[Appendix]{wz} for all graphs with $(g,n)=(2,2)$ and $(3,0)$);
and expanding the logarithm of a formal Gaussian integral
is no less complicated than writing down all graphs
since it involves summing over all partitions of some integers.
It is not surprising that rewriting the Feynman sum as
a logarithm of a formal integral may not easily lead to
the solution of the problem of finding closed expressions.

In quantum field theory,
deriving some recursion relations is
a general strategy to solve the problem of evaluating the Feynman sums.
This will also be our strategy to solve the problem of computing $\chi(\Mbar_{g,n})$.
We will take the works of Harer-Zagier \cite{hz} and Bini-Harer \cite{bh}
as our starting point
and develop a more computable method
to calculate $\chi(\Mbar_{g,n})$.
Because this problem can be formulated both as
a sum over graphs and as a formal Gaussian integrals,
we will attack this problem from both points of view.
A general theory for deriving recursion relations
for sums over graphs have been developed
by us in an earlier work \cite{wz}.
A general theory of evaluating formal Gaussian integrals
has been developed by the second author
in \cite{zhou1}.
In this article, we will solve the problem of computing $\chi(\Mbar_{g,n})$
by applying  the results in these works.

\subsection{Abstract quantum field theory and its realizations}
Let us first review
the formalism of the abstract quantum field theory and its realizations,
which yields various types of recursion relations that provide effective
algorithm for both the concrete computations and for deriving closed formulas
involving summation over stable graphs.

This formalism is developed by the authors in a previous work \cite{wz}.
In that work,
we have construct an abstract QFT using the diagrammatics of stable graphs.
Our original intention in that work is to
understand the mathematical structures of
the BCOV holomorphic anomaly equation \cite{bcov1, bcov2},
see also \cite{abk,ey,emo,gkmw,kz,wit2,yy}.
Inspired by these physics literatures,
we define the abstract free energy $\widehat{\cF}_g$ ($g\geq 2$)
to be formal summations of stable graphs of genus $g$ without external edges,
and the abstract $n$-point functions $\widehat{\cF}_{g,n}$ ($2g-2+n>0$)
to be formal summations of stable graphs of genus $g$ with $n$ external edges
(see \eqref{abs-fe} and \eqref{abs-n-pt}).
We also derive some recursion relations for these functions
in terms of the edge-cutting operator $K$ and
edge-adding operators $\cD=\pd+\gamma$
(see Lemma \ref{lem-original-D} and Theorem \ref{thm-original-rec}).
It is worth mentioning that the edge-cutting and edge-adding operators
are constructed inspired by the stratification of $\Mbar_{g,n}$.

Another part of such a general formalism is the notion of
a `realization' of the abstract QFT.
As input we need a collection of functions $F_{g,n}(t)$ with $2g-2+n > 0$,
and a suitable formal variable $\kappa$ we take as propagator,
and from which we use the Feynman rules to construct a collection of functions
$\widehat{F}_{g,n}(t,\kappa)$.
Under suitable conditions,
the recursion relations for the abstract $n$-point functions $\widehat{\cF}_{g,n}$
will induce the recursion relations for $\widehat{F}_{g,n}(t,\kappa)$.
In this way one can unify various recursion relations in the literature
as special examples.

We will review some preliminaries on this formalism in \S \ref{sec2}.

\subsection{Operator formalism for evaluations of formal Gaussian integrals}

One important feature of our formalism of realization of abstract QFT is that
we can express the resulting partition function as
a formal Gaussian integral in one dimension
(similar to the integral of Bini-Harer).
In \cite{zhou1} various methods have been developed to calculate the partition functions.
One of these methods is to expand the exponential in the integrand.
This converts the summation over Feynman graphs into a summation over partitions,
or pictorially,
over Young diagrams.
For the application of this idea to our problem,
see \eqref{eqn:Young}.
Using a recursion developed in the abstract QFT formalism,
we derive an operator formalism to compute the summation over Young diagram in
\S \ref{sec:Young}.
This involves operations on Young diagrams similar to
that in Littlewood-Richardson rule.

It turns out that this operator formalism appears in \cite{zhou1} in the setting of
deriving W-constraints for the partition function of topological 1D gravity,
i.e., formal Gaussian integrals in one dimension.
There are some easily derived constraints of this partition function,
called the flow equations and the polymer equation,
from which one can derive the Virasoro constraints (see e.g. \cite[\S 5-\S 6]{zhou1}).
They can also be used to derive operator formulas for computing the partition functions.
For their applications to $\chi(\Mbar_{g,0})$,
see \S \ref{sec-orb-genus0}.

\subsection{Descriptions of our results}

Let us briefly describe some of the results we obtain in this paper
using the methods reviewed in last two subsections.

\subsubsection{Recursion relations}

Since the orbifold Euler characteristics of $\cM_{h,m}$ have been explicitly computed
in the literatures \cite{hz, pe},
we will use them to construct $F_{g,n}(t)$
in our formalism of realization of the abstract QFT.
Of course if one has a way to compute other orbifold characteristics of $\cM_{g,n}$,
then our method is applicable to the corresponding orbifold characteristics of $\Mbar_{g,n}$.
We consider the following realization of the abstract QFT.
The abstract $n$-point function $\widehat{\cF}_{g,n}$
is realized by:
\ben
\chi_{g,n}(t,\kappa):=\sum_{\Gamma\in \cG^c_{g,n}}
\frac{\kappa^{|E(\Gamma)|}}{|\Aut(\Gamma)|}\cdot
\prod_{v\in V(\Gamma)}\biggl(
\chi(\cM_{g_v,n_v})\cdot t^{\chi(C_{g_v,\val_v})}
\biggr),
\een
where $g_v$ and $\val_v$ are the genus and valence of
the vertex $v$ respectively,
and $\chi(C_{g_v,\val_v})=2-2g-n$ is the Euler characteristic of
a Riemann surface $C_{g_v,\val_v}$ of genus $g$ with $n$ punctures.
By Bini-Harer's graph sum formula we have:
\be
\chi_{g,n}(1,1)=\chi(\overline\cM_{g,n}/S_n)
=\frac{1}{n!}\chi(\Mbar_{g,n}),
\ee
and so we can reduce the computations of
$\chi(\Mbar_{g,n})$ to that of $\chi_{g,n}(t,\kappa)$.
The functions $\chi_{g,n}(t,\kappa)$
will be called the `refined orbifold Euler characteristics' in this paper.
By our general formalism we derive the following recursion relations:
For $2g-2+n>0$, we have
\bea
&& D\chi_{g,n}=(n+1)\chi_{g,n+1}, \label{eqn:Lin} \\
&& \frac{\pd}{\pd\kappa}\chi_{g,n}=\frac{1}{2}(DD\chi_{g-1,n}
+\sum_{\substack{g_1+g_2=g,\\n_1+n_2=n}}
D\chi_{g_1,n_1}D\chi_{g_2,n_2}). \label{eqn:Quad}
\eea
where $D$ is a differential operator defined by:
\be
D\chi_{g,n}
=\biggl(\frac{\pd}{\pd t}+\kappa^2 t^{-1}\cdot\frac{\pd}{\pd \kappa}
+n\cdot \kappa t^{-1}\biggr)\chi_{g,n}.
\ee
Using these equations,
one can either reduce the number $n$ or reduce the genus $g$,
hence the problem of computations of $\chi(\Mbar_{g,n})$ is solved.
It is very easy to automate the concrete computations
by a computer algebra system.
We will present various examples in the text
(see eg. \S \eqref{sec-quadr-rec}).

\subsubsection{Connection to configuration spaces and Ramanujan polynomials}

Using our recursion relations \eqref{eqn:Lin} and \eqref{eqn:Quad},
we will derive the following functional equation in genus zero:
\be
\kappa(1+\chi)\log(1+\chi)=(\kappa+1)\chi-y
\ee
for the following generating series of refined orbifold Euler characteristics of $\Mbar_{0,n}$:
\be
\chi(y,\kappa):=y+\sum_{n=3}^{\infty}n\kappa y^{n-1}\cdot \chi_{0,n}(1,\kappa).
\ee
When $\kappa=1$,
this recovers Manin's result \cite{ma2} mentioned above.
It is very striking that when $\kappa$ is taken to be a positive integer $m$,
the above functional equation is related to the formula for Euler characteristics
of Fulton-MacPherson compactifications of configuration spaces of points on complex manifold $X$,
also due to Manin \cite{ma2},
where $m$ is the complex dimension of $X$.
We do not have an explanation for this mysterious coincidence.

Furthermore,
we  derive similar functional equations by the same method for higher genera.
The results are formulated in Theorem \ref{thm:Functional1} and Theorem \ref{thm:Functional2}.

We also make a connection of the linear recursion relation \ref{eqn:Lin}
to Ramanujan polynomials \cite{be, bew} and Shor \cite{sh} polynomials,
hence to their related combinatorial interpretations in terms of
counting trees that dates back to Cayley.
See \S \ref{sec:Rama} for details.

\subsubsection{Results for $\chi(\Mbar_{g,0})$}

By \eqref{eqn:Quad} one may get:
\begin{equation*}
\begin{split}
\widetilde\chi_{g,0}=&\frac{B_{2g}}{2g(2g-2)}+\frac{1}{2}\int_{0}^{\kappa}
\biggl[\biggl((3-2g)+\kappa^2\frac{\pd}{\pd\kappa}+2\kappa-\frac{1}{6}\biggr)
\biggl((4-2g)+\kappa^2\frac{\pd}{\pd\kappa}\biggr)\widetilde\chi_{g-1,0}
\\
&\qquad+\sum_{r=2}^{g-2}
\biggl((2-2r)+\kappa^2\frac{\pd}{\pd\kappa}\biggr)\widetilde\chi_{r,0}
\cdot\biggl((2-2g+2r)+\kappa^2\frac{\pd}{\pd\kappa}\biggr)\widetilde\chi_{g-r,0}
\biggr]d\kappa
\end{split}
\end{equation*}
for $g\geq 3$,
where $\widetilde{\chi}_{g,n}(\kappa):=\chi_{g,n}(1,\kappa)$
is a polynomial in $\kappa$ of degree $3g-3$.
This provides a very effective algorithm for specific computations.

The refined orbifold Euler characteristics $\chi_{g,n}(t, \kappa)$ has the following form:
\be
\chi_{g,n}(t,\kappa)=
\sum_{k=0}^{3g-3+n}a_{g,n}^k\kappa^k\cdot (\frac{1}{t})^{n-2+2g}.
\ee
Consider the generating series $G_k(z):=\sum_{g\geq 2}a_{g,0}^k\cdot z^{2-2g}$
of the coefficients in the case of $n=0$.
We express their generating series as a formal Gaussian integral
(see Corollary \ref{cor-gk-barnes}):
\be
\label{eq-intro-Gauss-G}
\sum_{k\geq 0} \lambda^{2k} G_k(z) = \log
\bigg( \frac{1}{\sqrt{2\pi\lambda^2}}
\int G(x+z+1)\cdot \exp (\widetilde S(x,z)) dx
\bigg),
\ee
where $G(z)$ is the Barnes G-function, and
\begin{equation*}
\begin{split}
\widetilde S(x,z)=&-\frac{\lambda^{-2}}{2}x^2
-\biggl( \zeta'(-1) + \frac{z}{2}\log (2\pi) + \biggl(\frac{z^2}{2}-\frac{1}{12}\biggr)\log z
- \frac{3z^2}{4} \biggr)\\
&+x\bigg(-\frac{1}{2} \log(2\pi) -z\log z +z\bigg)
-\frac{x^2}{2} \log z.
\end{split}
\end{equation*}
Moreover,
we prove that $G_k(z)$ can be expressed in terms of some generating series
$\{V_n(z)\}$ of $\chi(\cM_{g,n})$ in a simple way for every $k\geq 1$,
where $V_n(z)$ is defined by:
\be
V_n(z) = \sum_{g=0}^\infty \chi(\cM_{g,n}) z^{2-2g-n}
\ee
for $n \geq 3$ (for the modifications when $n < 3$,
see Definition \ref{def:V}).
Using the Harer-Zagier formula,
we are able to show that $V_n$ can be expressed in terms of the Barnes $G$-function $G(z)$:
\begin{equation*}
\begin{split}
&V_0(z) \sim \log G(z+1)
-\biggl( \zeta'(-1) + \frac{z}{2}\log (2\pi) + \biggl(\frac{z^2}{2}-\frac{1}{12}\biggr)\log z
- \frac{3z^2}{4} \biggr),\\
&V_1(z) \sim
\frac{d}{dz} \log  G(z+1)
-\frac{1}{2} \log(2\pi) -z\log z +z,\\
&V_2(z) \sim \frac{d^2}{dz^2} \log  G(z+1)-\log z,\\
&V_n (z) \sim \frac{d^n}{dz^n} \log  G(z+1),
\qquad \forall n\geq 3.
\end{split}
\end{equation*}
Here $\sim$ indicates that even though as power series,
the radii of convergence of $V_n(z)$ are all zero,
nevertheless as $z \to \infty$,
they are the asymptotic series of some analytic functions.
Let $v_n$ be a family of formal variables with $\deg(v_n):=n$,
then there exists a family of polynomials
$H_k(v_1,\cdots,v_{2k})$ such that $\deg(H_k)=2k$ and:
\ben
G_k(z)=
H_k(v_1,v_2,\cdots)\big|_{v_n=V_n(z)}.
\een
See \S \ref{sec4} for various methods to compute these polynomials based
on the above formal Gaussian integral representation of $G_k(z)$
and the theory of topological 1D gravity.

\subsubsection{Orbifold Euler characteristics of $\Mbar_{g,0}$ and KP hierarchy}

Another remarkable consequence of \eqref{eq-intro-Gauss-G} is that since
it is the specialization of the partition function of the topological 1D gravity,
it is the evaluation of a $\tau$-function of the KP hierarchy evaluated at
the following time variables:
\be
T_n = \frac{1}{n!} V_n(z).
\ee
See \S \ref{sec:KP} for details and the generalization to $\chi(\Mbar_{g,n})$ for general $n$.
It is interesting to find a geometric model that describes
all the KP flows away from this special choice of values of the time variables.

\subsubsection{Explicit formulas for $\chi(\Mbar_{g,n})$}

In \S \ref{sec:Operator} we have derived an operator formalism that computes $\chi(\Mbar_{g,n+1})$
from $\chi(\Mbar_{g,n})$ by reformulating a linear recursion relation from abstract QFT.
We are able to use it to derive some closed formulas for $\chi(\Mbar_{g,n})$.

For example,
in genus zero,
we have:
\ben
\chi(\Mbar_{0,n})=n!\cdot
\biggl[\sum_{k=0}^{n-3}A_k(x)
\biggr]_n,
\een
where $[\cdot]_n$ means the coefficient of $x^n$,
and $A_k(x)$ is sequence of power series in $x$ independent of $n$,
defined as follows:
\ben
&& A_0=\frac{1}{2}(1+x)^2\log(1+x)-\frac{1}{2}x-\frac{3}{4}x^2, \\
&& A_1 = \half + (1+x) (\log(1+x)-1) + \half (1+x)^2 (\log(1+x)-1)^2,
\een
and for $k\geq 2$,
\begin{equation*}
A_k=\sum_{m=2-k}^{2}\frac{(x+1)^m}{(m+2)!}\sum_{l=0}^{k-m-1}
\frac{\big(\log(x+1)-1\big)^l}{l!}
e_{l+m}(1-m,2-m,\cdots,k-m-1),
\end{equation*}
where $e_l$ are the elementary symmetric polynomials.
It turns out that for every $g\geq 1$,
$\chi(\Mbar_{g,n})$ all has an expression of this form.
In fact, $\chi(\Mbar_{g,n})$ is given by the following formula:
\ben
\chi(\Mbar_{g,n})=n!\cdot\sum_{k=0}^{3g-3+n}a_{g,n}^k=n!\cdot
\biggl[\sum_{k=0}^{3g-3+n}\sum_{p=0}^{3g-3}a_{g,0}^p A_{g,k}^p(x)\biggr]_n
\een
for $g\geq 2$,
where the closed formulas for $A_{g,k}^p(x)$ can be written down explicitly.
See \S \ref{sec5} for details.

\subsubsection{Open-closed duality of orbifold Euler characteristics}

We describe a new version of open-closed duality that represents the orbifold Euler characteristics
of $\cM_{g,n}$ and $\Mbar_{g,n}$ in terms of each other
via inversion formulas.
More precisely,
we have the following formula:
\be
\chi(\cM_{g,n})=n!\cdot \sum_{\Gamma\in \cG^c_{g,n}}
\frac{(-1)^{|E(\Gamma)|}}{|\Aut(\Gamma)|}\prod_{v\in V(\Gamma)}\chi(\Mbar_{g_v,\val_v}).
\ee
In a subsequent work \cite{wz4},
we give an interpretation of this open-closed duality
as a generalization of the M\"obius inversion formula
(see Rota \cite{ro}).

\subsection{Plan of the paper}
The rest of this paper is arranged as follows.
In \S \ref{sec2} we review our formalism of abstract QFT and
its realizations.
We apply this formalism to derive some recursion relations
for the refined Euler characteristics of $\Mbar_{g,n}$ in \S \ref{sec3},
and show how these recursions provide us a large number of numerical data.
In \S \ref{sec4} we present our results on $\chi(\Mbar_{g, 0})$;
in particular,
we show that the theory for $\chi_{g,0}$ is equivalent to
the topological 1D gravity.
The results for $\chi(\Mbar_{g,n})$ with $n > 0$
obtained by solving the linear recursion directly
are presented in \S \ref{sec5}.
In \S \ref{sec:KP},
we related the orbifold Euler characteristics of $\Mbar_{g,n}$ to the KP hierarchy.
Finally in \S \ref{sec-duality},
we describe the open-closed duality.

\section{Preliminaries on the Abstract QFT and Its Realizations}
\label{sec2}

In this section,
let us introduce some preliminaries on
the formalism of the abstract quantum field theory for stable graphs
and its realizations developed in \cite{wz}.
The free energies of this `abstract theory' satisfy
a family of quadratic recursion relations,
and we will recall the realization of such recursions
under some particular Feynman rules.

\subsection{Edge-cutting and edge-adding operators on stable graphs}
\label{sec-pre-strat}

In this subsection we introduce some geometric backgrounds of stable graphs,
and recall the constructions of the edge-cutting and edge-adding operators
on stable graphs inspired by these geometric backgrounds
(see \cite[\S 2]{wz}).

Let $\Mbar_{g,n}$ be the Deligne-Mumford moduli space
of stable curves \cite{dm, kn}.
If we do not distinguish the marked points on a stable curve,
then the equivalent classes of such stable curves
of genus $g$ with $n$ marked points will be parametrized by $\Mbar_{g,n}/S_n$.
The stratification of $\overline\cM_{g,n}/S_n$ can be described
in terms of dual graphs of stable curves.

A stable graph $\Gamma$ is a graph consisting of some vertices and edges
satisfying the stability condition.
Each vertex $v$ is marked by a nonnegative integer $g_v \geq 0$
called the `genus' of this vertex.
There are two types of edges in stable graphs:
an internal edge connects two vertices
(the two vertices may be a same one and such an edge is called a loop on this vertex);
and an external edge is attached to a vertex.
A `half-edge' is either an internal edge together with a choice of one endpoint of it,
or an external edge.
The `valence' of a vertex is defined to be the number of half-edges
attached to it.
The stability condition is that,
the valence of each vertex of genus $0$ is at least three,
and the valence of each vertex of genus $1$ is at least one.
Given a graph $\Gamma$,
We will denote by $V(\Gamma)$, $E(\Gamma)$, $E^{ext}(\Gamma)$
the sets of vertices, internal edges, and external edges of $\Gamma$
throughout this paper.

The genus of a connected stable graph $\Gamma$ is defined to be
\ben
g(\Gamma):=h^1(\Gamma)+
\sum_{v \in V(\Gamma)} g_v,
\een
where $h^1 (\Gamma)$ is the number of independent loops in $\Gamma$.
Let $\cG_{g,n}^c$ be the set of all connected stable graphs
of genus $g$ with $n$ external edges,
then the stratification of $\overline\cM_{g,n}/S_n$ is given by:
\ben
\overline\cM_{g,n}/S_n=\bigsqcup_{\Gamma\in\cG_{g,n}^c}\cM_\Gamma,
\een
where $\cM_\Gamma$ is the moduli space of stable curves
whose dual graph is $\Gamma$.

There are two families of natural maps on these moduli spaces of curves--
the gluing maps and forgetful maps.
By considering `inverses' of these maps,
the authors have constructed two families of operators on stable graphs in \cite{wz},
called the edge-cutting $K$ and edge-adding operator $\cD=\pd+\gamma$.

Denote by $\cG_{g,n}$ the set of all stable graphs of genus $g$
with $n$ external edges (not necessarily connected),
and denote
\ben
\cV := \bigoplus_{\substack{\Gamma \in \cG_{g,n}\\ 2g-2+n > 0}} \bQ\Gamma.
\een
The edge-cutting operator $K$ is defined to be
\ben
K:\cV \to \cV,
\qquad
\Gamma \in \cG_{g,n}\mapsto
K(\Gamma)= \sum_{e\in E(\Gamma)} \Gamma_e,
\een
and $\Gamma_e$ is obtained from $\Gamma$ by cutting the internal edge $e$
(and so obtaining two new external edges).
If the stable graph $\Gamma$ has no internal edge,
then we assign $K(\Gamma)=0$.
Such a procedure inverses the `gluing' of stable maps.

The construction of $\cD:\cV\to \cV$ is a little subtle but still natural.
The operator $\pd$ consists two parts,
the first one is to attach an external edge to a vertex and sum over all vertices:
\begin{equation*}
\begin{tikzpicture}
\draw (0,0) circle [radius=0.2];
\node [align=center,align=center] at (0,0) {$g$};
\draw (-0.2,0)--(-0.55,0);
\draw (-0.16,0.1)--(-0.55,0.25);
\draw (-0.16,-0.1)--(-0.55,-0.25);
\node [align=center,align=center] at (-0.7,0.1) {$\vdots$};
\node [align=center,align=center] at (1.8,0) {$\to$};
\draw (0+4,0) circle [radius=0.2];
\node [align=center,align=center] at (0+4,0) {$g$};
\draw (-0.2+4,0)--(-0.55+4,0);
\draw (-0.16+4,0.1)--(-0.55+4,0.25);
\draw (-0.16+4,-0.1)--(-0.55+4,-0.25);
\node [align=center,align=center] at (-0.7+4,0.1) {$\vdots$};
\draw (4.2,0)--(4.6,0);
\node [align=center,align=center] at (5,-0.15) {$,$};
\end{tikzpicture}
\end{equation*}
and the other one is to break up an internal edge and
insert a trivalent vertex of genus $0$
and then sum over all internal edges:
\begin{equation*}
\begin{split}
&\begin{tikzpicture}
\draw (0,0) circle [radius=0.2];
\node [align=center,align=center] at (0,0) {$g_1$};
\draw (1,0) circle [radius=0.2];
\node [align=center,align=center] at (1,0) {$g_2$};
\draw (0.2,0)--(0.8,0);
\draw (-0.2,0)--(-0.55,0);
\draw (-0.16,0.1)--(-0.55,0.25);
\draw (-0.16,-0.1)--(-0.55,-0.25);
\node [align=center,align=center] at (-0.7,0.1) {$\vdots$};
\draw (1.2,0)--(1.55,0);
\draw (1.16,0.1)--(1.55,0.25);
\draw (1.16,-0.1)--(1.55,-0.25);
\node [align=center,align=center] at (1.7,0.1) {$\vdots$};
\node [align=center,align=center] at (3.1,0) {$\to$};
\draw (0+5,0) circle [radius=0.2];
\node [align=center,align=center] at (0+5,0) {$g_1$};
\draw (1+6,0) circle [radius=0.2];
\node [align=center,align=center] at (1+6,0) {$g_2$};
\draw (-0.2+5,0)--(-0.55+5,0);
\draw (-0.16+5,0.1)--(-0.55+5,0.25);
\draw (-0.16+5,-0.1)--(-0.55+5,-0.25);
\node [align=center,align=center] at (-0.7+5,0.1) {$\vdots$};
\draw (1.2+6,0)--(1.55+6,0);
\draw (1.16+6,0.1)--(1.55+6,0.25);
\draw (1.16+6,-0.1)--(1.55+6,-0.25);
\node [align=center,align=center] at (1.7+6,0.1) {$\vdots$};
\draw (6,0) circle [radius=0.2];
\node [align=center,align=center] at (6,0) {$0$};
\draw (5.2,0)--(5.8,0);
\draw (6.2,0)--(6.8,0);
\draw (6,0.2)--(6,0.5);
\node [align=center,align=center] at (8.15,-0.15) {$,$};
\end{tikzpicture}\\
&\begin{tikzpicture}
\draw (0,0) circle [radius=0.2];
\node [align=center,align=center] at (0,0) {$g$};
\draw (-0.2,0)--(-0.55,0);
\draw (-0.16,0.1)--(-0.55,0.25);
\draw (-0.16,-0.1)--(-0.55,-0.25);
\node [align=center,align=center] at (-0.7,0.1) {$\vdots$};
\draw (0.16,0.1) .. controls (0.5,0.2) and (0.5,-0.2) ..  (0.16,-0.1);
\node [align=center,align=center] at (2.45,0) {$\to$};
\draw (0+5,0) circle [radius=0.2];
\node [align=center,align=center] at (0+5,0) {$g$};
\draw (-0.2+5,0)--(-0.55+5,0);
\draw (-0.16+5,0.1)--(-0.55+5,0.25);
\draw (-0.16+5,-0.1)--(-0.55+5,-0.25);
\node [align=center,align=center] at (-0.7+5,0.1) {$\vdots$};
\draw (5.8,0) circle [radius=0.2];
\node [align=center,align=center] at (5.8,0) {$0$};
\draw (5.18,0.07)--(5.62,0.07);
\draw (5.18,-0.07)--(5.62,-0.07);
\draw (6,0)--(6.4,0);
\node [align=center,align=center] at (6.95,-0.15) {$.$};
\end{tikzpicture}
\end{split}
\end{equation*}
And the operator $\gamma$ acts on $\Gamma$ by attaching
a trivalent vertex of genus $0$ to an external edge of $\Gamma$
and summing over all external edges.
If the stable graph $\Gamma$ has no external edge,
then we assign $\gamma(\Gamma)=0$.
See \cite[\S 2.2]{wz} for more details about the definitions of these operators.
The procedures in $\pd$ and $\gamma$ inverse the procedure of
forgetting a marked point on a stable graph,
and here we have taken all unstable-contractions into consideration.
It is not hard to see that $\cD:=\pd+\gamma$ preserves the subspace of connected graphs:
\ben
\cV^c := \bigoplus_{\substack{\Gamma \in \cG^c_{g,n}\\ 2g-2+n > 0}} \bQ\Gamma
\quad \subset \cV.
\een

\begin{Example}
Here we give some examples of these operators:
\begin{equation*}
\begin{split}
&\begin{tikzpicture}[scale=0.95]
\node [align=center,align=center] at (-0.4,0) {$K$};
\draw (1,0) circle [radius=0.2];
\draw (0.4,0) circle [radius=0.2];
\draw (0.6,0)--(0.8,0);
\draw (1.16,0.1) .. controls (1.5,0.2) and (1.5,-0.2) ..  (1.16,-0.1);
\draw (0.24,0.1) .. controls (-0.1,0.2) and (-0.1,-0.2) ..  (0.24,-0.1);
\node [align=center,align=center] at (1,0) {$0$};
\node [align=center,align=center] at (0.4,0) {$0$};
\node [align=center,align=center] at (1.8,0) {$=2$};
\draw (3.3,0) circle [radius=0.2];
\draw (2.7,0) circle [radius=0.2];
\draw (2.9,0)--(3.1,0);
\draw (3.46,0.1)--(3.7,0.15);
\draw (3.46,-0.1)--(3.7,-0.15);
\draw (2.54,0.1) .. controls (2.2,0.2) and (2.2,-0.2) ..  (2.54,-0.1);
\node [align=center,align=center] at (3.3,0) {$0$};
\node [align=center,align=center] at (2.7,0) {$0$};
\node [align=center,align=center] at (4.1,0) {$+\big($};
\draw (5.7+0.2,0) circle [radius=0.2];
\draw (4.7+0.2,0) circle [radius=0.2];
\draw (4.9+0.2,0)--(5.1+0.2,0);
\draw (5.3+0.2,0)--(5.5+0.2,0);
\draw (5.86+0.2,0.1) .. controls (6.2+0.2,0.2) and (6.2+0.2,-0.2) ..  (5.86+0.2,-0.1);
\draw (4.54+0.2,0.1) .. controls (4.2+0.2,0.2) and (4.2+0.2,-0.2) ..  (4.54+0.2,-0.1);
\node [align=center,align=center] at (5.7+0.2,0) {$0$};
\node [align=center,align=center] at (4.7+0.2,0) {$0$};
\node [align=center,align=center] at (6.6,0) {$\big),$};
\end{tikzpicture}\\
&\begin{tikzpicture}[scale=0.95]
\node [align=center,align=center] at (0.4,0) {$K$};
\draw (1,0) circle [radius=0.2];
\draw (1.2,0)--(1.4,0);
\draw (1.16,0.1)--(1.44,0.1);
\draw (1.16,-0.1)--(1.44,-0.1);
\draw (1.6,0) circle [radius=0.2];
\node [align=center,align=center] at (1,0) {$0$};
\node [align=center,align=center] at (1.6,0) {$1$};
\node [align=center,align=center] at (2.2,0) {$=3$};
\draw (3.1,0) circle [radius=0.2];
\draw (3.28,0.07)--(3.52,0.07);
\draw (3.28,-0.07)--(3.52,-0.07);
\draw (2.6,0)--(2.9,0);
\draw (3.9,0)--(4.2,0);
\draw (3.7,0) circle [radius=0.2];
\node [align=center,align=center] at (3.1,0) {$0$};
\node [align=center,align=center] at (3.7,0) {$1$};
\node [align=center,align=center] at (4.45,-0.12) {$,$};
\end{tikzpicture}\\
&\begin{tikzpicture}[scale=0.95]
\node [align=center,align=center] at (-0.4,0) {$\pd$};
\draw (1,0) circle [radius=0.2];
\draw (0.4,0) circle [radius=0.2];
\draw (0.6,0)--(0.8,0);
\draw (1.16,0.1) .. controls (1.5,0.2) and (1.5,-0.2) ..  (1.16,-0.1);
\draw (0.24,0.1) .. controls (-0.1,0.2) and (-0.1,-0.2) ..  (0.24,-0.1);
\node [align=center,align=center] at (1,0) {$0$};
\node [align=center,align=center] at (0.4,0) {$0$};
\node [align=center,align=center] at (1.8,0) {$=2$};
\draw (3.2,0) circle [radius=0.2];
\draw (2.6,0) circle [radius=0.2];
\draw (2.8,0)--(3,0);
\draw (2.6,0.2)--(2.6,0.4);
\draw (3.36,0.1) .. controls (3.7,0.2) and (3.7,-0.2) ..  (3.36,-0.1);
\draw (2.44,0.1) .. controls (2.1,0.2) and (2.1,-0.2) ..  (2.44,-0.1);
\node [align=center,align=center] at (2.6,0) {$0$};
\node [align=center,align=center] at (3.2,0) {$0$};
\node [align=center,align=center] at (4,0) {$+2$};
\draw (5.4,0) circle [radius=0.2];
\draw (4.8,0) circle [radius=0.2];
\draw (5,0)--(5.2,0);
\draw (4.64,0.1) .. controls (4.3,0.2) and (4.3,-0.2) ..  (4.64,-0.1);
\draw (5.58,0.07)--(5.82,0.07);
\draw (5.58,-0.07)--(5.82,-0.07);
\draw (6.2,-0)--(6.5,0);
\draw (6,0) circle [radius=0.2];
\node [align=center,align=center] at (5.4,0) {$0$};
\node [align=center,align=center] at (4.8,0) {$0$};
\node [align=center,align=center] at (6,0) {$0$};
\node [align=center,align=center] at (7,0) {$+$};
\draw (8.4,0) circle [radius=0.2];
\draw (7.8,0) circle [radius=0.2];
\draw (9,0) circle [radius=0.2];
\draw (8,0)--(8.2,0);
\draw (8.6,0)--(8.8,0);
\draw (8.4,0.2)--(8.4,0.4);
\draw (9.16,0.1) .. controls (9.5,0.2) and (9.5,-0.2) ..  (9.16,-0.1);
\draw (7.64,0.1) .. controls (7.3,0.2) and (7.3,-0.2) ..  (7.64,-0.1);
\node [align=center,align=center] at (8.4,0) {$0$};
\node [align=center,align=center] at (7.8,0) {$0$};
\node [align=center,align=center] at (9,0) {$0$};
\node [align=center,align=center] at (9.65,-0.12) {$,$};
\end{tikzpicture}\\
&\begin{tikzpicture}[scale=0.95]
\node [align=center,align=center] at (0.4,0) {$\pd$};
\draw (1,0) circle [radius=0.2];
\draw (1.2,0)--(1.4,0);
\draw (1.16,0.1)--(1.44,0.1);
\draw (1.16,-0.1)--(1.44,-0.1);
\draw (1.6,0) circle [radius=0.2];
\node [align=center,align=center] at (1,0) {$0$};
\node [align=center,align=center] at (1.6,0) {$1$};
\node [align=center,align=center] at (2.1,0) {$=$};
\draw (3,0) circle [radius=0.2];
\draw (3.2,0)--(3.4,0);
\draw (2.5,0)--(2.8,0);
\draw (3.16,0.1)--(3.44,0.1);
\draw (3.16,-0.1)--(3.44,-0.1);
\draw (3.6,0) circle [radius=0.2];
\node [align=center,align=center] at (3,0) {$0$};
\node [align=center,align=center] at (3.6,0) {$1$};
\node [align=center,align=center] at (4.1,0) {$+$};
\draw (4.7,0) circle [radius=0.2];
\draw (4.9,0)--(5.1,0);
\draw (5.5,0)--(5.8,0);
\draw (4.86,0.1)--(5.14,0.1);
\draw (4.86,-0.1)--(5.14,-0.1);
\draw (5.3,0) circle [radius=0.2];
\node [align=center,align=center] at (4.7,0) {$0$};
\node [align=center,align=center] at (5.3,0) {$1$};
\node [align=center,align=center] at (6.2,0) {$+3$};
\draw (6.8,-0.2) circle [radius=0.2];
\draw (6.98,-0.13)--(7.42,-0.13);
\draw (6.98,-0.27)--(7.42,-0.27);
\draw (7.6,-0.2) circle [radius=0.2];
\draw (7.2,0.25) circle [radius=0.2];
\draw (6.94,-0.06)--(7.06,0.11);
\draw (7.46,-0.06)--(7.34,0.11);
\node [align=center,align=center] at (6.8,-0.2) {$0$};
\node [align=center,align=center] at (7.6,-0.2) {$1$};
\node [align=center,align=center] at (7.2,0.25) {$0$};
\draw (7.2,0.45)--(7.2,0.65);
\node [align=center,align=center] at (8.05,-0.15) {$,$};
\end{tikzpicture}\\
&\begin{tikzpicture}[scale=0.95]
\node [align=center,align=center] at (0.1,0) {$\gamma$};
\draw (0.9,0) circle [radius=0.2];
\draw (1.06,0.1)--(1.3,0.15);
\draw (1.06,-0.1)--(1.3,-0.15);
\draw (0.74,0.1) .. controls (0.4,0.2) and (0.4,-0.2) ..  (0.74,-0.1);
\node [align=center,align=center] at (0.9,0) {$0$};
\node [align=center,align=center] at (1.8,0) {$=2$};
\draw (2.7,0) circle [radius=0.2];
\draw (3.3,0) circle [radius=0.2];
\draw (2.9,0)--(3.1,0);
\draw (2.7,0.2)--(2.7,0.4);
\draw (3.47,0.1)--(3.7,0.15);
\draw (3.47,-0.1)--(3.7,-0.15);
\draw (2.54,0.1) .. controls (2.2,0.2) and (2.2,-0.2) ..  (2.54,-0.1);
\node [align=center,align=center] at (2.7,0) {$0$};
\node [align=center,align=center] at (3.3,0) {$0$};
\node [align=center,align=center] at (4,-0.12) {$.$};
\end{tikzpicture}
\end{split}
\end{equation*}

\end{Example}

\subsection{Abstract free energies and their recursion relations}
\label{sec-pre-absrec}

In this subsection
we recall the construction of the abstract QFT,
including of the definition of the abstract free energies
and their recursion relations
(see \cite[\S 2.3-2.4]{wz}).

The following definition is \cite[Definition 2.1]{wz}:
\begin{Definition}
The free energy of genus $g$ for the abstract QFT of stable graphs
is defined to be:
\be
\label{abs-fe}
\widehat{\cF}_g=\sum_{\Gamma\in\cG_{g,0}^c}\frac{1}{|\Aut(\Gamma)|}\Gamma,
\qquad g\geq 2.
\ee
The abstract $n$-point function of genus $g$ is defined to be:
\be
\label{abs-n-pt}
\widehat{\cF}_{g,n}=\sum_{\Gamma\in\cG_{g,n}^c}\frac{1}{|\Aut(\Gamma)|}\Gamma,
\qquad
2g-2+n>0.
\ee
Here $\Aut(\Gamma)$ is the group of automorphisms of $\Gamma$.
\end{Definition}

\begin{Remark}
The stability condition ensures that the right-hand-sides
of the above formulas are all finite summations.
\end{Remark}

\begin{Example}

We have:
\begin{equation*}
\begin{split}
&\begin{tikzpicture}[scale=0.95]
\node [align=center,align=center] at (0.3,0) {$\widehat{\cF}_{0,3}=\frac{1}{6}$};
\draw (1.6,0) circle [radius=0.2];
\draw (1.15,0)--(1.4,0);
\draw (1.76,0.1)--(2,0.15);
\draw (1.76,-0.1)--(2,-0.15);
\node [align=center,align=center] at (1.6,0) {$0$};
\node [align=center,align=center] at (2.35,-0.15) {$,$};
\end{tikzpicture}\\
&\begin{tikzpicture}[scale=0.95]
\node [align=center,align=center] at (0.3-0.2,0) {$\widehat{\cF}_{1,1}=$};
\draw (1,0) circle [radius=0.2];
\draw (1.2,0)--(1.5,0);
\node [align=center,align=center] at (1,0) {$1$};
\node [align=center,align=center] at (2,0) {$+\frac{1}{2}$};
\draw (1+1.8,0) circle [radius=0.2];
\draw (1.2+1.8,0)--(1.5+1.8,0);
\draw (0.84+1.8,0.1) .. controls (0.5+1.8,0.2) and (0.5+1.8,-0.2) ..  (0.84+1.8,-0.1);
\node [align=center,align=center] at (1+1.8,0) {$0$};
\node [align=center,align=center] at (3.65,-0.15) {$,$};
\end{tikzpicture}\\
&\begin{tikzpicture}[scale=0.95]
\node [align=center,align=center] at (0.1+0.4,0) {$\widehat{\cF}_{2}=$};
\draw (1+0.3,0) circle [radius=0.2];
\node [align=center,align=center] at (1+0.3,0) {$2$};
\node [align=center,align=center] at (1.6+0.2,0) {$+\frac{1}{2}$};
\draw (1+1.4+0.2,0) circle [radius=0.2];
\draw (0.84+1.4+0.2,0.1) .. controls (0.5+1.4+0.2,0.2) and (0.5+1.4+0.2,-0.2) ..  (0.84+1.4+0.2,-0.1);
\node [align=center,align=center] at (1+1.4+0.2,0) {$1$};
\node [align=center,align=center] at (3+0.2,0) {$+\frac{1}{2}$};
\draw (1+2.6+0.2,0) circle [radius=0.2];
\draw (1.2+2.6+0.2,0)--(1.4+2.6+0.2,0);
\draw (1.6+2.6+0.2,0) circle [radius=0.2];
\node [align=center,align=center] at (1+2.6+0.2,0) {$1$};
\node [align=center,align=center] at (1.6+2.6+0.2,0) {$1$};
\node [align=center,align=center] at (5,0) {$+\frac{1}{8}$};
\draw (1+4.8,0) circle [radius=0.2];
\draw (0.84+4.8,0.1) .. controls (0.5+4.8,0.2) and (0.5+4.8,-0.2) ..  (0.84+4.8,-0.1);
\draw (1.16+4.8,0.1) .. controls (1.5+4.8,0.2) and (1.5+4.8,-0.2) ..  (1.16+4.8,-0.1);
\node [align=center,align=center] at (1+4.8,0) {$0$};
\node [align=center,align=center] at (6.6,0) {$+\frac{1}{2}$};
\draw (1+6.8,0) circle [radius=0.2];
\draw (0.4+6.8,0) circle [radius=0.2];
\draw (0.6+6.8,0)--(0.8+6.8,0);
\draw (1.16+6.8,0.1) .. controls (1.5+6.8,0.2) and (1.5+6.8,-0.2) ..  (1.16+6.8,-0.1);
\node [align=center,align=center] at (1+6.8,0) {$0$};
\node [align=center,align=center] at (0.4+6.8,0) {$1$};
\node [align=center,align=center] at (8.6,0) {$+\frac{1}{8}$};
\draw (1+9,0) circle [radius=0.2];
\draw (0.4+9,0) circle [radius=0.2];
\draw (0.6+9,0)--(0.8+9,0);
\draw (1.16+9,0.1) .. controls (1.5+9,0.2) and (1.5+9,-0.2) ..  (1.16+9,-0.1);
\draw (0.24+9,0.1) .. controls (-0.1+9,0.2) and (-0.1+9,-0.2) ..  (0.24+9,-0.1);
\node [align=center,align=center] at (1+9,0) {$0$};
\node [align=center,align=center] at (0.4+9,0) {$0$};
\node [align=center,align=center] at (10.6+0.2,0) {$+\frac{1}{12}$};
\draw (1+10.2+0.2,0) circle [radius=0.2];
\draw (1.2+10.2+0.2,0)--(1.4+10.2+0.2,0);
\draw (1.16+10.2+0.2,0.1)--(1.44+10.2+0.2,0.1);
\draw (1.16+10.2+0.2,-0.1)--(1.44+10.2+0.2,-0.1);
\draw (1.6+10.2+0.2,0) circle [radius=0.2];
\node [align=center,align=center] at (1+10.2+0.2,0) {$0$};
\node [align=center,align=center] at (1.6+10.2+0.2,0) {$0$};
\node [align=center,align=center] at (12.45,-0.15) {$.$};
\end{tikzpicture}
\end{split}
\end{equation*}

\end{Example}

Using the edge-cutting and edge-adding operators,
we are able to formulate some recursion relations for $\wcF_{g,n}$.
The followings are \cite[Theorem 2.1; Lemma 2.1]{wz}:

\begin{Theorem}
\label{original-rec-2}
For $2g-2+n>0$, we have
\be\label{eq-thm1}
K\widehat{\cF}_{g,n}=\binom{n+2}{2}\widehat{\cF}_{g-1,n+2}+
\frac{1}{2}\sum_{\substack{g_1+g_2=g,\\n_1+n_2=n+2,\\n_1\geq 1,n_2\geq 1}}
(n_1\widehat{\cF}_{g_1,n_1})\cdot(n_2\widehat{\cF}_{g_2,n_2}),
\ee
where the sum on the right hand side is taken over all stable cases.
\end{Theorem}

\begin{Remark}
Given two graphs $\Gamma_1$ and $\Gamma_2$,
the `product' $\Gamma_1\cdot \Gamma_2$ will be understood as
the (disconnected) graph obtained by taking the disjoint union of
$\Gamma_1$ and $\Gamma_2$ throughout this paper.
\end{Remark}

\begin{Lemma}
\label{lem-original-D}
For $2g-2+n>0$, we have
\be
\cD\widehat{\cF}_{g,n}=(n+1)\widehat{\cF}_{g,n+1}.
\ee
\end{Lemma}

Theorem \ref{original-rec-2} can be rewritten in the following way
using Lemma \ref{lem-original-D}
(see \cite[Theorem 2.2]{wz}):

\begin{Theorem}
\label{thm-original-rec}
For $2g-2+n>0$, we have
\be\label{eq-thm2}
K\widehat{\cF}_{g,n}=\frac{1}{2}(\cD \cD\widehat{\cF}_{g-1,n}
+\sum_{\substack{g_1+g_2=g,\\n_1+n_2=n}}
\cD\widehat{\cF}_{g_1,n_1}\cdot
\cD\widehat{\cF}_{g_2,n_2}).
\ee
In particular, by taking $n=0$ we obtain a quadratic recursion
for $\wcF_g$ ($g\geq 2$):
\be\label{thm-free}
K\widehat{\cF}_g=
\frac{1}{2}(\cD \pd\widehat{\cF}_{g-1}+
\sum_{r=1}^{g-1}\pd\widehat{\cF}_{r} \cdot \pd\widehat{\cF}_{g-r}).
\ee
Here we use the conventions:
\be
\label{pre-convention-1}
\pd\wcF_1=\cD\widehat{\cF}_{1}:=\widehat{\cF}_{1,1},
\qquad
\cD\widehat{\cF}_{0,2}:=3\widehat{\cF}_{0,3},
\qquad
\cD \cD\widehat{\cF}_{0,1}:=6\widehat{\cF}_{0,3},
\ee
by formally applying Lemma \ref{lem-original-D}.

\end{Theorem}

\begin{Remark}
Notice that the graphs appearing in the expression of $\wcF_r$
do not have external edges,
thus $\gamma(\wcF_r)=0$ and then $\cD\wcF_r$ becomes $\pd\wcF_r$
when we derive \eqref{thm-free} from \eqref{eq-thm2}.
\end{Remark}

\begin{Example}
Using the expression of $\wcF_2$ in the previous example, we have:
\begin{equation*}
\begin{split}
&\begin{tikzpicture}[scale=0.95]
\node [align=center,align=center] at (-1.1,0) {$K\wcF_2=\frac{1}{2}$};
\draw (0,0) circle [radius=0.2];
\node [align=center,align=center] at (0,0) {$1$};
\draw (0.16,0.1)--(0.5,0.15);
\draw (0.16,-0.1)--(0.5,-0.15);
\node [align=center,align=center] at (1,0) {$+\frac{1}{2}\bigl($};
\draw (1.7,0) circle [radius=0.2];
\node [align=center,align=center] at (1.7,0) {$1$};
\draw (3,0) circle [radius=0.2];
\node [align=center,align=center] at (3,0) {$1$};
\draw (1.9,0)--(2.2,0);
\draw (2.5,0)--(2.8,0);
\node [align=center,align=center] at (3.8,0) {$\bigr)+\frac{1}{4}$};
\draw (4.9,0) circle [radius=0.2];
\node [align=center,align=center] at (4.9,0) {$0$};
\draw (4.74,0.1) .. controls (4.4,0.2) and (4.4,-0.2) ..  (4.74,-0.1);
\draw (5.06,0.1)--(5.4,0.15);
\draw (5.06,-0.1)--(5.4,-0.15);
\node [align=center,align=center] at (6,0) {$+\frac{1}{2}\bigl($};
\draw (7.1-0.4,0) circle [radius=0.2];
\node [align=center,align=center] at (7.1-0.4,0) {$1$};
\draw (7.3-0.4,0)--(7.6-0.4,0);
\draw (8.4-0.4,0) circle [radius=0.2];
\node [align=center,align=center] at (8.4-0.4,0) {$0$};
\draw (7.9-0.4,0)--(8.2-0.4,0);
\draw (8.56-0.4,0.1) .. controls (8.9-0.4,0.2) and (8.9-0.4,-0.2) ..  (8.56-0.4,-0.1);
\node [align=center,align=center] at (8.6,0) {$\bigr)$};
\end{tikzpicture}\\
&\begin{tikzpicture}[scale=0.95]
\node [align=center,align=center] at (-0.6,0) {$+\frac{1}{2}$};
\draw (0,0) circle [radius=0.2];
\node [align=center,align=center] at (0,0) {$1$};
\draw (0.6,0) circle [radius=0.2];
\node [align=center,align=center] at (0.6,0) {$0$};
\draw (0.2,0)--(0.4,0);
\draw (0.76,0.1)--(1.1,0.15);
\draw (0.76,-0.1)--(1.1,-0.15);
\node [align=center,align=center] at (1.6,0) {$+\frac{1}{4}$};
\draw (2.4,0) circle [radius=0.2];
\node [align=center,align=center] at (2.4,0) {$0$};
\draw (3,0) circle [radius=0.2];
\node [align=center,align=center] at (3,0) {$0$};
\draw (2.24,0.1) .. controls (1.9,0.2) and (1.9,-0.2) ..  (2.24,-0.1);
\draw (2.6,0)--(2.8,0);
\draw (3.16,0.1)--(3.5,0.15);
\draw (3.16,-0.1)--(3.5,-0.15);
\node [align=center,align=center] at (4.2,0) {$+\frac{1}{8}\bigl($};
\draw (5.1,0) circle [radius=0.2];
\node [align=center,align=center] at (5.1,0) {$0$};
\draw (6.4,0) circle [radius=0.2];
\node [align=center,align=center] at (6.4,0) {$0$};
\draw (5.3,0)--(5.6,0);
\draw (5.9,0)--(6.2,0);
\draw (4.94,0.1) .. controls (4.6,0.2) and (4.6,-0.2) ..  (4.94,-0.1);
\draw (6.56,0.1) .. controls (6.9,0.2) and (6.9,-0.2) ..  (6.56,-0.1);
\node [align=center,align=center] at (7.4,0) {$\bigr)+\frac{1}{4}$};
\draw (8.3+0.3,0) circle [radius=0.2];
\node [align=center,align=center] at (8.3+0.3,0) {$0$};
\draw (8.9+0.3,0) circle [radius=0.2];
\node [align=center,align=center] at (8.9+0.3,0) {$0$};
\draw (7.8+0.3,0)--(8.1+0.3,0);
\draw (9.1+0.3,0)--(9.4+0.3,0);
\draw (8.78,0.07)--(9.02,0.07);
\draw (8.78,-0.07)--(9.02,-0.07);
\node [align=center,align=center] at (10,-0.15) {$.$};
\end{tikzpicture}
\end{split}
\end{equation*}
Then using the expression
\begin{equation*}
\begin{tikzpicture}[scale=0.95]
\node [align=center,align=center] at (-0.3,0) {$\widehat{\cF}_{1,2}=\frac{1}{2}$};
\draw (1,0) circle [radius=0.2];
\draw (1.2,0)--(1.5,0);
\draw (0.5,0)--(0.8,0);
\node [align=center,align=center] at (1,0) {$1$};
\node [align=center,align=center] at (2,0) {$+\frac{1}{4}$};
\draw (1+1.8,0) circle [radius=0.2];
\draw (1.17+1.8,0.1)--(1.4+1.8,0.15);
\draw (1.17+1.8,-0.1)--(1.4+1.8,-0.15);
\draw (0.84+1.8,0.1) .. controls (0.5+1.8,0.2) and (0.5+1.8,-0.2) ..  (0.84+1.8,-0.1);
\node [align=center,align=center] at (1+1.8,0) {$0$};
\node [align=center,align=center] at (3.7,0) {$+\frac{1}{2}$};
\draw (1+3.9,0) circle [radius=0.2];
\draw (0.4+3.9,0) circle [radius=0.2];
\draw (1.17+3.9,0.1)--(1.4+3.9,0.15);
\draw (1.17+3.9,-0.1)--(1.4+3.9,-0.15);
\draw (0.6+3.9,0)--(0.8+3.9,0);
\node [align=center,align=center] at (1+3.9,0) {$0$};
\node [align=center,align=center] at (0.4+3.9,0) {$1$};
\node [align=center,align=center] at (5.8,0) {$+\frac{1}{4}$};
\draw (1+6.2,0) circle [radius=0.2];
\draw (0.4+6.2,0) circle [radius=0.2];
\draw (1.17+6.2,0.1)--(1.4+6.2,0.15);
\draw (1.17+6.2,-0.1)--(1.4+6.2,-0.15);
\draw (0.6+6.2,0)--(0.8+6.2,0);
\node [align=center,align=center] at (1+6.3-0.1,0) {$0$};
\node [align=center,align=center] at (0.4+6.3-0.1,0) {$0$};
\draw (0.24+6.3-0.1,0.1) .. controls (-0.1+6.3-0.1,0.2) and (-0.1+6.3-0.1,-0.2) ..  (0.24+6.3-0.1,-0.1);
\node [align=center,align=center] at (8.2-0.1,0) {$+\frac{1}{4}$};
\draw (1+8.1-0.1,0) circle [radius=0.2];
\draw (0.5+8.1-0.1,0)--(0.8+8.1-0.1,0);
\draw (1.18+8.1-0.1,0.07)--(1.42+8.1-0.1,0.07);
\draw (1.18+8.1-0.1,-0.07)--(1.42+8.1-0.1,-0.07);
\draw (1.8+8.1-0.1,-0)--(2.1+8.1-0.1,0);
\draw (1.6+8.1-0.1,0) circle [radius=0.2];
\node [align=center,align=center] at (1+8.1-0.1,0) {$0$};
\node [align=center,align=center] at (1.6+8.1-0.1,0) {$0$};
\node [align=center,align=center] at (10.45,-0.15) {$,$};
\end{tikzpicture}
\end{equation*}
one can easily check:
\ben
K\wcF_2=\wcF_{1,2}+\frac{1}{2}\wcF_{1,1}\cdot\wcF_{1,1}
=\frac{1}{2}\big(\cD\cD\wcF_1+\cD\wcF_1\cdot\cD\wcF_1\big).
\een

\end{Example}

\subsection{Realization of the abstract QFT by Feynman rules}

In this subsection we discuss the realizations of the abstract QFT.
Here we present a generalization of the realization given in \cite[\S 4]{wz}.

Our input data are a sequence of formal functions $\{F_{g,n}(t)\}_{2g-2+n>0}$
in a formal variable $t$,
together with a formal variable $\kappa$ called the propagator.
Using these data,
we will construct a sequence of functions $\{\wF_{g,n}(t,\kappa)\}_{2g-2+n>0}$
by assigning the following Feynman rules to stable graphs.
Let $\Gamma$ be a stable graph,
then we assign a weight $w_\Gamma$ to it:
\be
\label{Feynman-ordinary}
\Gamma \quad\mapsto\quad
w_\Gamma = \prod_{v\in V(\Gamma)} w_v \cdot
\prod_{e\in E(\Gamma)} w_e,
\ee
where the weight of a vertex $v\in V(\Gamma)$
of genus $g_v$ and valence $\val_v$ is given by
\be
\label{Feynman-ordinary-v}
w_v:=F_{g_v,\val_v}(t),
\ee
and the weight of an internal edge $e\in E(\Gamma)$ is given by
\be
\label{Feynman-ordinary-e}
w_e=\kappa.
\ee
Then the function $\wF_{g,n}(t,\kappa)$ is defined to be
the weight of $\wcF_{g,n}$ ($2g-2+n>0$).
Such a procedure will be called a `realization' of the abstract QFT.

More specifically,
the abstract free energy $\wcF_g$ will be realized by:
\be
\label{eq-pre-realization-fe}
\wF_g(t,\kappa):=\sum_{\Gamma\in\cG_{g,0}^c}
\frac{w_\Gamma}{|\Aut(\Gamma)|},
\qquad
g\geq 2,
\ee
and the abstract $n$-point functions $\wcF_{g,n}$ are realized by:
\be
\label{eq-pre-realization-npt}
\wF_{g,n}(t,\kappa):=\sum_{\Gamma\in\cG_{g,n}^c}
\frac{w_\Gamma}{|\Aut(\Gamma)|},
\qquad
2g-2+n>0.
\ee

\begin{Example}
\label{eg-realization}
We give some examples of $\wF_{g,n}(t,\kappa)$ for small $g$ and $n$:
\be
\label{eq1-eg-realization}
\begin{split}
&\wF_{0,3}=\frac{1}{6}F_{0,3},\\
&\wF_{0,4}=\frac{1}{24}F_{0,4}+
\frac{1}{8}\kappa F_{0,3}^2,\\
&\wF_{1,1}=F_{1,1}+\frac{1}{2}\kappa F_{0,3},\\
&\wF_2=F_{2,0}+\kappa\big(\frac{1}{2}F_{1,2}+\frac{1}{2}F_{1,1}^2\big)
+\kappa^2\big(\frac{1}{8}F_{0,4}+\frac{1}{2}F_{1,1}F_{0,3}\big)
+\frac{5}{24}\kappa^3F_{0,3}^2.
\end{split}
\ee

\end{Example}

The above construction can be represented as a formal Gaussian integral.
The following result is well-known in literatures
(see eg. \cite[Theorem 4.1]{wz}):

\begin{Theorem}
\label{new-gauss}
Define a partition function
$\widehat{Z}(t,\kappa):=\exp\bigg(\sum\limits_{g=2}^\infty
\lambda^{2g-2}\widehat{f}_g(t,\kappa)\bigg)$ by
\be
\label{eq-pre-partition}
\widehat{Z}(t,\kappa)=\frac{1}{(2\pi\kappa\lambda^2)^{\frac{1}{2}}}
\int \exp\bigg( \sum_{2g-2+n>0}\frac{\lambda^{2g-2}}{n!}F_{g,n}(t)
\cdot \eta^n -\frac{\lambda^{-2}}{2\kappa} \eta^2 \bigg) d\eta,
\ee
then $\widehat{f}_g (t,\kappa)=\wF_g (t,\kappa)$ for every $g\geq 2$.
\end{Theorem}

\subsection{Realization of the operators and recursion relations}
\label{sec-pre-realization-rec}

We recall the realizations of recursion relations in
Lemma \ref{lem-original-D} and Theorem \ref{thm-original-rec}
in this subsection.
In order to achieve this,
we need to construct the realizations of the operators
$K$, $\pd$, and $\cD=\pd+\gamma$ first.

Clearly the edge-cutting operator $K$ will be realized by
the partial derivative $\pd_\kappa=\frac{\pd}{\pd \kappa}$;
that means,
\be
\label{eq-pre-realizationK}
w_{K(\Gamma)}=\pd_\kappa (w_\Gamma)
\ee
for every stable graph $\Gamma$.
Moreover, the operator $\gamma$ will be realized by multiplying by
$|E^{ext}(\Gamma)|\cdot\kappa F_{0,3}(t)$,
since
\ben
w_{\gamma(\Gamma)}=
|E^{ext}(\Gamma)|\cdot\kappa F_{0,3}(t)
w_\Gamma.
\een
The realization of the operator $\pd$ can be given as follows.

\begin{Definition}\label{def-compatible}
An operator $\hpd$ is called a formal differential operator
compatible with Feynman rules
\eqref{Feynman-ordinary}-\eqref{Feynman-ordinary-e}, if it satisfies
the following conditions:

\begin{itemize}

\item[1)]
Given a stable graph $\Gamma$,
the operator $\hpd$ acts on
$w_\Gamma=\kappa^{|E(\Gamma)|}\cdot\big(\prod\limits_{v\in V(\Gamma)}w_v\big)$
via `Leibniz rules':
\begin{equation*}
\hpd(w_\Gamma)=\kappa^{|E(\Gamma)|}\sum_{v\in V(\Gamma)}\biggl(
\hpd (w_v) \cdot\prod_{v'\not= v}w_{v'}\biggr)
+|E(\Gamma)|\kappa^{|E(\Gamma)|-1}\hpd(\kappa)
\cdot\prod_{v\in V(\Gamma)}w_{v}.
\end{equation*}

\item[2)]
$\hpd$ acts on $w_v$ for a vertex $v$ of
genus $g_v$ and valence $\val_v$ by
\ben
\hpd (F_{g_v,\val_v})=F_{g_v,\val_v +1}.
\een

\item[3)]
$\hpd$ acts on the propagator $\kappa$ by
\ben
\hpd (\kappa)=\kappa^2\cdot F_{0,3}.
\een

\end{itemize}
\end{Definition}

It is clear that for the given Feynman rules
\eqref{Feynman-ordinary}-\eqref{Feynman-ordinary-e},
the operator $\pd$ can be realized by
a compatible formal differential operator $\hpd$,
i.e., $w_{\pd(\Gamma)}=\hpd(w_\Gamma)$.

\begin{Example}
Let $\{F_g(t)\}_{g\geq 0}$ be a sequence of holomorphic functions,
and
\ben
&&F_{g,0}:=F_g;\\
&&F_{g,n}:=F_g^{(n)}(t)=(\frac{\pd}{\pd t})^n F_g(t),\quad n>0,
\een
and we choose the propagator $\kappa$ to be the following function in $t$:
\ben
\kappa=\frac{1}{C-F_0''(t)},
\een
where $C$ is either a constant or an anti-holomorphic function in $t$.
Then clearly the operator $\pd_t=\frac{\pd}{\pd t}$ is compatible with
this Feynman rule, since
\ben
\frac{\pd\kappa}{\pd t}=\kappa^2 F_{0,3}.
\een
This is the original
realization introduced in \cite[\S 4]{wz}.

\end{Example}

\begin{Example}
Let $\{F_g(t)\}$ be a sequence of holomorphic functions for $g\geq 0$,
and
\ben
&&F_{g,0}:=F_g;\\
&&F_{g,n}:=F_g^{(n)}(t)=(\frac{\pd}{\pd t})^n F_g(t),\quad n>0.
\een
And now we choose the propagator $\kappa$ to be
a formal variable independent of $t$.
Then the operator
\ben
\frac{\pd}{\pd t}+\kappa^2 F_{0,3} \cdot \frac{\pd}{\pd \kappa}
\een
is compatible with this Feynman rule, thus it realizes the operator $\pd$.

\end{Example}

Now let us recall the realization of the recursion relations.
Fix a family of input data $\{F_{g,n}\}_{2g-2+n>0}$ and $\kappa$,
and suppose that we have found a suitable realization $\hpd$ of the operator $\pd$.
Denote by $\hat D$ the realization of the operator $\cD$,
then:
\be
\label{eq-pre-realizationD}
\hat D(w_\Gamma)=(\hpd+|E^{ext}(\Gamma)|\kappa F_{0,3})w_\Gamma.
\ee
Now applying such a Feynman rule to
Lemma \ref{lem-original-D} and Theorem \eqref{thm-original-rec},
we obtain:

\begin{Lemma}
\label{lem-realization-D}
For $2g-2+n>0$, we have
\be
\hat D \wF_{g,n}= (n+1) \wF_{g,n+1}.
\ee
Or more explicitly,
\ben
\wF_{g,n}=\frac{1}{n!}
\big(\hpd+(n-1)\kappa F_{0,3}\big)\cdots
\big(\hpd+2\kappa F_{0,3}\big)\big(\hpd+\kappa F_{0,3}\big)\hpd\wF_g(t,\kappa)
\een
for $g\geq 2$; and for $g=0$, $1$,
\begin{equation*}
\begin{split}
&\wF_{1,n}=\frac{1}{n!}
\big(\hpd+(n-1)\kappa F_{0,3}\big)\cdots
\big(\hpd+2\kappa F_{0,3}\big)\big(\hpd+\kappa F_{0,3}\big)
\wF_{1,1}(t,\kappa),\quad n\geq 1;\\
&\wF_{0,n}=
\frac{3!}{n!}
\big(\hpd+(n-1)\kappa F_{0,3}\big)\cdots
\big(\hpd+4\kappa F_{0,3}\big)\big(\hpd+3\kappa F_{0,3}\big)
\wF_{0,3}(t,\kappa),\quad n\geq 3.
\end{split}
\end{equation*}

\end{Lemma}

\begin{Theorem}
\label{thm-realization-rec}
For $2g-2+n>0$, we have
\be
\pd_\kappa\wF_{g,n}=\frac{1}{2}\bigg(\hat D\hat D\wF_{g-1,n}
+\sum_{\substack{g_1+g_2=g,\\n_1+n_2=n}}
\hat D\wF_{g_1,n_1} \cdot \hat D\wF_{g_2,n_2}\bigg).
\ee
In particular, by taking $n=0$ we obtain:
\be
\pd_\kappa\wF_g=\frac{1}{2}\bigg(\hat D \hpd\wF_{g-1}+
\sum_{r=1}^{g-1}\hpd\wF_{r}\cdot \hpd\wF_{g-r}\bigg),
\qquad
g\geq 2.
\ee
Here we use the convention
\be
\hpd\wF_1= \hat D \wF_{1}:=\wF_{1,1},
\qquad
\hat D\wF_{0,2}:=3\wF_{0,3},
\qquad
\hat D\hat D\wF_{0,1}:=6\wF_{0,3},
\ee
by formally applying Lemma \ref{lem-realization-D}.

\end{Theorem}

\section{Refined Orbifold Euler Characteristic of $\overline{\cM}_{g,n}$}
\label{sec3}

In this section,
we introduce a refined orbifold Euler characteristic
$\chi_{g,n}(t,\kappa)$ of the moduli space $\Mbar_{g,n}$.
The orbifold Euler characteristic of $\Mbar_{g,n}$ will be recovered by
$\chi(\Mbar_{g,n})=\chi_{g,n}(1,1)$.
This construction is a particular realization of the abstract QFT for stable graphs,
thus we are able to use the formalism recalled in the previous section
to derive various recursion relations to compute $\chi_{g,n}(t,\kappa)$.
The last two sections of this paper will be devoted to explicit computations
using these recursion relations.

\subsection{A realization of the abstract QFT}

In this subsection,
we first recall some well-known results about
the orbifold Euler characteristic of $\Mbar_{g,n}$ (see \cite{bh, hz, pe}),
and then construct a particular realization of the abstract QFT
by assigning Feynman rules inspired by these works to the stable graphs.
The refined orbifold Euler characteristic $\chi_{g,n}(t,\kappa)$
will be defined to be the realization of $\wcF_{g,n}$.

Let $\cM_{g,n}$ be the moduli space of smooth stable curves
of genus $g$ with $n$ marked points.
The orbifold Euler characteristic of ${\cM}_{g,n}$
has been first computed by Harer and Zagier \cite{hz},
and then also by Penner \cite{pe} and Kontsevich \cite{kon1}.
This problem can be reformulated as a enumeration of fat graphs,
and thus can be studied using methods in Hermitian matrix models.
See also \cite{lz} for an introduction for these results.
The conclusion is as follows:
\begin{Theorem}
[\cite{hz}]
\label{harer-zagier}

The orbifold Euler characteristic of ${\cM}_{g,n}$ is given by:
\be
\label{eq-harer-zagier}
\chi({\cM}_{g,n})=(-1)^{n}\cdot\frac{(2g-1)B_{2g}}{(2g)!}(2g+n-3)!,
\qquad
2g-2+n>0,
\ee
where $B_{2g}$ is the $(2g)$-th Bernoulli number.
\end{Theorem}

Using the stratification of $\overline{\cM}_{g,n}$
(and thus $\Mbar_{g,n}/S_n$),
Bini and Harer \cite{bh} have derived the following formula
for their orbifold Euler characteristics:

\begin{Theorem}
[\cite{bh}]
\label{bini-harer}

Assume $2g-2+n>0$,
then the orbifold Euler characteristic of $\overline{\cM}_{g,n}/S_n$ is
given by the following graph sum formula:
\be
\label{eq-bini-harer}
\chi(\overline{\cM}_{g,n}/S_n)=\sum_{\Gamma\in \cG^c_{g,n}}
\frac{1}{|\Aut(\Gamma)|}\prod_{v\in V(\Gamma)}\chi(\cM_{g_v,\val_v}),
\ee
where $g_v$ is the genus of a vertex $v$,
and $\val_v$ is the valence of $v$.
\end{Theorem}

In what follows,
we will use the two formulas \eqref{eq-harer-zagier} and \eqref{eq-bini-harer}
to formulate a realization of the abstract QFT for stable graphs.

In the formula \eqref{eq-bini-harer} of Bini-Harer,
the weight of a vertex of genus $g$ and valence $n$ is $\chi({\cM}_{g,n})$,
and the propagator is just $1$.
However,
as mentioned in the Introduction,
it is very difficult to carry out concrete calculations directly
using this Feynman rule.

In order to apply our formalism of abstract QFT and its realizations,
we will modify the above Feynman rule by
introducing two formal variables $t$ and $\kappa$.
We take $\kappa$ to be the propagator, and take
\be
\label{eq-realization-vertices}
F_{g,n}^{orb}(t) : = \chi(\cM_{g,n}) t^{-2g+2-n}
=(-1)^{n}\frac{(2g-1)B_{2g}}{(2g)!}(2g+n-3)!\cdot t^{-2g+2-n}
\ee
to be the weight of vertices.

More precisely,
for a stable graph $\Gamma$,
our Feynman rule defines the weight $w_\Gamma$
of a stable graph $\Gamma$ as follows:
\be\label{weight-Gamma}
w_\Gamma = \prod_{v\in V(\Gamma)} w_v \cdot
\prod_{e\in E(\Gamma)} w_e,
\ee
where the weight of an internal edge is
\be
\label{eq-realization-kappa}
w_e:=\kappa,
\ee
and the weight of a vertex $v$ of genus $g_v$ and valence $\val_v$ is:
\be
\label{weight-v}
\begin{split}
\omega_v=F_{g_v,\val_v}^{orb}(t)=&\chi(\cM_{g_v,\val_v})\cdot t^{\chi(C_{g_v,\val_v})}\\
=&(-1)^{\val_v}\frac{(2g_v-1)B_{2g_v}}{(2g_v)!}(2g_v+\val_v-3)!\cdot t^{-2g_v+2-\val_v},
\end{split}
\ee
here $\chi(C_{g,n})$ is the Euler characteristic of a curve $C_{g,n}$
of genus $g$ with $n$ punctures.

It is clear that if we take:
\be
\begin{split}
&F_g^{orb}(t):=\frac{(2g-1)B_{2g}}{(2g)!}(2g-3)!\cdot t^{-2g+2},\quad g\geq 2;\\
&F_{1,1}^{orb}(t):=-\frac{B_2}{2}\cdot t^{-1};\\
&F_{0,3}^{orb}(t):=t^{-1},
\end{split}
\ee
then we have:
\be
\begin{split}
&F_{g,n}^{orb}(t)=(\frac{d}{dt})^n F_{g}^{orb}(t),\quad g\geq 2;\\
&F_{1,n}^{orb}(t)=(\frac{d}{dt})^{n-1} F_{1,1}^{orb}(t),\quad n\geq 1;\\
&F_{0,n}^{orb}(t)=(\frac{d}{dt})^{n-3} F_{0,3}^{orb}(t),\quad n\geq 3.
\end{split}
\ee

\begin{Definition}
Assume $2g-2+n>0$.
We define the refined orbifold Euler characteristic $\chi_{g,n}(t,\kappa)$
of $\overline\cM_{g,n}/S_n$ to be the realization of $\wcF_{g,n}$
under the Feynman rule \eqref{weight-Gamma},
i.e.,
\be\label{chi-g,n}
\chi_{g,n}(t,\kappa):=\sum_{\Gamma\in \cG^c_{g,n}}
\frac{\kappa^{|E(\Gamma)|}}{|\Aut(\Gamma)|}\cdot
\prod_{v\in V(\Gamma)}F_{g_v,\val_v}^{orb}(t),
\ee
where $g_v$ and $\val_v$ are the genus and valence of the vertex $v$ respectively.
\end{Definition}

Clearly $\chi_{g,n}(t,\kappa)$ is a polynomial in $\kappa$ and $t^{-1}$,
and the degree of $\kappa$ encodes the codimension
of the strata $\cM_\Gamma$ in $\overline\cM_{g,n}$.

For every $(g,n)$ with $2g-2+n>0$,
the orbifold Euler characteristic of $\overline{\cM}_{g,n}/S_n$
is obtained from $\chi_{g,n}(t,\kappa)$ by simply taking $t=1$ and $\kappa=1$:
\be
\chi(\overline{\cM}_{g,n}/S_n)=\chi_{g,n}(1,1),
\ee
thus we have
\be
\chi(\overline{\cM}_{g,n})=n!\cdot\chi_{g,n}(1,1).
\ee

The above construction can be represented in terms of a formal Gaussian integral
by Theorem \ref{new-gauss} as follows.

\begin{Theorem} \label{thm:chi(g,0)}
Define a partition function  $\widehat{Z}^{orb}(t,\kappa)$
by the following formal Gaussian integral:
\be \label{gaussian-chi(t,k)}
\widehat{Z}^{orb}(t,\kappa)=\frac{1}{(2\pi\lambda^2\kappa)^{\frac{1}{2}}}
\int \exp\bigg(\sum_{2g-2+n>0}\frac{\lambda^{2g-2}}{n!}F_{g,n}^{orb}(t)
\eta^n-\frac{\lambda^{-2}}{2\kappa}\eta^2\bigg)d\eta,
\ee
then its logarithm (i.e., the free energy) equals to
$\sum\limits_{g=2}^\infty \chi_{g,0}(t,\kappa)
\lambda^{2g-2}$.
In particular, we can take $t=1$, then this integral becomes:
\be\label{gaussian-chi}
\widehat{Z}^{orb}(1,\kappa)=\frac{1}{(2\pi\lambda^2\kappa)^{\frac{1}{2}}}
\int \exp\bigg(\sum_{2g-2+n>0}\frac{\lambda^{2g-2}}{n!}\chi(\cM_{g,n})
\cdot\eta^n-\frac{\lambda^{-2}}{2\kappa}\eta^2\bigg)d\eta,
\ee
and its free energy is $\sum\limits_{g\geq 2} \chi_{g,0}(1,\kappa) \lambda^{2g-2}$.

\end{Theorem}

\begin{Remark}
The sum $\sum\limits_{2g-2+n>0}\frac{\lambda^{2g-2}}{n!}\chi(\cM_{g,n})\cdot\eta^n$
in the formal Gaussian integral \eqref{gaussian-chi}
can be rewritten in the following form:
\begin{equation*}
\begin{split}
&\sum_{2g-2+n>0}\frac{\lambda^{2g-2}}{n!}\chi(\cM_{g,n})\cdot\eta^n\\
 =&  \sum_{2g-2+n>0} (-1)^{n}\frac{(2g-1)B_{2g}}{(2g)!}(2g+n-3)!\cdot \lambda^{2g-2}
\cdot \frac{\eta^n}{n!} \\
 = &  \sum_{n = 3}^\infty \lambda^{-2}
\frac{(-1)^{n+1}\eta^n}{n(n-1)(n-2)}
+ \sum_{n = 1}^{\infty}
\frac{B_{2}}{2!}\cdot \frac{(-1)^n\eta^n}{n}\\
& +  \sum_{g=2}^{\infty}  \frac{(2g-1)B_{2g}}{(2g)!} \cdot \lambda^{2g-2}
\sum_{n=0}^\infty\frac{(2g+n-3)!\cdot (-1)^n\eta^n}{n!}\\
 = & \biggl(\half (1+\eta)^2 \log(1+\eta) - \frac{\eta}{2} - \frac{3\eta^2}{4} \biggr) \lambda^{-2}
- \frac{1}{12} \log (1+\eta) \\
& +  \sum_{g=2}^{\infty}  \frac{B_{2g}}{2g(2g-2)} \cdot \lambda^{2g-2} (1+\eta)^{2-2g}.
\end{split}
\end{equation*}
Similarly, we have
\begin{equation*}
\begin{split}
&\sum_{2g-2+n>0}\frac{\lambda^{2g-2}}{n!}F_{g,n}^{orb}(t)
\cdot \eta^n\\
=&\lambda^{-2}\bigg(
-\frac{3}{4}\eta^2-\half \eta t
+\half (\eta+t)^2 \log(1+\frac{\eta}{t})\bigg)
-\frac{1}{12}\log(1+\frac{\eta}{t})\\
&+\sum_{g=2}^\infty
\lambda^{2g-2}\cdot
\frac{B_{2g}}{2g(2g-2)} (\eta+t)^{2-2g}.
\end{split}
\end{equation*}
See \cite{dv} for a similar computation.
Here we are taking summation over $n$ first.
Later in \S \ref{sec:Gamma} we will see that it is also important
to take summation over $g$ first.
\end{Remark}

By evaluating the formal Gaussian integrals in \eqref{gaussian-chi(t,k)} and \eqref{gaussian-chi} above
directly from their definitions,
we get the following two formulas for generating series for
$\{\chi_{g,0}(t,\kappa)\}_{g\geq 2}$ and
$\{\chi_{g,0}(1,\kappa)\}_{g\geq 2}$ respectively:
\be\label{gene-chi(t,k)}
\begin{split}
\sum_{g\geq 2}\lambda^{2g-2}\chi_{g,0}(t,\kappa)=&\log\biggl(
\sum_{k\geq 0}\sum_{2g_i-2+l_i>0}\sum_{l}\frac{1}{k!}\cdot
\frac{\lambda^{2(g_1+\cdots+g_k)-2k+2l}}{l_1!\cdots l_k!}\\
&\quad\times(2l-1)!!\cdot \prod_{i=1}^{k}F_{g_i,l_i}^{orb}(t)
\cdot\kappa^l\cdot\delta_{l_1+\cdots+l_k,2l}
\biggr),
\end{split}
\ee
\be\label{gene-chi}
\begin{split}
\sum_{g\geq 2}\lambda^{2g-2}\chi_{g,0}(1,\kappa)=&\log\biggl(
\sum_{k\geq 0}\sum_{2g_i-2+l_i>0}\sum_{l}\frac{1}{k!}\cdot
\frac{\lambda^{2(g_1+\cdots+g_k)-2k+2l}}{l_1!\cdots l_k!}\\
&\quad\times(2l-1)!!\cdot \prod_{i=1}^{k}\chi(\cM_{g_i,l_i})
\cdot\kappa^l\cdot\delta_{l_1+\cdots+l_k,2l}
\biggr).
\end{split}
\ee
The complexity of the right-hand side of these formulas make it unpractical
to use them for computations or deriving simpler expressions.

In the remaining of this section,
we will apply our formalism to derive a
quadratic recursion for $\chi_{g,n}(t,\kappa)$
(and also for $\chi_{g,n}(1,\kappa)$);
and furthermore,
we will present a linear recursion for $\chi_{g,n}(t,\kappa)$
with fixed genus $g$.

\subsection{Quadratic recursion relation}
\label{sec-quadr-rec}

In this subsection we derive the quadratic recursion relations
for $\chi_{g,n}(t,\kappa)$.

Consider the realizations of Theorem \ref{thm-original-rec}.
As pointed out in \S \ref{sec-pre-realization-rec},
the edge-cutting operator in this case is just the partial derivative
$\pd_\kappa=\frac{\pd}{\pd \kappa}$.
The operator $\pd$ is realized by a formal differential operator
which takes $F_{g,n}^{orb}$ to $F_{g,n+1}^{orb}$,
and takes $\kappa$ to $\kappa^2\cdot t^{-1}$
(see Definition \ref{def-compatible}).
Here since we are regarding $t$ and $\kappa$ as
two independent formal variables,
we can simply take the realization of $\pd$ to be:
\be
\label{eq-realization-chi-pd}
d:=\frac{\pd}{\pd t}+\kappa^2 t^{-1}\cdot\frac{\pd}{\pd \kappa}
\ee
as a realization of the operator $\pd$.
Then the operator $\cD=\pd+\gamma$ is realized by
\be
\label{eq-realization-D}
D=\frac{\pd}{\pd t}+\kappa^2 t^{-1}\cdot\frac{\pd}{\pd \kappa}
+|E^{ext}(\Gamma)|\cdot \kappa t^{-1},
\ee
where $|E^{ext}(\Gamma)|$ is the number of external edges of $\Gamma$.

Now the quadratic recursion in Theorem \ref{thm-original-rec}
is realized as follows:

\begin{Theorem}
Assume $2g-2+n>0$, then we have:
\be\label{rec-euler-chi}
\frac{\pd}{\pd\kappa}\chi_{g,n}=\frac{1}{2}(DD\chi_{g-1,n}
+\sum_{\substack{g_1+g_2=g,\\n_1+n_2=n}}
D\chi_{g_1,n_1}\cdot
D\chi_{g_2,n_2}).
\ee
In particular, by taking $n=0$ we get:
\be
\frac{\pd}{\pd\kappa}\chi_{g,0}=\frac{1}{2}(D d\chi_{g-1,0}+
\sum_{r=1}^{g-1}d\chi_{r,0}\cdot
d\chi_{g-r,0}),
\qquad g\geq 2.
\ee
Here we use the convention:
\be
\label{convention-1}
d\chi_{1,0}=D\chi_{1,0}:=\chi_{1,1},
\qquad
D\chi_{0,2}:=3\chi_{0,3},
\qquad
DD\chi_{0,1}:=6\chi_{0,3}.
\ee

\end{Theorem}

Integrating the recursion formula \eqref{rec-euler-chi} with respect to $\kappa$,
we can determine $\chi_{g,n}$ up to a term independent of $\kappa$.
From the definition \eqref{chi-g,n}
we may easily see that this term is just $\frac{1}{n!}F_{g,n}^{orb}(t)$
where $F_{g,n}^{orb}(t)$ is given by \eqref{weight-v}.
Thus we get an algorithm that finds $\{\chi_{g,n}(t,\kappa)\}$ recursively:

\begin{Theorem}\label{thm-rec-chi-g,n}
For $2g-2+n>0$, we have
\be
\chi_{g,n}=\frac{1}{n!}F_{g,n}^{orb}(t)+\frac{1}{2}\int_{0}^{\kappa}\bigg(
DD\chi_{g-1,n}+
\sum_{\substack{g_1+g_2=g,\\n_1+n_2=n}}
D\chi_{g_1,n_1}\cdot
D\chi_{g_2,n_2}\bigg)d\kappa.
\ee

\end{Theorem}

For our original problem of the
computations of orbifold Euler characteristic
$\chi(\overline{\cM}_{g,n}/S_n)=\chi_{g,n}(1,1)$,
the above recursion can be simplified.
In fact, we can apply Theorem \ref{original-rec-2} directly,
and in this way we do not need the variable $t$
or the operator $D$
in our quadratic recursion relation.
In other words, we may define
\be
\label{eq-def-tildechi}
\widetilde{\chi}_{g,n}(\kappa):=\chi_{g,n}(1,\kappa)=
\sum_{\Gamma\in \cG^c_{g,n}}
\frac{\kappa^{|E(\Gamma)|}}{|\Aut(\Gamma)|}\cdot
\prod_{v\in V(\Gamma)}\chi(\cM_{g_v,\val_v}),
\ee
then the following result is obtained by directly applying the Feynman rule
\begin{equation*}
\begin{split}
&v\mapsto \chi(\cM_{g_v,\val_v}),\qquad v\in V(\Gamma),\\
&e \mapsto \kappa , \qquad\qquad\qquad e\in E(\Gamma),
\end{split}
\end{equation*}
to the relation \eqref{eq-thm1}:

\begin{Theorem}
For $2g-2+n>0$, we have
\be\label{rec-integral}
\begin{split}
\widetilde\chi_{g,n}(\kappa)=&\frac{1}{2}\int_{0}^{\kappa}\bigg(
(n+2)(n+1)\widetilde\chi_{g-1,n+2}+
\sum_{\substack{g_1+g_2=g,\\n_1+n_2=n+2\\n_1\geq 1,n_2 \geq 1}}
n_1n_2\widetilde\chi_{g_1,n_1}\widetilde\chi_{g_2,n_2}\bigg)d\kappa\\
&+\frac{1}{n!}\chi(\cM_{g,n}),
\end{split}
\ee
where the sum is taken over all stable cases.
\end{Theorem}

The refined orbifold Euler characteristic $\chi_{g,n}(t,\kappa)$ can be
recovered from $\widetilde\chi_{g,n}(\kappa)=\chi_{g,n}(1,\kappa)$
using the following relation
(thus sometimes we will also call $\widetilde\chi_{g,n}(\kappa)$
the refined orbifold Euler characteristic):

\begin{Proposition}
\label{prop-chi-homog}
Assume $2g-2+n>0$.
We have:
\be
\chi_{g,n}(t,\kappa)=t^{2-2g-n}\cdot
\widetilde\chi_{g,n}(\kappa).
\ee
\end{Proposition}

\begin{proof}
Let $\Gamma$ be a connected graph of genus $g$ with $n$ external edges.
Its weight $w_\Gamma$ (under the Feynman rule \eqref{weight-Gamma})
is a monomial in $t^{-1}$ and $\kappa$,
and it suffices to show that the degree of $t^{-1}$ is $2g-2+n$.
In fact,
the Euler's formula tells us:
\ben
1-h^1=|V(\Gamma)|-|E(\Gamma)|,
\een
where $h^1$ is the number of independent loops in $\Gamma$.
Thus the degree of $t^{-1}$ is:
\begin{equation*}
\begin{split}
\sum_{v\in V(\Gamma)} (2g_v-2+\val_v)=&
2\sum_{v\in V(\Gamma)} g_v +\sum_{v\in V(\Gamma)} \val_v -2|V(\Gamma)|\\
=&2\sum_{v\in V(\Gamma)} g_v
+\bigg(2|E(\Gamma)|+n\bigg) -2|V(\Gamma)|\\
=&2\sum_{v\in V(\Gamma)} g_v +n +2h^1 -2\\
=&2g-2+n.
\end{split}
\end{equation*}
This proves the conclusion.
\end{proof}

\begin{Remark}
In \cite{dn},
Do and Norbury give the following quadratic recursion (\cite[Prop 6.1]{dn})
for $\chi(\Mbar_{g,n})$:
\ben
\chi(\overline\cM_{g,n+1})&=&(2-2g-n)\chi(\overline\cM_{g,n})
+\half\chi(\overline\cM_{g-1,n+2})\\
&&+\half\sum_{h=0}^{g}\sum_{k=0}^{n}\binom{n}{k}
\chi(\overline\cM_{h,k+1})\chi(\overline\cM_{g-h,n-k+1}).
\een
To apply such recursion relations to compute $\chi(\Mbar_{g,n})$,
one needs the formulas for $\chi(\Mbar_{g,0})$ for $g \geq 2$ as initial values.
Such formulas do not seem to be accessible by their method.
For $g=0$ and $g=1$,
 the fact that $\chi(\overline\cM_{0,3})=1$ and
$\chi(\overline\cM_{1,1})=\frac{5}{12}$
can be used to recursively compute
$\chi(\overline\cM_{0,n}) (n\geq 4)$ and $\chi(\overline\cM_{1,n}) (n\geq 2)$,
but the problems of finding their explicit formulas are not addressed by these authors.
\end{Remark}

\begin{Example}
Consider the special case $g=0$.
This has already been studied by Keel \cite{ke} and Manin \cite{ma}.
Now let us study this case using our recursion.

Taking $g=0$ in \eqref{rec-integral},
we get a recursion formula for $\{\widetilde{\chi}_{0,n}\}_{n\geq 3}$:
\be\label{genus0-qrec}
\widetilde\chi_{0,n}(\kappa)=\frac{1}{2}\int_{0}^{\kappa}\bigg(
\sum_{i=3}^{n-1}
i(n+2-i)\widetilde\chi_{0,i}\widetilde\chi_{0,n+2-i}\bigg)d\kappa
+\frac{1}{n!}\chi(\cM_{0,n})
\ee
for every $n\geq 3$.
The Harer-Zagier formula \eqref{eq-harer-zagier}
provides the initial data:
\ben
&&\chi(\cM_{0,n})=(-1)^{n+1}\cdot(n-3)!, \quad n\geq 3;\\
&&\widetilde\chi_{0,3}=
\frac{1}{6}\chi(\cM_{0,3})=\frac{1}{6}.
\een
Then we easily obtain the following data:
\ben
&&\widetilde\chi_{0,3}=\frac{1}{6};\\
&&\widetilde\chi_{0,4}=-\frac{1}{24} + \frac{1}{8}\kappa;\\
&&\widetilde\chi_{0,5}=\frac{1}{60} - \frac{1}{12}\kappa + \frac{1}{8}\kappa^2;\\
&&\widetilde\chi_{0,6}=-\frac{1}{120} + \frac{1}{18}\kappa
 - \frac{7}{48}\kappa^2 + \frac{7}{48}\kappa^3;\\
&&\widetilde\chi_{0,7}=\frac{1}{210} - \frac{7}{180}\kappa + \frac{5} {36}\kappa^2
 - \frac{1}{4}\kappa^3 + \frac{3}{16}\kappa^4;\\
&&\widetilde\chi_{0,8}=-\frac{1}{336} + \frac{41}{1440}\kappa - \frac{181}{ 1440}\kappa^2
 + \frac{5}{ 16}\kappa^3 - \frac{55} {128}\kappa^4 + \frac{33}{128}\kappa^5;\\
&&\widetilde\chi_{0,9}=\frac{1}{504} - \frac{109 }{5040}\kappa
+ \frac{97 }{864}\kappa^2 - \frac{451 }{1296}\kappa^3 +
 \frac{385 }{576}\kappa^4 - \frac{143 }{192}\kappa^5 + \frac{143 }{384}\kappa^6;\\
&&\widetilde\chi_{0,10}=-\frac{1}{720} + \frac{853 }{50400}\kappa
- \frac{6061 }{60480}\kappa^2 + \frac{1903 }{5184}\kappa^3 - \frac{
 15301}{17280}\kappa^4 + \frac{1001}{720}\kappa^5 \\
&&\qquad\quad- \frac{1001 }{768}\kappa^6
 + \frac{143 }{256}\kappa^7;\\
&&\cdots\cdots
\een
By taking $\kappa=1$ and multiplying by $n!$, we get:
\begin{equation*}
\begin{split}
&\chi(\overline{\cM}_{0,3})=1,\quad
\chi(\overline{\cM}_{0,4})=2,\quad
\chi(\overline{\cM}_{0,5})=7,\quad
\chi(\overline{\cM}_{0,6})=34,\quad
\chi(\overline{\cM}_{0,7})=213,\\
&\chi(\overline{\cM}_{0,8})=1630,\quad
\chi(\overline{\cM}_{0,9})=14747,\quad
\chi(\overline{\cM}_{0,10})=153946,\quad
\cdots\cdots
\end{split}
\end{equation*}
These numbers coincide with the results obtained in \cite{bh, dn}.

\end{Example}

\begin{Example}
Now let us move on to the case $g=1$.
Taking $g=1$ in \eqref{rec-integral}, we obtain:

\be\label{genus1-qrec}
\begin{split}
\widetilde\chi_{1,n}(\kappa)=&\frac{1}{2}\int_{0}^{\kappa}\bigg(
(n+2)(n+1)\widetilde\chi_{0,n+2}+
2\sum_{i=3}^{n+1}
i(n+2-i)\widetilde\chi_{0,i}\widetilde\chi_{1,n+2-i}\bigg)d\kappa\\
&+\frac{1}{n!}\chi(\cM_{1,n})
\end{split}
\ee
for every $n\geq 1$.
Here by \eqref{eq-harer-zagier} the initial values are:
\ben
\chi(\cM_{1,n})=(-1)^n\cdot\frac{(n-1)!}{12}.
\een
Then explicit computations give us:
\begin{equation*}
\begin{split}
\widetilde\chi_{1,1}=&-\frac{1}{12} +\frac{ 1}{2}\kappa;\\
\widetilde\chi_{1,2}=&\frac{1}{24} - \frac{7 }{24}\kappa
+\frac{ 1}{2}\kappa^2;\\
\widetilde\chi_{1,3}=&-\frac{1}{36} + \frac{2 }{9}\kappa
- \frac{5 }{8}\kappa^2 + \frac{2 }{3}\kappa^3;\\
\widetilde\chi_{1,4}=&\frac{1}{48} - \frac{3 }{16}\kappa
+ \frac{199 }{288}\kappa^2 - \frac{41 }{32}\kappa^4 + \kappa^4;\\
\widetilde\chi_{1,5}=&-\frac{1}{60} + \frac{1}{6}\kappa
- \frac{533 }{720}\kappa^2 + \frac{89 }{48}\kappa^3
- \frac{83 }{32}\kappa^4 + \frac{8 }{5}\kappa^5;\\
\widetilde\chi_{1,6}=&\frac{1}{72} -\frac{11 }{72}\kappa
+ \frac{677 }{864}\kappa^2 -\frac{5203 }{2160}\kappa^3
+ \frac{ 2669 }{576}\kappa^4 - \frac{1003 }{192}\kappa^5 + \frac{8 }{3}\kappa^6;\\
\widetilde\chi_{1,7}=&-\frac{1}{84} + \frac{1}{7}\kappa
- \frac{277 }{336}\kappa^2 + \frac{3197 }{1080}\kappa^3
-\frac{1131 }{160}\kappa^4 + \frac{
 799 }{72}\kappa^5 - \frac{2015 }{192}\kappa^6 + \frac{32 }{7}\kappa^7;\\
\widetilde\chi_{1,8}=&\frac{1}{96} - \frac{13 }{96}\kappa
+ \frac{2323 }{2688}\kappa^2 - \frac{425491 }{120960}\kappa^3
+ \frac{ 341639 }{34560}\kappa^4 - \frac{223829 }{11520}\kappa^5
 + \frac{39673 }{1536}\kappa^6\\
& - \frac{
 32339 }{1536}\kappa^7 +\kappa^8;\\
 \cdots&\cdots
\end{split}
\end{equation*}
By taking $\kappa=1$ and multiplying by $n!$, we get:
\begin{equation*}
\begin{split}
&\chi(\overline{\cM}_{1,1})=\frac{5}{12},\qquad
\chi(\overline{\cM}_{1,2})=\frac{1}{2},\qquad
\chi(\overline{\cM}_{1,3})=\frac{17}{12},\qquad
\chi(\overline{\cM}_{1,4})=\frac{35}{6},\\
&\chi(\overline{\cM}_{1,5})=\frac{389}{12},\qquad
\chi(\overline{\cM}_{1,6})=\frac{1349}{6},\qquad
\chi(\overline{\cM}_{1,7})=\frac{22489}{12},\\
&\chi(\overline{\cM}_{1,8})=\frac{36459}{2},\qquad
\cdots\cdots
\end{split}
\end{equation*}

\end{Example}

\begin{Example}
Next consider the case $g=2$. The recursion is as follows:
\be
\begin{split}
\widetilde\chi_{2,n}(\kappa)=&\frac{1}{2}\int_{0}^{\kappa}\bigg(
(n+2)(n+1)\widetilde\chi_{1,n+2}+
2\sum_{i=3}^{n+2}
i(n+2-i)\widetilde\chi_{0,i}\widetilde\chi_{2,n+2-i}\\
&+\sum_{j=1}^{n+1}\chi_{1,j}\chi_{1,n+2-j}\bigg) d\kappa
+\frac{1}{n!}\chi(\cM_{2,n})
\end{split}
\ee
for every $n\geq 0$.
Here by \eqref{eq-harer-zagier} we have:
\ben
\chi(\cM_{2,n})=(-1)^{n+1}\cdot\frac{(n+1)!}{240}.
\een
Then explicit computations give us
\begin{equation*}
\begin{split}
\widetilde\chi_{2,0}=&-\frac{1}{240} + \frac{13 }{288}\kappa
-\frac{ 1}{6}\kappa^2 + \frac{5 }{24}\kappa^3;\\
\widetilde\chi_{2,1}=&\frac{1}{120} - \frac{13 }{144}\kappa
+ \frac{109 }{288}\kappa^2 - \frac{3 }{4}\kappa^3 + \frac{5 }{8}\kappa^4;\\
\widetilde\chi_{2,2}=&-\frac{1}{80} + \frac{67 }{480}\kappa
- \frac{379 }{576}\kappa^2 + \frac{325 }{192}\kappa^3
- \frac{39 }{16}\kappa^4 + \frac{ 25 }{16}\kappa^5;\\
\widetilde\chi_{2,3}=&\frac{1}{60} - \frac{7 }{36}\kappa
+ \frac{4393 }{4320}\kappa^2 - \frac{677 }{216}\kappa^3
+ \frac{3497 }{576}\kappa^4 - \frac{ 167 }{24}\kappa^5 + \frac{175 }{48}\kappa^6;\\
\widetilde\chi_{2,4}=&-\frac{1}{48} + \frac{23 }{90}\kappa
- \frac{5065 }{3456}\kappa^2 + \frac{17933 }{3456}\kappa^3 - \frac{
 9439 }{768}\kappa^4 + \frac{44519 }{2304}\kappa^5
 - \frac{3547 }{192}\kappa^6 + \frac{525 }{64}\kappa^7;\\
\widetilde\chi_{2,5}=&\frac{1}{40} - \frac{97 }{300}\kappa
+ \frac{5801 }{2880}\kappa^2 - \frac{3833 }{480}\kappa^3 + \frac{
 11887 }{540}\kappa^4 - \frac{5485 }{128}\kappa^5 + \frac{14579 }{256}\kappa^6\\
  & - \frac{1123 }{24}\kappa^7 + \frac{
 1155 }{64}\kappa^8;\\
\widetilde\chi_{2,6}=&-\frac{7}{240} + \frac{1433 }{3600}\kappa
- \frac{230971 }{86400}\kappa^2 + \frac{201593 }{17280}\kappa^3 - \frac{
 58853 }{1620}\kappa^4 + \frac{956309 }{11520}\kappa^5\\
 & - \frac{211753 }{1536}\kappa^6 + \frac{
 732659 }{4608}\kappa^7 - \frac{44021 }{384}\kappa^8 + \frac{5005 }{128}\kappa^9;\\
 \cdots&\cdots
\end{split}
\end{equation*}
By taking $\kappa=1$ and multiplying by $n!$, we get:
\begin{equation*}
\begin{split}
&\chi(\overline{\cM}_{2,0})=\frac{119}{1440},\quad
\chi(\overline{\cM}_{2,1})=\frac{247}{1440},\quad
\chi(\overline{\cM}_{2,2})=\frac{413}{720},\quad
\chi(\overline{\cM}_{2,3})=\frac{89}{32},\\
&\chi(\overline{\cM}_{2,4})=\frac{12431}{720},\quad
\chi(\overline{\cM}_{2,5})=\frac{189443}{144},\quad
\chi(\overline{\cM}_{2,6})=\frac{853541}{720},\quad
\cdots\cdots
\end{split}
\end{equation*}

\end{Example}

In this subsection we have found an algorithm
that effectively computes $\chi(\Mbar_{g,n})$
and provided some examples and numerical data.
We will show that our method can also be used to derived some
closed formula  in \S \ref{sec4} and \S \ref{sec5}.

\subsection{Linear recursion relation}
\label{sec:Linear}

In this subsection we show that there is a simple linear recursion
for $\widetilde\chi_{g,n}(\kappa)$ at each fixed genus $g$.

It is clear that when $2g-2+n>0$,
$\widetilde\chi_{g,n}(\kappa)$ is a polynomial in $\kappa$.
We have:
\begin{Lemma}
$\widetilde\chi_{g,n}(\kappa)$ is a polynomial in $\kappa$
of degree $3g-3+n$.
\end{Lemma}
\begin{proof}
It is not hard to see that the maximum of $|E(\Gamma)|$ for
a connected $\Gamma$ of genus $g$ with $n$ external edges
is obtained when the vertices of $\Gamma$ are all trivalent and of genus $0$.
Assume $\Gamma$ is such a graph,
then
\ben
1-g=|V(\Gamma)|-|E(\Gamma)|=
\frac{1}{3}\big( n+2|E(\Gamma)| \big) -|E(\Gamma)|,
\een
and this completes the proof.
\end{proof}

Throughout this paper,
we will denote by $\{a_{g,n}^i\}$ the coefficients of this polynomial:
\be
\label{eq-def-agni}
\widetilde\chi_{g,n}(\kappa)=
\sum_{i=0}^{3g-3+n}a_{g,n}^i\kappa^i.
\ee
Then it is clear that:
\ben
a_{g,n}^i=
\sum_{\Gamma}\frac{1}{|\Aut(\Gamma)|}
\prod_{v\in V(\Gamma)}\chi(\cM_{g_v, \val_v}),
\een
where the summation is over all connected stable graphs of genus $g$,
with $i$ internal edges and $n$ external edges.

For a fixed $g\geq 0$,
we now derive
a linear recursion relation that computes $\chi_{g,n}(t,\kappa)$
as well as the coefficients $\{a_{g,n}^i\}$
from the initial data $\{a_{g,0}^i\}$ (or $\{a_{0,3}^i\}$, $\{a_{1,1}^i\}$
for $g=0,1$ respectively).

\begin{Theorem}\label{linear-general}
For every $2g-2+n>0$, we have
\be\label{prop1-eq1}
D\chi_{g,n}=(n+1)\chi_{g,n+1},
\ee
where $D$ is given by \eqref{eq-realization-D}.
Or equivalently,
\be\label{eq-linear-general}
(n+1)a_{g,n+1}^{i}=(2-2g-n) a_{g,n}^{i}+(n+i-1)a_{g,n}^{i-1}.
\ee
Here all possible unstable terms appearing in the right hand side is
set to be zero.
\end{Theorem}

\begin{proof}
The recursion \eqref{prop1-eq1} is simply
the realization of Lemma \ref{lem-original-D}.
Recall that Proposition \eqref{prop-chi-homog} tells us
$\chi_{g,n}(t,\kappa)$ is of the form:
\ben
\chi_{g,n}(t,\kappa)=
\sum_{i=0}^{3g-3+n}a_{g,n}^i\kappa^i\cdot (\frac{1}{t})^{n-2+2g}.
\een
Now compare the coefficients of $\kappa^k$ in two sides of \eqref{prop1-eq1},
where the left-hand-side is:
\ben
D\chi_{g,n}
=\biggl(\frac{\pd}{\pd t}+\kappa^2 t^{-1}\cdot\frac{\pd}{\pd \kappa}
+n\cdot \kappa t^{-1}\biggr)\chi_{g,n},
\een
then we may obtain:
\be
(2-2g-n)a_{g,n}^{i}+(i-1)\cdot a_{g,n}^{i-1}+n\cdot a_{g,n}^{i-1}
=(n+1)a_{g,n+1}^{i}.
\ee
\end{proof}

Theorem \ref{linear-general} tells us that for every $k\geq 1$,
the coefficient $a_{g,n+1}^{i}$ in $\chi_{g,n+1}$ is uniquely determined by
two coefficients $a_{g,n}^{i}$ and $a_{g,n}^{i-1}$ in $\chi_{g,n}$.
Therefore this suggests us to write down the coefficients in a triangle
(just like the YangHui's Triangle or Pascal's Triangle)
for a fixed genus $g\geq 0$,
such that an element in this triangle is a linear combination of
the two elements above it.

\subsection{The genus zero linear recursion and Manin's functional equation}
\label{eg-triangle0}

Let us study the linear recursion at genus $0$ in this subsection.
We will recover a functional equation of Manin for $\chi(\Mbar_{0,n})$
using this linear recursion.

For $g=0$, we put the coefficients $\{a_{0,n}^k\}$ for $n\geq 3$
in a triangle as follows.

\ben
&\begin{tikzpicture}
\node [align=center,align=center] at (0,0) {$\frac{1}{6}$};
\end{tikzpicture}&\\
&\begin{tikzpicture}
\node [align=center,align=center] at (-1,0) {$-\frac{1}{24}$};
\node [align=center,align=center] at (1,0) {$\frac{1}{8}$};
\end{tikzpicture}&\\
&\begin{tikzpicture}
\node [align=center,align=center] at (-2,0) {$\frac{1}{60}$};
\node [align=center,align=center] at (0,0) {$-\frac{1}{12}$};
\node [align=center,align=center] at (2,0) {$\frac{1}{8}$};
\end{tikzpicture}&\\
&\begin{tikzpicture}
\node [align=center,align=center] at (-3,0) {$-\frac{1}{120}$};
\node [align=center,align=center] at (-1,0) {$\frac{1}{18}$};
\node [align=center,align=center] at (1,0) {$-\frac{7}{48}$};
\node [align=center,align=center] at (3,0) {$\frac{7}{48}$};
\end{tikzpicture}&\\
&\begin{tikzpicture}
\node [align=center,align=center] at (-4,0) {$\frac{1}{210}$};
\node [align=center,align=center] at (-2,0) {$-\frac{1}{180}$};
\node [align=center,align=center] at (0,0) {$\frac{5}{36}$};
\node [align=center,align=center] at (2,0) {$-\frac{1}{4}$};
\node [align=center,align=center] at (4,0) {$\frac{3}{16}$};
\end{tikzpicture}&\\
&\begin{tikzpicture}
\node [align=center,align=center] at (0,0) {$\cdots \cdots$};
\end{tikzpicture}&
\een

The first row of this triangle is given by $\widetilde\chi_{0,3}=\frac{1}{6}$,
and the first element of each row is given by $\frac{1}{n!}\chi(\cM_{0,n})$
for $n\geq 3$.
Then by Theorem \ref{linear-general}
all the other elements are determined uniquely by the two elements above it
via the linear recursion relation
\be\label{genus0-linear}
(n+1)a_{0,n+1}^{k}=(2-n)a_{0,n}^{k}+(n+k-1)a_{0,n}^{k-1}.
\ee

\begin{Theorem}
Define a generating series $\chi(y,\kappa)$ of $\widetilde\chi_{0,n}(\kappa)$ by
\be\label{genus0-chi}
\chi(y,\kappa):=y+\sum_{n=3}^{\infty}n\kappa y^{n-1}\cdot\widetilde\chi_{0,n}(\kappa),
\ee
then $\chi$ satisfies the equation:
\be\label{refined-eqn}
\kappa(1+\chi)\log(1+\chi)=(\kappa+1)\chi-y.
\ee
\end{Theorem}

\begin{proof}
Recall that
$\widetilde\chi_{0,n}(\kappa)=\sum\limits_{i=0}^{n-3}a_{0,n}^i\kappa^i$.
The linear recursion \eqref{genus0-linear} gives us:
\be\label{rec-proof-manin}
(n+1)\widetilde\chi_{0,n+1}=(2-n)\widetilde\chi_{0,n}
+n\kappa\widetilde\chi_{0,n}+\kappa^2\frac{d}{d\kappa}\widetilde\chi_{0,n},
\ee
where $\frac{d}{d\kappa}\widetilde\chi_{0,n}$ can be rewritten as
\ben
\frac{d}{d\kappa}\widetilde\chi_{0,n}=
\half\sum_{i=3}^{n-1}
i(n+2-i)\widetilde\chi_{0,i}\widetilde\chi_{0,n+2-i}
\een
by the quadratic recursion \eqref{genus0-qrec}.
Now it is easy to check that \eqref{rec-proof-manin}
is equivalent to the following equation for $\chi$:
\ben
\chi=2\int_0^y \chi dy-y\chi+\half\kappa\chi^2+y.
\een
Applying $\frac{d}{dy}$ on both sides of this equation, we get:
\ben
\chi'=\kappa\chi\chi'+\chi-y\chi'+1,
\een
where $\chi'=\frac{\pd\chi}{\pd y}$.
Therefore we have:
\be\label{genus0-chi'}
\chi'=\frac{1+\chi}{1+y-\kappa\chi},
\ee
and the conclusion follows from this equation
by integrating with respect to $y$.
\end{proof}

By Taking $\kappa=1$ in the above theorem,
we recover the following result of Manin
(see \cite[(0.9)]{ma2}):
\begin{Corollary}
The generating series
\ben
\chi(y,1)=y+\sum_{n=3}^{\infty}\frac{y^{n-1}}{(n-1)!}\chi(\Mbar_{0,n})
\een
satisfies the functional equation:
\be\label{eqn-manin}
\big(1+\chi(y,1)\big)\log\big(1+\chi(y,1)\big)=2\chi(y,1)-y.
\ee
\end{Corollary}

\begin{Remark}
In our formalism,
the equation \eqref{refined-eqn} is a generalization of
Manin's equation \eqref{eqn-manin} to the case of
refined orbifold Euler characteristics.
Notice that the equation \eqref{refined-eqn} also appear
in \cite{ma2} and \cite{gls},
as a generalization of equation \eqref{eqn-manin} in a totally different sense!

In Manin's work \cite[Theorem 0.4.2]{ma2} (see also \cite[p. 197]{ma}),
he proved that for a compact smooth algebraic manifold $X$ of dimension $m$,
\ben
1+\sum_{n\geq 1}\chi(X[n])\frac{t^n}{n!}=(1+\eta)^{\chi(X)},
\een
where $X[n]$ is the Fulton-MacPherson compactification \cite{fm} of
the configuration space associated to $X$,
and $\eta$ is the unique solution in $t+t^2\bQ[[t]]$ to the equation:
\be\label{manin-gen}
m(1+\eta)\log(1+\eta)=(m+1)\eta-t.
\ee
This is of exactly the same form as our equation
\eqref{refined-eqn} for the refined orbifold Euler characteristic.
In Manin's result, the integer $m$ is the dimension of $X$,
while in our formalism $\kappa$ is a formal variable
encoding the codimensions of the strata in $\Mbar_{g,n}$.

In the work \cite{gls} of Goulden, Litsyn and Shevelev,
they study the generalized equation \eqref{manin-gen}
and give some results about the coefficients of the solution
(see \cite[\S 2.1]{gls}):
\be\label{genus0-gls}
\begin{split}
&a_{0,n}^{n-3}=\frac{(2n-5)!!}{n!};\\
&a_{0,n}^{n-4}=-\frac{n-3}{3}\cdot\frac{(2n-5)!!}{n!};\\
&a_{0,n}^{n-5}=\frac{(n-2)(n-3)(n-4)}{3^2}\cdot\frac{(2n-7)!!}{n!};\\
&a_{0,n}^{n-6}=-\frac{(n-4)(n-5)(5(n-2)^2+1)}{3^4\cdot 5}
\cdot\frac{(2n-7)!!}{n!};\\
&a_{0,n}^{n-7}=\frac{(n-4)(n-5)(n-6)(5(n-2)^3+4n-5)}
{2\cdot 3^5\cdot 5}\cdot\frac{(2n-9)!!}{n!};\\
&\cdots\cdots
\end{split}
\ee
Our $a_{0,n}^{n-k}$ is related to their notation $\mu_j(n)$ in the following way:
\ben
a_{0,n}^{n-k} = \frac{1}{n!} \mu_{k-2}(n-1).
\een

\end{Remark}

Solving the equation \eqref{refined-eqn} directly does not provide us a simple way
to obtain an explicit expression for the solution.
We will see in \S \ref{sol-genus0} that
one is able to obtain an explicit formula for the
generating series of coefficients $\{a_{0,n}^k\}$
with the help of the linear recursion \eqref{genus0-linear}.

\subsection{Relationship to Ramanujan polynomials}
\label{sec:Rama}

Note that the genus zero coefficients
$\{(-1)^{k+n+1}\cdot n!\cdot a_{0,n}^k\}$ are all integers.
They are the sequence  A075856 on
Sloane's on-line Encyclopedia of Integer Sequences \cite{Sloane}.
The references listed at this website leads us to note
the relationship of $\chi(y,\kappa)$
to many interesting works in combinatorics,
in particular, to  Ramanujan psi polynomials \cite{be, bew}.

This relationship also holds in the higher genus case.
Indeed, if we define
\ben
b_{g,n}^k:=(-1)^{k+n+1}\cdot n!\cdot a_{g,n}^k,
\een
then our recursion \eqref{eq-linear-general} for $\{a_{g,n}^k\}$
with fixed $g$ becomes:
\be
b_{g,n}^k=(n+2g-3)b_{g,n-1}^k+(n+k-2)b_{g,n-1}^{k-1}.
\ee
This is a special case of the following recursion relation for $x=2g-1$:
\be \label{eqn:Q-Rec}
Q_{n,k}(x)=(x+n-1)Q_{n-1,k}(x)+(n+k-2)Q_{n-1,k-1}(x)
\ee
discovered by Shor \cite[\S 2]{sh} in his proof of Cayley's formula for counting labelled trees.
In fact,
the recursion for the special values
$P_{g,n}^k:=Q_{n-1,k+1}(2g-1)$ is:
\be\label{zeng-rec-2g-1}
P_{g,n}^k=(n+2g-3)P_{g,n-1}^k+(n+k-2)P_{g,n-1}^{k-1}.
\ee
Shor shows that
\be
\sum_{k=0}^{n-1} Q_{n,k}(x) = (x+n)^{n-1}.
\ee
Zeng \cite[Proposition 7]{ze} establishes the following remarkable connection:
\be
Q_{n,k}(x) = \psi_{k+1}(n - 1, x + n),
\ee
where $\psi_k(r,x)$ are the Ramanujan polynomials  ($1\leq k\leq r+1$)   defined by:
\be \label{def:psi}
\sum_{j=0}^{\infty}\frac{(x+j)^{r+j}e^{-u(x+j)}u^j}{j!}=
\sum_{k=1}^{r+1}\frac{\psi_k(r,x)}{(1-u)^{r+k}}.
\ee
Ramanujan gives the following recursion relation of $\psi_k(r, x)$:
\be
\psi_k(r + 1, x) = (x - 1)\psi_k(r, x - 1) + \psi_{k-1}(r + 1, x) - \psi_{k-1}(r + 1, x - 1).
\ee
Berndt et al. \cite{be, bew}  obtain the following recursion relation:
\be
\psi_k(r, x) = (x - r - k + 1)\psi_k(r - 1, x) + (r + k - 2)\psi_{k-1}(r - 1, x).
\ee
This is used by Zeng \cite{ze} to connect Ramanujan polynomials to Shor polynomials.
He also gives the combinatorial interpretations of such a connection.
For some other related works,
see \cite{cg, dr, dgx}.
In particular,
Chen and Yang give a context-free grammar for the Ramanujan-Shor polynomials
in \cite{cy}.

However, we do not have
$b_{g,n}^k = P_{g,n}^k = Q_{n-1, k+1}(2g-1)$ for $g > 0$ because the initial values for $Q_{n,k}(x)$ are
\ben
&&Q_{1,0}(x)  =  1, \\
&&Q_{n,-1}(x)  =  0 , \quad n \geq 1, \\
&&Q_{1,k}(x)  =  0, \quad k \geq 1,
\een
and these correspond to
\ben
&&b_{g,2}^{-1}  =  1, \\
&&b_{g,n}^{-2}  =  0 , \quad n \geq 1, \\
&&b_{g,2}^{k-1}  =  0, \quad k \geq 1,
\een
but these are not satisfied by $b_{g,n}^k$.
Nevertheless,
when $g=0$, since $a_{0,n}^k$ makes sense only for $n \geq 3$ and $0 \leq k \leq n -3$,
these conditions are automatically satisfied,
so we get:
\be \label{eqn:a-Ramanujan}
(-1)^{k+n+1}\cdot n!\cdot a_{0,n}^k=Q_{n-1,k+1}(-1)
=\psi_{k+2}(n-2,n-2).
\ee
This relates our refined orbifold Euler characteristics of $\Mbar_{0,n}$
to special values of the Shor polynomials and Ramanujan polynomials,
hence to their combinatorial interpretations related to counting trees.
In particular,
one can use \cite[(35),  (21)]{ze} to compute $a_{0,n}^k$.
When $g>0$,
such connection is partially lost,
but it is still desirable to find combinatorial interpretations of the refined orbifold Euler
characteristics of $\Mbar_{g,n}$
that explains the recursion relation \eqref{eq-linear-general}.

\subsection{Generalization of Manin's functional equation to higher genera}

In this subsection,
we derive functional equations for generating series of $\{\widetilde\chi_{g,n}(\kappa)\}$
for $g\geq 1$.

First let us consider the case of genus one.
For $g=1$, we put the coefficients $\{a_{1,n}^k\}$ for $n\geq 1$
in a triangle (without the first row) as follows.

\ben
&\begin{tikzpicture}
\node [align=center,align=center] at (-1,0) {$-\frac{1}{12}$};
\node [align=center,align=center] at (1,0) {$\frac{1}{2}$};
\end{tikzpicture}&\\
&\begin{tikzpicture}
\node [align=center,align=center] at (-2,0) {$\frac{1}{24}$};
\node [align=center,align=center] at (0,0) {$-\frac{7}{24}$};
\node [align=center,align=center] at (2,0) {$\frac{1}{2}$};
\end{tikzpicture}&\\
&\begin{tikzpicture}
\node [align=center,align=center] at (-3,0) {$-\frac{1}{36}$};
\node [align=center,align=center] at (-1,0) {$\frac{2}{9}$};
\node [align=center,align=center] at (1,0) {$-\frac{5}{8}$};
\node [align=center,align=center] at (3,0) {$\frac{2}{3}$};
\end{tikzpicture}&\\
&\begin{tikzpicture}
\node [align=center,align=center] at (-4,0) {$\frac{1}{48}$};
\node [align=center,align=center] at (-2,0) {$-\frac{3}{16}$};
\node [align=center,align=center] at (0,0) {$\frac{199}{288}$};
\node [align=center,align=center] at (2,0) {$-\frac{41}{32}$};
\node [align=center,align=center] at (4,0) {$1$};
\end{tikzpicture}&\\
&\begin{tikzpicture}
\node [align=center,align=center] at (0,0) {$\cdots \cdots$};
\end{tikzpicture}&
\een

The first row is given by
$\widetilde\chi_{1,1}=-\frac{1}{12} +\frac{ 1}{2}\kappa$,
and the first element of each row is given by $\frac{1}{n!}\chi(\cM_{1,n})$
for $n\geq 1$.
Then all the other elements are determined by the two elements above it
via the linear recursion relation
\be\label{genus1-linear}
(n+1)a_{1,n+1}^{k}=-n\cdot a_{1,n}^{k}+(n+k-1)a_{1,n}^{k-1}.
\ee

Similarly to the case of genus zero,
we define:
\be
\label{eq-maningenerating-1}
\psi(y,\kappa):=\sum_{n=1}^\infty
n\kappa y^{n-1}\widetilde\chi_{1,n}(\kappa),
\ee
then $\psi$ can be determined by the following:

\begin{Theorem} \label{thm:Functional1}
We have
\be
(1+y-\kappa\chi)\psi=
\half \kappa^2\chi'-\frac{1}{12}\kappa,
\ee
where $\chi(y,\kappa)$ is defined by \eqref{genus0-chi} and $\chi'=\frac{\pd\chi}{\pd y}$.
\end{Theorem}

\begin{proof}
The linear recursion \eqref{genus1-linear} tells us:
\be\label{proof-genus1-rec}
(n+1)\widetilde\chi_{1,n+1}=-n\widetilde\chi_{1,n}
+n\kappa\widetilde\chi_{1,n}+\kappa^2\frac{d}{d\kappa}\widetilde\chi_{1,n},
\ee
where $\frac{d}{d\kappa}\widetilde\chi_{1,n}$ can be rewritten as
\ben
\frac{d}{d\kappa}\widetilde\chi_{1,n}=
\half(n+2)(n+1)\widetilde\chi_{0,n+2}+
\sum_{i=3}^{n+1}
i(n+2-i)\widetilde\chi_{0,i}\widetilde\chi_{1,n+2-i}
\een
by the quadratic recursion \eqref{genus1-qrec}.
Therefore the recursion \eqref{proof-genus1-rec} is equivalent to
the following equation for $\chi(y,\kappa)$ and $\psi(y,\kappa)$:
\ben
\psi=-y\psi+\half \kappa^2\chi'+\kappa\chi\psi-\frac{1}{12}\kappa.
\een
\end{proof}

Notice that we have \eqref{genus0-chi'}, thus
the above result can be rewritten as:
\ben
\psi=\frac{\half\kappa^2(1+\chi)}{(1+y-\kappa\chi)^2}-\frac{\frac{1}{12}\kappa}{(1+y-\kappa\chi)}.
\een
In particular, take $\kappa=1$ in the above results,
we ontain:

\begin{Corollary}
We have:
\ben
\psi(y,1)
=\frac{(1+\chi(y,1))}{2(1+y-\chi(y,1))^2}-\frac{1}{12(1+y-\chi(y,1))},
\een
where
\begin{equation*}
\chi(y,1)=y+\sum_{n=3}^{\infty}\frac{y^{n-1}}{(n-1)!}\chi(\Mbar_{0,n}),
\qquad
\psi(y,1)=\sum_{n=1}^{\infty}\frac{y^{n-1}}{(n-1)!}\chi(\Mbar_{1,n}).
\end{equation*}
\end{Corollary}

In general, for a fixed genus $g\geq 2$,
we can put the coefficients $\{a_{g,n}^k\}$ into a triangle
(without fisrt $3g-3$ rows for $g\geq 2$).
The first row is the coefficients of $\widetilde\chi_{g,0}(\kappa)$,
which can be determined either by the quadratic recursion relation \eqref{rec-integral}
using lower genus data,
or by expanding \eqref{gene-chi} directly
(actually the first method is more efficient than the second one).
Then every other element in this triangle
can be determined by the two elements above it
via the linear recursion \eqref{eq-linear-general}.
In \S \ref{sec5} we will present explicit solutions to such linear recursion relations
in all genera $g\geq 0$.

Now let us define $\varphi_0(y,\kappa):=\chi(y,\kappa)$,
and
\be
\label{eq-maningenerating-g}
\varphi_{g}(y,\kappa):=
\sum_{n=1}^\infty n\kappa y^{n-1}\widetilde\chi_{g,n}(\kappa),
\qquad g \geq 1.
\ee
In particular, it is clear that $\varphi_1=\psi$.
Then we have:

\begin{Theorem}
\label{thm:Functional2}
For every $g\geq 1$,
we have
\be
(y+1)\varphi_g'+(2g-1)\varphi_g=\half\kappa^2\varphi_{g-1}''+
\kappa\cdot\sum_{g_1+g_2=g}\varphi_{g_1}'\varphi_{g_2}.
\ee
where $\varphi_g':=\frac{\pd}{\pd y}\varphi_g$.
\end{Theorem}
\begin{proof}
Similar to the case of genus zero and genus one,
we combine the linear recursion \eqref{eq-linear-general}
and the quadratic recursion \eqref{rec-integral} to get:
\begin{equation*}
\begin{split}
(n+1)\widetilde\chi_{g,n+1}=&
(2-2g-n)\widetilde\chi_{g,n}+n\kappa\widetilde\chi_{g,n}
+\half\kappa^2(n+2)(n+1)\widetilde\chi_{g-1,n+2}\\
&+\half\kappa\sum_{\substack{g_1+g_2=g\\n_1+n_2=n+2\\n_1\geq 1,n_2\geq 1}}
n_1n_2\widetilde\chi_{g_1,n_1}\widetilde\chi_{g_2,n_2}.
\end{split}
\end{equation*}
This equation is equivalent to
\ben
\varphi_g=(2-2g)\int_0^y\varphi_g dy-y\varphi_g+\half\kappa^2\varphi_{g-1}'
+\half\kappa\sum_{g_1+g_2=g}\varphi_{g_1}\varphi_{g_2}+\widetilde\varphi_{g}(\kappa),
\een
where $\widetilde\varphi_{g}(\kappa)$ is a term independent of $y$.
Then the conclusion holds by applying $\frac{\pd}{\pd y}$ to the above equation.
\end{proof}

In particular,
we can take $\kappa=1$,
then then the above theorem gives us:

\begin{Corollary}
We have:
\ben
(y+1)\varphi_g'(y,1)+(2g-1)\varphi(y,1)=\half\varphi_{g-1}''(y,1)+
\sum_{g_1+g_2=g}\varphi_{g_1}'(y,1)\varphi_{g_2}(y,1),
\een
where
\begin{equation*}
\begin{split}
&\varphi_0(y,1)=y+\sum_{n=3}^{\infty}\frac{y^{n-1}}{(n-1)!}\chi(\Mbar_{0,n}),\\
&\varphi_g(y,1)=\sum_{n=1}^{\infty}\frac{y^{n-1}}{(n-1)!}\chi(\Mbar_{g,n}),
\qquad g\geq 1.
\end{split}
\end{equation*}
\end{Corollary}

\subsection{Operator formalism for the linear recursion}
\label{sec:Operator}

In this subsection we present a reformulation of the results in last subsection using an operator
formalism.

First recall that
the definition \eqref{eq-realization-D} of the operator $D$
depends on the number of external edges $|E^{ext}(\Gamma)|$
(which is just $n$ when acting on $\chi_{g,n}(t,\kappa)$).
Here in order to derive an operator formalism,
we are supposed to modify $D$ first to get an operator which does not depend on $n$.

Notice that the refined orbifold Euler characteristic $\chi_{g,n}(t,\kappa)$
is of the form:
\ben
\chi_{g,n}(t,\kappa)=t^{2-2g-n}\cdot
\widetilde\chi_{g,n}(\kappa),
\een
it follows that
\ben
\frac{\pd}{\pd t}\chi_{g,n}(t,\kappa)&=&
(2-2g-n)t^{1-2g-n}\widetilde\chi_{g,n}(\kappa)\\
&=&(2-2g)t^{-1}\cdot\chi_{g,n}(\kappa)
-nt^{-1}\cdot\chi_{g,n}(\kappa).
\een
Therefore the operator $D$ acts on $\widetilde\chi_{g,n}(t,\kappa)$ by:
\begin{equation*}
\begin{split}
D\chi_{g,n}(t,\kappa)=&
\biggl(\frac{\pd}{\pd t}+\kappa^2 t^{-1}\cdot\frac{\pd}{\pd \kappa}
+n\cdot \kappa t^{-1}\biggr)\chi_{g,n}(t,\kappa)\\
=&\biggl((1-\kappa)\frac{\pd}{\pd t}+\frac{\kappa^2}{t}\frac{\pd}{\pd\kappa}
+\frac{(2-2g)\kappa}{t}\biggr)\chi_{g,n}(t,\kappa).
\end{split}
\end{equation*}
Define $\widetilde D$ to be the operator:
\be
\label{eq-operator-tildeD}
\widetilde D:=(1-\kappa)\frac{\pd}{\pd t}+\frac{\kappa^2}{t}\frac{\pd}{\pd\kappa}
+\frac{(2-2g)\kappa}{t},
\ee
then $\widetilde{D}$ is a suitable modification of $D$,
and the recursion \eqref{prop1-eq1} is equivalent to:
\be\label{linear-new}
\chi_{g,n+1}(t,\kappa)=\frac{1}{n+1}\widetilde D \chi_{g,n}(t,\kappa),
\ee
or,
\begin{equation}
\begin{split}
&\chi_{0,n}=\frac{3!}{n!}\cdot\widetilde D ^{n-3}\chi_{0,3}, \quad n\geq 3;\\
&\chi_{1,n}=\frac{1}{n!}\cdot\widetilde D ^{n-1}\chi_{1,1}, \quad n\geq 1;\\
&\chi_{g,n}=\frac{1}{n!}\cdot\widetilde D ^{n}\chi_{g,0}, \quad g\geq 2.\\
\end{split}
\end{equation}

Now let us define:
\ben
&&\chi_0:=\sum_{n=3}^\infty n!\cdot\chi_{0,n}(t,\kappa)
=t^2\cdot\sum_{n=3}^\infty \frac{n!}{t^n}\widetilde\chi_{0,n}(\kappa);\\
&&\chi_1:=\sum_{n=1}^\infty n!\cdot\chi_{1,n}(t,\kappa)
=\sum_{n=1}^\infty \frac{n!}{t^n}\widetilde\chi_{1,n}(\kappa);\\
&&\chi_g:=\sum_{n=0}^\infty n!\cdot\chi_{g,n}(t,\kappa)
=t^{2-2g}\cdot\sum_{n=0}^\infty \frac{n!}{t^n}\widetilde\chi_{g,n}(\kappa),
\quad g\geq 2,
\een
then clearly they are generating series of $\{\widetilde\chi_{g,n}(\kappa)\}$.
Similarly, we can define another type of generating series:
\ben
&&\widehat\chi_0:=\sum_{n=3}^\infty \chi_{0,n}(t,\kappa)
=t^2\cdot\sum_{n=3}^\infty \frac{1}{t^n}\widetilde\chi_{0,n}(\kappa);\\
&&\widehat\chi_1:=\sum_{n=1}^\infty \chi_{1,n}(t,\kappa)
=\sum_{n=1}^\infty \frac{1}{t^n}\widetilde\chi_{1,n}(\kappa);\\
&&\widehat\chi_g:=\sum_{n=0}^\infty \chi_{g,n}(t,\kappa)
=t^{2-2g}\cdot\sum_{n=0}^\infty \frac{1}{t^n}\widetilde\chi_{g,n}(\kappa),
\quad g\geq 2.
\een
Then the linear recursion \eqref{linear-new} gives us the following:

\begin{Theorem} \label{thm:Add-Point}
We have
\ben
&&\chi_{g}(t,\kappa)=\frac{1}{1-\widetilde{D}} \chi_{g,0}(t,\kappa);\\
&&\widehat\chi_{g}(t,\kappa)=e^{\widetilde{D}} \chi_{g,0}(t,\kappa),
\een
where $\widetilde{D}$ is given by \eqref{eq-operator-tildeD}.
Here we use the convention:
\begin{equation*}
\begin{split}
&\widetilde D ^3\chi_{0,0}:=6\chi_{0,3};\qquad
\widetilde D\chi_{1,0}:=2\chi_{1,1};\qquad
\widetilde\chi_{1,0}:=0;\\
&\widetilde D ^j\chi_{0,0}:=0,\qquad j=0,1,2.
\end{split}
\end{equation*}
\end{Theorem}

\subsection{Motivic realization of the abstract quantum field theory}

The Euler characteristic is an example of motivic characteristic classes.
In this subsection we speculate on the possibility of
a realization of our abstract quantum field theory
by using the orbifold motivic classes of the
Deligne-Mumford moduli space $\overline\cM_{g,n}$ of stable curves.

The theory of motivic measures and motivic integrals was first
introduce by Kontsevich \cite{kon2},
and generalized to singular spaces by Denef and Loeser \cite{dl}.
Let $\mathcal{VAR}$ be the category of complex algebraic varieties of finite type,
and $R$ be a commutative ring with unity.
A motivic class is a ring homomorphism
\ben
[\cdot]:K_0(\mathcal{VAR})\to R,
\een
where $K_0(\mathcal{VAR})$ is the Grothendieck ring of complex varieties.
In other words, a motivic class is a map $[\cdot]$ satisfying:
\ben
&&\text{(1) }[X]=[X'],  \quad \text{for } X\cong X';\\
&&\text{(2) }[X]=[X\setminus Y]+[Y], \quad \text{for a closed subvariety $Y\subset X$};\\
&&\text{(3) }[X\times Y]=[X]\cdot[Y];\\
&&\text{(4) }[pt]=1.
\een

For our purpose we need to consider the orbifold motivic class of
the Deligne-Mumford moduli space $\overline\cM_{g,n}$ of stable curves.
We need to have a relation
\be
[X/G] = [X]/|G|
\ee
for a finite group $G$ acting on $X$.
We also need to make sense of the orbifold motivic class of $\cM_{g,n}$.
(The work \cite{ya} might be useful for this purpose.)
Then we understand the orbifold motivic class of $\overline\cM_{g,n}/S_n$ to be
\be\label{motivic}
[\overline\cM_{g,n}/S_n]=\sum_{\Gamma\in\cG_{g,n}^c}\biggl(
\frac{1}{|\Aut(\Gamma)|}\prod_{v\in V(\Gamma)}[\cM_{g(v),\val(v)}]\biggr),
\ee
where $[\cM_{g(v),\val(v)}]$ denotes the orbifold motivic class of $\cM_{g(v),\val(v)}$
(whatever that means).

Such consideration would give us a natural realization
of our abstract quantum field theory.
We assign the Feynman rule to a stable graph $\Gamma\in\cG_{g,n}^c$ as follows.
The contribution of a vertex $v\in V(\Gamma)$ is defined to be
the orbifold motivic class of $\cM_{g(v),\val(v)}$:
\ben
\omega_v=[\cM_{g(v),\val(v)}],
\een
and the contribution of an internal edge $e\in E(\Gamma)$ is set to be $\omega_e=1$.
Thus the Feynman rule is
\ben
\Gamma \mapsto \omega_\Gamma = \prod_{v\in V(\Gamma)} [\cM_{g(v),\val(v)}].
\een
Therefore the abstract $n$-point function $\wcF_{g,n}$ is realized by
\ben
\wF_{g,n}=\sum_{\Gamma\in\cG_{g,n}^c}\biggl(
\frac{1}{|\Aut(\Gamma)|}\prod_{v\in V(\Gamma)}[\cM_{g(v),\val(v)}]\biggr),
\een
i.e., $\wF_{g,n}=[\overline\cM_{g,n}/S_n]$ is the
orbifold motivic class of $\overline\cM_{g,n}$.
Once this has been done,
we can apply our formalism introduced in \cite{wz} to this case
to derive some quadratic recursion relations
for the orbifold motivic class $[\overline\cM_{g,n}/S_n]$.

\section{Structures of $\chi_{g,0}(t,\kappa)$}
\label{sec4}

In \S \ref{sec:Linear} and \S \ref{sec:Operator}
we have reduced the computations of $\chi_{g,n}(t,\kappa)$
to the problem of computing $\chi_{g,0}(t, \kappa)$.
In this section we will present various methods to solve $\chi_{g,0}(t, \kappa)$.
In particular,
we show that the generating series $G_k(z)$ of the coefficients $\{a_{g,0}^k\}$
can be represented as some polynomials in certain generating series $V_n(z)$
of $\chi(\cM_{g,n})$,
and the explicit formulas for $V_n(z)$ can be given in terms of the Barnes $G$-function.
We show that finding the expression of $G_k(z)$ in terms of $V_n(z)$
is equivalent to the topological 1D gravity
(with a genus-shift).

\subsection{Computations of $\chi_{g,0}(t,\kappa)$ by quadratic recursions}

In this subsection we specialize the quadratic recursion
in \S \ref{sec-quadr-rec} to the case $n=0$.

Taking $n=0$ in Theorem \ref{thm-rec-chi-g,n},
we obtain a recursion:
\be
\begin{split}
&\chi_{g,0}(t,\kappa)\\
=&F_{g,0}^{orb}(t)+\frac{1}{2}\int_{0}^{\kappa}\bigg(
DD\chi_{g-1,0}+
\sum_{r=2}^{g-2}
D\chi_{r,0}D\chi_{g-r,0}
+2D\chi_{g-1,0}\cdot \chi_{1,1}
\bigg)d\kappa\\
=&F_{g,0}^{orb}(t)+\frac{1}{2}\int_{0}^{\kappa}\biggl[
\big(\frac{\pd}{\pd t}+\frac{\kappa^2}{t}\frac{\pd}{\pd \kappa}+\frac{\kappa}{t}\big)
\big(\frac{\pd}{\pd t}+\frac{\kappa^2}{t}\frac{\pd}{\pd \kappa}\big)\chi_{g-1,0}\\
&\qquad\qquad+\sum_{r=2}^{g-2}
\big(\frac{\pd}{\pd t}+\frac{\kappa^2}{t}\frac{\pd}{\pd \kappa}\big)\chi_{r,0}
\big(\frac{\pd}{\pd t}+\frac{\kappa^2}{t}\frac{\pd}{\pd \kappa}\big)\chi_{g-r,0}\\
&\qquad\qquad+2\big(\frac{\pd}{\pd t}+\frac{\kappa^2}{t}\frac{\pd}{\pd \kappa}\big)\chi_{g-1,0}
\cdot (-\frac{1}{12}+\frac{1}{2}\kappa)t^{-1}
\biggr]d\kappa\\
\end{split}
\ee
for $g\geq 3$,
where $F_{g,0}^{orb}(t)$ is given by \eqref{eq-realization-vertices}.
Recall that the refined orbifold Euler characteristic
$\chi_{g,n}(t,\kappa)$ is of the following form:
\ben
\chi_{g,n}(t,\kappa)
=\widetilde\chi_{g,n}(\kappa)\cdot t^{2-2g-n}
=\sum_{k=0}^{3g-3+n}a_{g,n}^k\kappa^k\cdot (\frac{1}{t})^{n-2+2g},
\een
then we have
\be
\bigg(\frac{\pd}{\pd t} \chi_{g,n}\bigg)\bigg|_{t=1}
=(2-2g-n)\widetilde\chi_{g,0}.
\ee
Thus by taking $t=1$ in the above quadratic recursion,
we get the following:

\begin{Theorem}
For $g\geq 3$, we have:
\begin{equation*}
\begin{split}
\widetilde\chi_{g,0}=&\frac{B_{2g}}{2g(2g-2)}+\frac{1}{2}\int_{0}^{\kappa}
\biggl[\biggl((3-2g)+\kappa^2\frac{\pd}{\pd\kappa}+2\kappa-\frac{1}{6}\biggr)
\biggl((4-2g)+\kappa^2\frac{\pd}{\pd\kappa}\biggr)\widetilde\chi_{g-1,0}
\\
&+\sum_{r=2}^{g-2}
\biggl((2-2r)+\kappa^2\frac{\pd}{\pd\kappa}\biggr)\widetilde\chi_{r,0}
\cdot\biggl((2-2g+2r)+\kappa^2\frac{\pd}{\pd\kappa}\biggr)\widetilde\chi_{g-r,0}
\biggr]d\kappa.
\end{split}
\end{equation*}
\end{Theorem}

It follows that the recursion for the coefficients $\{a_{g,0}^k\}$ is:
\begin{Corollary}
We have:
\begin{equation*}
\begin{split}
&a_{g,0}^k=\frac{1}{2k}\cdot\biggl\{
(\frac{17}{6}-2g)(4-2g)a_{g-1,0}^{k-1}
+\biggl(k(4-2g)+(k-2)(\frac{17}{6}-2g)\biggr)a_{g-1,0}^{k-2}\\
&\quad+(k^2-3k)a_{g-1,0}^{k-3}
+\sum_{r=2}^{g-2}\biggl[\sum_{l+m=k-1}(2-2r)(2-2g+2r)a_{r,0}^{l}a_{g-r,0}^{m}\\
&\quad+\sum_{l+m=k-2}\biggl(m(2-2r)+l(2-2g+2r)\biggr)a_{r,0}^{l}a_{g-r,0}^{m}
+\sum_{l+m=k-3}lm a_{r,0}^{l}a_{g-r,0}^{m}
\biggr]\biggr\}
\end{split}
\end{equation*}
for $g\geq3$ and $k>0$, and the initial values are
$a_{g,0}^0=\frac{B_{2g}}{2g(2g-2)}$.
In particular, the recursion for the sequence
$\{a_{g,0}^{3g-3}\}_{g\geq 2}$ is:
\begin{equation*}
a_{g,0}^{3g-3}=\frac{1}{2}(3g-6)a_{g-1,0}^{3g-6}+
\frac{1}{6g-6}
\sum_{r=2}^{g-2}(3r-3)(3g-3r-3)
a_{r,0}^{3r-3}a_{g-r,0}^{3g-3r-3}
\end{equation*}
for $g\geq 3$.
\end{Corollary}

\begin{Example}
Using the quadratic recursion relation in the above theorem and
the initial value
\begin{equation*}
\widetilde\chi_{2,0}(\kappa)=
-\frac{1}{240} + \frac{13 }{288}\kappa -\frac{ 1}{6}\kappa^2 + \frac{5 }{24}\kappa^3,
\end{equation*}
one can recursively compute $\widetilde\chi_{g,0}(\kappa)$.
The following are the first few examples:
\begin{equation*}
\begin{split}
\widetilde\chi_{3,0}(\kappa)=&\frac{
1}{1008} - \frac{19 }{1440}\kappa + \frac{1307 }{17280}\kappa^2
- \frac{2539 }{10368}\kappa^3 + \frac{
 35 }{72}\kappa^4 - \frac{55 }{96}\kappa^5 + \frac{5 }{16}\kappa^6,
 \\
\widetilde\chi_{4,0}(\kappa)=&
-\frac{1}{1440} + \frac{6221}{604800}\kappa
- \frac{17063 }{241920}\kappa^2 +\frac{
 187051 }{622080}\kappa^3 - \frac{2235257 }{2488320}\kappa^4 \\
 &
 + \frac{182341 }{92160}\kappa^5 - \frac{
 66773 }{20736}\kappa^6 + \frac{8549 }{2304}\kappa^7
 - \frac{1045 }{384}\kappa^8 + \frac{1105 }{1152}\kappa^9,
\\
\widetilde\chi_{5,0}(\kappa)=&\frac{
1}{1056} - \frac{181 }{12096}\kappa + \frac{32821 }{290304}\kappa^{2}
- \frac{667199 }{1209600}\kappa^{3} + \frac{
 114641981 }{58060800}\kappa^{4} \\
 &
 - \frac{578872613 }{104509440}\kappa^{5} + \frac{
 374564131 }{29859840}\kappa^{6} - \frac{229328099 }{9953280}\kappa^{7}
 + \frac{ 2805265 }{82944}\kappa^{8}\\
 &
  - \frac{3182161 }{82944}\kappa^{9} + \frac{145883 }{4608}\kappa^{10}
 - \frac{26015 }{1536}\kappa^{11} + \frac{565 }{128}\kappa^{12},\\
 \cdots & \cdots
\end{split}
\end{equation*}

\end{Example}

\subsection{The generating series of $a_{g,0}^k$ for fixed $k$}

Each $\widetilde\chi_{g,0}(\kappa)$ is
the generating function of $a_{g,0}^k$ for fixed $g$.
In this subsection let us consider the generating function
of $a_{g,0}^k$ for fixed $k$.

Define $G_k(z)$ to be such generating functions:
\be
\label{eq-def-generating-G}
G_k(z):=\sum_{g\geq 2}a_{g,0}^k\cdot z^{2-2g}.
\ee
The recursion relations for $a_{g,0}^k$ in last subsection can be translated into the following
recursion relations for $G_k(z)$ ($k\geq 1$):

\begin{Proposition}
We have:
\begin{equation}
\label{rec-G}
\begin{split}
&G_k(z)=a_{2,0}^k z^{-2}+
\frac{z^{-2}}{2k}\biggl[
(\theta-\frac{7}{6})\theta G_{k-1}(z)+
\biggl((2k-2)\theta-\frac{7}{6}k+\frac{7}{3}\biggr)G_{k-2}(z)\\
&\quad+(k^2-3k)G_{k-3}(z)
+\sum_{l=0}^{k-1}\theta G_l(z)\cdot \theta G_{k-1-l}(z)\\
&\quad+2\sum_{l=0}^{k-2}\theta G_l(z)\cdot(k-2-l)G_{k-2-l}(z)
+\sum_{l=0}^{k-3}l(k-3-l)G_{l}(z)G_{k-3-l}(z)
\biggr],
\end{split}
\end{equation}
where $\theta:=z\frac{d}{dz}$,
and the initial values are $a_{2,0}^k=0$ except for:
\begin{align*}
a_{2,0}^0 & = - \frac{1}{240}, & a_{2,0}^1 & = \frac{13}{288}, &
a_{2,0}^2 & = - \frac{1}{6}, & a_{2,0}^3 & = \frac{5}{24}.
\end{align*}

\end{Proposition}

\begin{Example}
By the Harer-Zagier formula \eqref{eq-harer-zagier},
we have:
\be
\label{eq-G0-Bern}
G_0(z)=\sum_{g=2}^{\infty}\frac{B_{2g}}{2g(2g-2)}z^{2-2g}.
\ee
At the Digital Library of Mathematical Functions (NIST) (\S24.11) one can find the formula
\be
B_{2n}\sim (-1)^{n-1}4\sqrt{\pi n} \biggl( \frac{n}{\pi e}\biggr)^{2n},
\ee
so the series on the right-hand of \eqref{eq-G0-Bern}
has zero radius of convergence,
and it should only be understood as an asymptotic series for now.
The quadratic recursion \eqref{rec-G} gives us the following expressions
in terms of Bernoulli numbers:
\begin{equation*}
\begin{split}
&G_1(z)
=\sum_{g\geq 3}\biggl(\frac{1}{2}-\frac{5}{24(g-1)}\biggr)B_{2g-2}z^{2-2g}
+\sum_{g\geq 4}
\sum_{\substack{g_1+g_2=g\\g_1,g_2\geq 2}}\frac{B_{2g_1}B_{2g_2}}{8g_1g_2}z^{2-2g}
+\frac{13}{288}z^{-2},\\
&G_2(z)
=\sum_{g\geq 4}(\frac{g^2}{2}-\frac{23}{12}g+\frac{493}{288}
+\frac{13}{288(2g-4)})B_{2g-4}z^{2-2g}
-\sum_{g\geq 3}\frac{B_{2g-2}}{4(g-1)}z^{2-2g}
\\
&\quad+\sum_{g\geq 5}\sum_{\substack{g_1+g_2=g-1\\g_1\geq 2,g_2\geq 2}}
\biggl(\frac{g_2}{2g_1}+\frac{1}{4}-\frac{11}{24g_1}+\frac{5}{96g_1g_2}
\biggr)B_{2g_1}B_{2g_2}z^{2-2g}\\
&\quad+\sum_{g\geq 6}\sum_{\substack{g_1+g_2+g_3=g\\g_1,g_2,g_3\geq 2}}
\biggl(\frac{1}{8g_1g_2}-\frac{1}{16g_1g_2g_3}
\biggr)B_{2g_1}B_{2g_2}B_{2g_3}z^{2-2g}
+\frac{247}{3456}z^{-4}-\frac{1}{6}z^{-2}.
\end{split}
\end{equation*}

\end{Example}

\subsection{Generating series of $\chi(\cM_{g,n})$ in terms of Barnes $G$-function}
\label{sec:Barnes}

By the method of last subsection it is clear that one can express each $G_k(z)$
in terms of Bernoulli numbers.
A priori,
they are just series whose radii of convergence are zero.
We will relate them to the Barnes $G$-function $G(z)$
in this subsection.

First let us recall a result due to Distler-Vafa \cite{dv}
which represents the generating series of $\chi(\cM_{g,0})$
(i.e., the series $G_0(z)$) in terms of
the Euler Gamma-function:
\be
\Gamma(z) : = \int_0^\infty t^{z-1}e^{-t}dt.
\ee
Recall that Gamma-function $\Gamma(z)$ has the following Weierstrass product:
\be
\label{eq-Gamma--Weierstrass}
\frac{1}{\Gamma(z)} = ze^{\gamma z} \prod_{n=1}^\infty
\biggl\{\biggl(1+\frac{z}{n}\biggr)e^{-\frac{z}{n}}\biggr\},
\ee
where $\gamma$ is the Euler-Mascheroni's constant defined by:
\be
\gamma:=\lim_{n\to \infty} \biggl( 1+\frac{1}{2} + \cdots + \frac{1}{n} - \log n\biggr).
\ee
The Stirling series is  the following asymptotic expansion of $\log \Gamma(z)$:
\be
\log \Gamma(z) \sim (z-\frac{1}{2})\log z -z + \frac{1}{2}\log(2\pi)
+ \sum_{k=1}^\infty \frac{B_{2k}}{2k(2k-1)z^{2k-1}}.
\ee
By \eqref{eq-G0-Bern} we have:
\begin{equation*}
\frac{d}{dz}G_0(z)=-\sum_{g=2}^{\infty}\frac{B_{2g}}{2g}z^{1-2g}
\end{equation*}
is the asymptotic series of
\ben
&& (-z)\biggl(-\frac{d}{dz}\log\Gamma(z+1)+\log z+\frac{1}{2}z^{-1}-\frac{1}{12}z^{-2}
\biggr) \\
& = & z\frac{d}{dz}\log\Gamma(z+1)-z\log z-\frac{1}{2}+\frac{1}{12}z^{-1}.
\een
Recall that $\Gamma(z)$ is singular at $z=0$,
in order to carry out the integration from $0$ to $z$,
in the above we have used the fact that
\be \label{eq-Gamma-z+1}
\Gamma(z+1) = z\Gamma(z).
\ee
Thus we obtain the following:

\begin{Lemma} \label{lm-g0}
Let $G_0(z)$ be the generating series of $\chi(\cM_{g,0})$ for $g \geq 2$,
\be
G_0(z):=\sum_{g\geq 2} \chi(\cM_{g,0})z^{2-2g}.
\ee
Then up to some constant $C$, we have for $z \gg 0$:
\be
\label{eq-G0-gamma}
 \int_0^z \biggl(z\frac{d}{dz}\log\Gamma(z+1)\biggr) dz
-\half z^2\log z +\frac{1}{4}z^2 - \half z+\frac{1}{12}\log z -C
\sim G_0(z).
\ee
\end{Lemma}

\begin{Remark}
This Lemma is essentially due to Distler and Vafa \cite{dv}.
Our modification is based on the following observation.
In \cite[(10)]{dv}, the following formula is used:
\be
F(\mu) = \int^\mu x\frac{d}{dx}\log\Gamma(x) \; dx.
\ee
Note $\Gamma(x)$ is not define at $x= 0$,
so this expression needs some suitable regularization at $x=0$.
Our formula \eqref{eq-G0-gamma} avoids this problem.
\end{Remark}

\begin{Remark}
We will later fix the constant $C$ to be $\zeta'(-1)$.
On the left-hand side of \eqref{eq-G0-gamma},
 $\half z^2\log(z)-\frac{1}{4}z^2+\half z$ should be understood as
`the contribution of $\cM_{0,0}$',
and $-\frac{1}{12}\log(z)$ should be understood as
`the contribution of $\cM_{1,0}$'.
In other words,
we define
\bea
&& F_{0,0}(z) = \half z^2\log(z)-\frac{1}{4}z^2+\half z+\zeta'(-1), \label{def-F00} \\
&& F_{1,0}(z) = -\frac{1}{12}\log(z). \label{def-F10}
\eea
With these understood,
\eqref{eq-G0-gamma} can be rewritten as:
\be \label{eq-Chi-M-g-0}
F_{0,0}(z)+F_{1,1}(z)+\sum_{g =2}^\infty \chi(\cM_{g,0})z^{2-2g}
= \int_0^z \biggl(z\frac{d}{dz}\log\Gamma(z+1)\biggr) dz.
\ee
\end{Remark}

\begin{Definition} \label{def:V}
Define
\be
\label{eq-def-vertexV}
\begin{split}
&V_0(z):=\sum_{g=2}^\infty \chi(\cM_{g,0})z^{2-2g};\\
&V_n(z):=\sum_{g=1}^\infty \chi(\cM_{g,n})z^{2-2g-n},\qquad n=1,2;\\
&V_n(z):=\sum_{g=0}^\infty \chi(\cM_{g,n})z^{2-2g-n},\qquad n\geq 3,
\end{split}
\ee
to be the total contributions of all stable vertices of valence $n$.

\end{Definition}

By \eqref{eq-harer-zagier}, one can see
\be
\chi({\cM}_{g,n+1}) = (2-2g-n) \cdot \chi({\cM}_{g,n}).
\ee
It follows that:
\be
\label{eq-Vn-derivative}
\begin{split}
& V_0'(z) = V_1(z)+\frac{1}{12}z^{-1}, \\
& V_1'(z) = V_2(z), \\
& V_2'(z) = V_3(z) - z^{-1}, \\
& V_n'(z) = V_{n+1}(z), \qquad n \geq 3.
\end{split}
\ee
The series $V_n(z)$ can be related to the Gamma function as follows:
\begin{equation*}
\begin{split}
V_0(z)=&G_0(z)
\sim \int_0^z \biggl(z\frac{d}{dz}\log\Gamma(z+1)\biggr) dz
-\half z^2\log z+\frac{1}{4}z^2 - \half z - C +\frac{1}{12}\log z,\\
V_1(z)
=&\frac{d}{dz}G_0(z)-\frac{1}{12}z^{-1} \sim z\frac{d}{dz}\log\Gamma(z)-z\log z+\half,\\
V_2(z)
=&\frac{d^2}{dz^2}G_0(z)+\frac{1}{12}z^{-2}
\sim z\frac{d^2}{dz^2}\log\Gamma(z)+\frac{d}{dz}\log\Gamma(z)-\log z-1,
\end{split}
\end{equation*}
and for $n\geq 3$,
\begin{equation*}
\begin{split}
V_n(z)
=&\frac{d^n}{dz^n}G_0(z)+(-1)^{n+1}(n-3)!\cdot z^{2-n}
+\frac{(-1)^n}{12}(n-1)!\cdot z^{-n}\\
\sim &z\frac{d^n}{dz^n}\log\Gamma(z)+(n-1)\frac{d^{n-1}}{dz^{n-1}}\log\Gamma(z).
\end{split}
\end{equation*}

The main result of this subsection is that
$V_n(z)$ can be represented in terms of the Barnes $G$-functions.
The Barnes G-function, also called the double Gamma function,
is defined by the equation:
\be
G(z+1):=(2\pi)^{z/2}e^{-z(z+1)/2- \frac{1}{2}\gamma z^2} \prod_{n=1}^\infty
\biggl\{ \biggl(1+\frac{z}{n}\biggr)^ne^{-z+z^2/(2n)}\biggr\}.
\ee
It satisfies the relations
\begin{align}
G(z + 1) & =\Gamma(z)G(z), & G(1) = 1.
\end{align}
One can show that:
\ben
\frac{d}{dz}\log G(z+1)
& = & \frac{1}{2}\log(2\pi)-z - \frac{1}{2} - \gamma z
+ \sum_{n=1}^\infty \bigg( \frac{1}{1+\frac{z}{n}}- 1 + \frac{z}{n}\bigg) \\
& = & \frac{1}{2}\log(2\pi)-z - \frac{1}{2} - \gamma z
+ \sum_{n=1}^\infty \sum_{m=2}^\infty (-1)^{m}\frac{z^m}{n^m},
\een
and similarly,
by \eqref{eq-Gamma--Weierstrass} and \eqref{eq-Gamma-z+1} we have:
\ben
\frac{d}{dz}\log \Gamma(z+1)
& = & - \gamma
-\sum_{n=1}^\infty \biggl\{ \frac{\frac{1}{n}}{1+\frac{z}{n}}- \frac{1}{n}\biggr\} \\
& = & - \gamma
- \sum_{n=1}^\infty \sum_{j=1}^\infty (-1)^{j}\frac{z^j}{n^{j+1}},
\een
and so  we also have:
\be
 \frac{d}{dz} \log G(z+1)
= \frac{1}{2}\log(2\pi)
- \frac{1}{2} - z +z \frac{d}{dz}\log \Gamma(z+1).
\ee
Therefore,
\be
\begin{split}
\log G(z+1)
= & \frac{z}{2}\log(2\pi)
- \frac{z}{2} - \frac{1}{2}z^2 + \int_0^z \bigg(z \frac{d}{dz}\log \Gamma(z+1)\bigg) dz \\
= & \frac{z}{2}\log(2\pi)
- \frac{z}{2} - \frac{1}{2}z^2 + z\log \Gamma(z+1) -\int_0^z \log \Gamma(z+1) dz.
\end{split}
\ee
Barnes \cite{ba} obtained the following asymptotic expansion of $\log G(z+1)$:
\begin{equation*}
\log G(z+1)
\sim \zeta'(-1) + \frac{z}{2}\log (2\pi) + \biggl(\frac{z^2}{2}-\frac{1}{12}\biggr)\log z
- \frac{3z^2}{4}
+  \sum_{k=1}^\infty \frac{B_{2k+2}}{2k(2k+2)z^{2k}},
\end{equation*}
where $\zeta$ is the Riemann zeta function.
(For a modern derivation,
see \cite{fl}.)
In particular,
we have
\be\begin{split}
\int_0^z \bigg(z \frac{d}{dz}\log \Gamma(z+1)\bigg) dz
= & \log G(z+1) -\frac{z}{2} \log (2\pi) + \frac{z}{2} + \frac{z^2}{2}  \\
\sim & \zeta'(-1)   + \biggl(\frac{z^2}{2}-\frac{1}{12}\biggr)\log z
- \frac{z^2}{4} + \frac{z}{2} \\
+ & \sum_{k=1}^\infty \frac{B_{2k+2}z^{-2k}}{2k(2k+2) },
\end{split}
\ee
so by \eqref{eq-G0-Bern} or \eqref{eq-G0-gamma} we get:
\be
\label{eq-G0-BarnesG}
G_0(z) = \log G(z+1)
-\biggl( \zeta'(-1) + \frac{z}{2}\log (2\pi) + \biggl(\frac{z^2}{2}-\frac{1}{12}\biggr)\log z
- \frac{3z^2}{4} \biggr),
\ee
and the constant $C$ in \eqref{eq-G0-gamma} is $\zeta'(-1)$.

Similarly,
one can rewrite $V_n(z)$ in terms of the Barnes $G$-function,
and the result is as follows:
\begin{Lemma}
\label{lem-barnes-V}
We have:
\be
\begin{split}
&V_0(z) \sim \log G(z+1)
-\biggl( \zeta'(-1) + \frac{z}{2}\log (2\pi) + \biggl(\frac{z^2}{2}-\frac{1}{12}\biggr)\log z
- \frac{3z^2}{4} \biggr),\\
&V_1(z) \sim
\frac{d}{dz}\log G(z+1)
-\frac{1}{2} \log(2\pi) -z\log z +z,\\
&V_2(z) \sim \frac{d^2}{dz^2} \log G(z+1)-\log z,
\end{split}
\ee
and
\be
V_n (z) = \frac{d^n}{dz^n}\log  G(z+1),
\qquad \forall n\geq 3.
\ee
\end{Lemma}
\begin{proof}
The conclusion follows form \eqref{eq-G0-BarnesG}, \eqref{eq-Vn-derivative} and direct computations.
\end{proof}

\begin{Remark}
With the notations in \eqref{def-F00} and \eqref{def-F10},
we can rewrite the above results in the following form:
\be \label{eq-0-point-G}
\log G(z+1)
- \frac{z}{2} \log (2\pi) + \frac{z}{2}+\frac{z^2}{2}
\sim F_{0,0}(z) + F_{1,0}(z) + \sum_{g \geq 2} \chi(\cM_{g,0})z^{2-2g},
\ee
\be \label{eq-1-point-G}
\frac{d}{dz}\log G(z+1)
- \frac{1}{2} \log (2\pi) + \frac{1}{2}+ z
\sim F_{0,0}'(z) + \sum_{g \geq 1} \chi(\cM_{g,1})z^{1-2g},
\ee
\be \label{eq-2-point-G}
\frac{d^2}{dz^2}\log G(z+1) + 1
\sim F_{0,0}''(z) + \sum_{g \geq 1} \chi(\cM_{g,2})z^{-2g},
\ee
and
\be \label{eq-n-point-G}
\frac{d^n}{dz^n}\log G(z+1)
\sim  \sum_{g \geq 0} \chi(\cM_{g,n})z^{2-2g-n}, \quad \forall n \geq 3.
\ee
\end{Remark}

\begin{Remark}
There is a probability interpretation of the left-hand side of \eqref{eq-0-point-G}
by \cite[Theorem 1.2]{ny}.
Recall that a gamma variable $\gamma_a$ with parameter $a > 0$
is a random variable with  distribution density function
\be
 \frac{t^{a-1} \exp(-t)}{\Gamma(a)}, \quad t > 0.
 \ee
If $\{\gamma_n\}_{n\geq 1}$ are independent gamma random variables
with respective parameters $n$, then for $z\in \bC$ such that $\Re(z) > -1$,
\be
\begin{split}
& \lim_{N\to \infty} \frac{1}{N^{z^2/2}}
\bE\biggl[\exp \biggl(-z \biggl(\sum^N_{n=1} \frac{\gamma_n}{n}-N\biggr)\biggr)\biggr] \\
=& \biggl(\exp\biggl(\frac{z}{2}+z^2-\frac{z}{2}\log(2\pi)+\log G(z+1)\biggr)^{-1}.
\end{split}
\ee
As a consequence of \eqref{eq-0-point-G},
we then have the following asymptotic expansion:
\be
\begin{split}
& \lim_{N\to \infty} \frac{1}{N^{z^2/2}}
\bE\biggl[\exp \biggl(-z \biggl(\sum^N_{n=1} \frac{\gamma_n}{n}-N\biggr)\biggr)\biggr] \\
\sim &  F_{0,0}(z) + F_{1,0}(z) + \sum_{g \geq 2} \chi(\cM_{g,0})z^{2-2g}.
\end{split}
\ee
This is a very surprising connection that relates a probability problem to
algebraic geometry and mathematical physics.
\end{Remark}

Now if we multiply \eqref{eq-0-point-G}-\eqref{eq-n-point-G} by $x^0$, $x^1$, $x^2/2!$ and $x^n/n!$
respectively and sum them up we get:
\be
\begin{split}
& \log G(z+x+1)- \frac{z}{2} \log (2\pi) + \frac{z}{2}+\frac{z^2}{2}
+ (- \frac{1}{2} \log (2\pi) + \frac{1}{2}+ z)x+ \frac{x^2}{2} \\
\sim &  F_{0,0}(z) + F_{1,0}(z)
+ F_{0,0}'(z)x+ \frac{x^2}{2}F_{0,0}''(z)
+  \sum_{2g-2+n>0} \chi(\cM_{g,n})z^{2-2g-n}x^n.
\end{split}
\ee

\subsection{Expressions of $G_k(z)$ in terms of Barnes $G$-function}
\label{sec:G}
Now let us find a way to express $G_k(z)$ in terms of the generating series $V_n(z)$.
Then using the above lemma,
this will give us a way to represent $G_k(z)$ in terms of the Barnes $G$-function.
First we need:

\begin{Theorem}
\label{thm:Gauss-Gk}
Define a partition function
\be\label{eq:Gauss-Gk}
\widetilde Z=\frac{1}{\sqrt{2\pi\lambda^2}}
\int\exp\biggl[\lambda^{-2}\biggl(-\half x^2+\sum_{n\geq 0}
\lambda^2  V_n(z) \cdot \frac{x^n}{n!}
\biggr)\biggr]dx,
\ee
then:
\be
\log(\widetilde Z)=
\sum_{k\geq 0} \lambda^{2k} G_k(z).
\ee
 the coefficient of $\lambda^{2k}$ in $\log(\widetilde Z)$ equals $G_{k}(z)$
for every $k\geq 0$.
\end{Theorem}
\begin{proof}
Recall in Theorem \ref{thm:chi(g,0)} we have shown that if
\begin{equation*}
\widehat{Z}^{orb}(t,\kappa)=\frac{1}{(2\pi\lambda^2\kappa)^{\frac{1}{2}}}
\int \exp\biggl\{\sum_{2g-2+n>0}\frac{\lambda^{2g-2}}{n!}\chi(\cM_{g,n})t^{2-2g-n}
\eta^n-\frac{\lambda^{-2}}{2\kappa}\eta^2\biggr\}d\eta,
\end{equation*}
then its logarithm equals to:
\ben
&& \sum_{g=2}^\infty \chi_{g,0}(t,\kappa)  \lambda^{2g-2}
= \sum_{g=2}^\infty \sum_{i\geq 0} a^i_{g,0}t^{2-2g} \kappa^i  \lambda^{2g-2}
= \sum_{i \geq 0} G_i(\frac{t}{\lambda}) \cdot \kappa^i.
\een

Now let us rewrite $\widehat{Z}^{orb}(t, \kappa)$ in terms of $V_n$ as follows:
\ben
\widehat{Z}^{orb}(t,\kappa)
= \frac{1}{(2\pi\lambda^2\kappa)^{\frac{1}{2}}}
\int \exp\biggl\{ \sum_{n \geq 0}\frac{1}{n!} V_n(\frac{t}{\lambda}) \cdot
\biggl(\frac{\eta}{\lambda}\biggr)^n
-\frac{1}{2\kappa}\biggl(\frac{\eta}{\lambda}\biggr)^2\biggr\}d\eta.
\een
We make the following change of variables: $z= \frac{t}{\lambda}$,
$x= \frac{\eta}{\lambda}$, and after that change $\kappa$ to $\tilde\lambda^2$.
Then we see that the logarithm of
\ben
\frac{1}{(2\pi\tilde\lambda^2)^{\frac{1}{2}}}
\int \exp\biggl\{ \sum_{n \geq 0}\frac{1}{n!} V_n(z) \cdot x^n
-\frac{1}{2\tilde\lambda^2}x^2\biggr\}dx
\een
is $\sum\limits_{k \geq 0} G_k(z) \cdot \tilde\lambda^{2k}
=\sum\limits_{k \geq 0} G_k(z) \cdot \kappa^k$.
\end{proof}

Now plugging Lemma \ref{lem-barnes-V} into this theorem,
we obtain:
\be
\begin{split}
\exp\bigg(\sum_{k\geq 0} \lambda^{2k} G_k(z)\bigg) =& \frac{1}{\sqrt{2\pi\lambda^2}}
\int dx \cdot \exp\bigg[ -\frac{\lambda^{-2}}{2}x^2 + \log G(x+z+1)\\
& -\biggl( \zeta'(-1) + \frac{z}{2}\log (2\pi) + \biggl(\frac{z^2}{2}-\frac{1}{12}\biggr)\log z
- \frac{3z^2}{4} \biggr)\\
&+x\bigg(-\frac{1}{2} \log(2\pi) -z\log z +z\bigg)
-\frac{x^2}{2} \log z \bigg].
\end{split}
\ee
Thus:
\begin{Corollary}
\label{cor-gk-barnes}
We have:
\be
\sum_{k\geq 0} \lambda^{2k} G_k(z) = \log
\bigg( \frac{1}{\sqrt{2\pi\lambda^2}}
\int G(x+z+1)\cdot \exp (\widetilde S(x,z)) dx
\bigg),
\ee
where
\be
\begin{split}
\widetilde S(x,z)=&-\frac{\lambda^{-2}}{2}x^2
-\biggl( \zeta'(-1) + \frac{z}{2}\log (2\pi) + \biggl(\frac{z^2}{2}-\frac{1}{12}\biggr)\log z
- \frac{3z^2}{4} \biggr)\\
&+x\bigg(-\frac{1}{2} \log(2\pi) -z\log z +z\bigg)
-\frac{x^2}{2} \log z.
\end{split}
\ee
\end{Corollary}

Now Theorem \ref{thm:Gauss-Gk} gives us a way to express $G_k(z)$
in terms of the generating series $\{V_n(z)\}$.
In the next two subsections,
we will see that $G_k(z)$ can be represented as a polynomial in
$V_1(z),V_2(z),\cdots,V_{2k}(z)$,
and we can use the formalism recalled in \S \ref{sec2} to
present a simple recursion to compute these polynomials.
Furthermore,
the formal Gaussian integral \eqref{eq:Gauss-Gk}
relates this problem to the theory of topological 1D gravity,
and this will be explained in \S \ref{sec-relation-1d}-\S \ref{sec-orb-genus0} and \S \ref{sec:KP}.

\subsection{Representing $G_k(z)$ in terms of $V_n(z)$}

Now in this subsection we show that $G_k(z)$ can be represented as a polynomial in
the generating series $\{V_n(z)\}$.
The following result is a consequence of Theorem \ref{thm:Gauss-Gk}:

\begin{Theorem}
\label{thm:FR-Gk}
For $k\geq 1$, we have
\be
\label{FR-chi-g,0}
G_k(z)=\sum_{\widetilde\Gamma\in\widetilde\cG_k}
\frac{1}{|\Aut(\widetilde\Gamma)|}
\prod_{v\in V(\widetilde\Gamma)}
V_{\val_v}(z),
\ee
where $\widetilde\cG_k$ is the set of all connected graphs with unmarked vertices
(not necessarily stable),
which has $k$ internal edges and no external edge;
and $V(\widetilde\Gamma)$ is the set of vertices of $\widetilde\Gamma$.
The weight of vertices $\{V_n(z)\}_{n\geq 1}$ are given by Lemma \ref{lem-barnes-V}.
\end{Theorem}
\begin{proof}
By Theorem \ref{thm:Gauss-Gk},
$G_k(z)$ can be expressed by a Feynman sum over graphs with vertex contributions given
by $V_n(z)$ for a vertex of valence $n$,
and with $1$ as propagator.
As in Example \ref{exm:G1G2},
the genera of the relevant graphs are defined by assigning every vertex to have genus $1$,
then
\ben
g(\widetilde\Gamma):=h^1(\widetilde\Gamma)+|V(\widetilde\Gamma)|,
\een
where $h^1(\widetilde\Gamma)$ is the number of independent loops in $\widetilde\Gamma$.
By Euler's formula,
\be
h^1(\widetilde{\Gamma}) = 1 + |E(\widetilde{\Gamma})| - |V(\widetilde{\Gamma})|,
\ee
and so  we have
\be\label{g=k+1}
g(\widetilde\Gamma)=|E(\widetilde{\Gamma})|+1.
\ee
It follows that  $G_k$ is the sum over graphs with $k$ internal edges.
\end{proof}

\begin{Corollary}
\label{cor-chi-g,0}
If we assign $ V_n(z)$ to be of degree $n$ for every $n\geq 0$,
then $G_{k}(z)$ is of degree $2k$.
\end{Corollary}

Now let us see a few examples.
\begin{Example} \label{exm:G1G2}
We can directly check the expressions for $G_1(z)$ and $G_2(z)$ can be reformulated
in the following form:
\ben
&&G_1(z)=\half V_2(z)+\half  V_1(z)^2,\\
&&G_2(z)=\frac{1}{8} V_4+\frac{1}{4} V_2^2
+\frac{1}{2} V_1 V_3+\frac{1}{2} V_1^2 V_2.
\een
The graphs for $G_1(z)$ are:
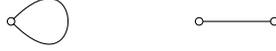
\ben
\begin{tikzpicture}[scale=1.05]
\draw [fill] (0,0) circle [radius=0.06];
\draw (0.03,0.03) .. controls (1,1) and (1,-1) ..  (0.03,-0.03);
\draw [fill] (2.5,0) circle [radius=0.06];
\draw [fill] (3.5,0) circle [radius=0.06];
\draw (2.55,0)--(3.45,0);
\end{tikzpicture}
\een
A first sight,
these two graphs have different genus: o
ne has genus one and the other has genus zero.
It seems odd that both contribute to $G_1(z)$.
To explain this discrepancy,
we need to let each vertex of graphs have genus $1$,
then both of the above graphs have genus $2$.
Similarly, the graphs for $G_2(z)$ are:
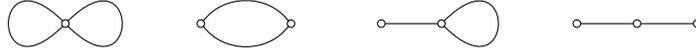
\ben
\begin{tikzpicture}[scale=1.05]
\draw [fill] (0,0) circle [radius=0.06];
\draw (-0.03,0.03) .. controls (-1,1) and (-1,-1) ..  (-0.03,-0.03);
\draw (0.03,0.03) .. controls (1,1) and (1,-1) ..  (0.03,-0.03);
\draw [fill] (1.8,0) circle [radius=0.06];
\draw [fill] (3,0) circle [radius=0.06];
\draw (1.83,0.03) .. controls (2.1,0.4) and (2.7,0.4) ..  (2.97,0.03);
\draw (1.83,-0.03) .. controls (2.1,-0.4) and (2.7,-0.4) ..  (2.97,-0.03);
\draw [fill] (4.2,0) circle [radius=0.06];
\draw [fill] (5,0) circle [radius=0.06];
\draw (4.25,0)--(4.95,0);
\draw (5.03,0.03) .. controls (6,1) and (6,-1) ..  (5.03,-0.03);
\draw [fill] (6.8,0) circle [radius=0.06];
\draw [fill] (7.6,0) circle [radius=0.06];
\draw [fill] (8.4,0) circle [radius=0.06];
\draw (6.85,0)--(7.55,0);
\draw (7.65,0)--(8.35,0);
\end{tikzpicture}
\een
When the vertices are assigned to have genus $1$,
then all these graphs have genus $3$.
\end{Example}

\subsection{Recursive computations for $G_k(z)$ and Young diagram representation}
\label{sec:Young}

In the previous subsection
we have seen that $G_k(z)$ can be expressed in terms of $\{V_1(z),\cdots,V_{2k}(z)\}$
by the formula \eqref{FR-chi-g,0},
and the explicit formulas for $V_n(z)$ are given in terms of the Barnes $G$-function.
In this subsection we present a simple quadratic recursion
to compute the expressions $G_k(z)$ in terms of $V_n(z)$.

Here we need first to clarify some possible confusions of our notations.
Although $G_k(z)$ and $V_k(z)$ are
some Laurent series in $z$ which have precise expressions,
in the following subsections in order to find simple polynomial expressions for
$G_k=G_k(V_1,\cdots,V_{2k})$
we will sometimes treat $V_k$
as some formal variables and use operations such like
taking partial derivatives with respect to $V_k$.

To be more precise,
one is supposed to introduce a family of new formal variables $\{v_n\}$
and a family of polynomials $H_k=H_k(v_1,\cdots,v_{2k})$
such that $H_k$ are represented using $\{v_n\}$ in such a way that
is similar to the formal Gaussian integral \eqref{eq:Gauss-Gk}
or the Feynman rule \eqref{FR-chi-g,0}.
In other works,
we study a `theory' with coulpling constants $\{v_n\}$.
And after finding the recursion formulas for $H_k$,
we take the special values $v_n=V_n(z)$ and then
\ben
H_k(v_1,v_2,\cdots)\big|_{v_n=V_n(z)}=G_k(z).
\een
But for simplicity,
in what follows simply regard $V_n$ as some formal variables by abuse of notations.

Now let us present a simple quadratic recursion for the
polynomial expressions of $G_k=G_k( V_1, V_2,\cdots)$
using the formalism of abstract QFT and its realization.
Although the graph sum \eqref{FR-chi-g,0} seems not to be
over stable graphs,
here we can again use the trick of `shifting the genus of vertices'
(see Example \ref{exm:G1G2}).
By assigning genus $g_v=1$ for every vertex $v\in V(\widetilde\Gamma)$,
we understand the problem of counting connected graphs
$\widetilde\Gamma\in \widetilde\cG_k$ for $k\geq 1$
as counting connected stable graphs
$\Gamma\in\cG_{k+1,0}^c$ of genus $k+1$ with no external edges,
whose vertices are all of genus one.
Here the Feynman rule we need is:
\ben
&&v \mapsto  V_{\val_v},\qquad
\text{for $v\in V(\Gamma)$ and $g_v=1$};\\
&&v \mapsto 0,\qquad
\text{for $v\in V(\Gamma)$ and $g_v\not= 1$};\\
&&e\mapsto 1, \qquad\text{for $e\in E(\Gamma)$}.
\een
Now consider the realization of the edge-cutting operator $K$ and
the edge-adding operators $\cD=\pd+\gamma$.
Since all graphs $\widetilde\Gamma\in \widetilde\cG_k$ has $k$ internal edges,
the operator $K$ is realized by simply multiplying by $k$.
Moreover,
since in this case the weight of all genus zero vertices are zero,
every procedure involving `adding a trivalent vertex of genus zero'
in the definition of the operator $\cD=\pd+\gamma$
will be realized to zero.
Therefore,
the only procedure we need to consider is
adding an external edge on one of the vertices of genus one.
Now it is clear that $\cD$ is realized by an operator $d_V$ with:
\begin{itemize}
\item[1)]$d_V  (V_l)= V_{l+1}$, for every $l\geq 1$;
\item[2)]$d_V$ acts on polynomials in $\{ V_l\}_{l\geq 1}$
via Leibniz rule.
\end{itemize}
Therefore, we can simply take:
\be
\label{eq-def-dV}
d_V:=\sum_{l=0}^{\infty} V_{l+1}\frac{\pd}{\pd  V_l}.
\ee
Now the realization of Theorem \ref{thm-original-rec} gives us:
\begin{Theorem}
For every $k\geq 1$, we have:
\be\label{rec-G-2}
k\cdot G_k=\half\biggl(
d_V ^2 G_{k-1}+\sum_{r=1}^k
d_V G_{r-1}\cdot d_V G_{k-r}
\biggr).
\ee
\end{Theorem}

Inductively,
one might see that $2^k\cdot k!\cdot G_k$ is a polynomial
in $V_1,V_2,\cdots,V_{2k}$ with integer coefficients.

\begin{Example}
Using this recursion, one can easily compute:
\begin{equation*}
\begin{split}
G_3=&\frac{1}{2}   V_1 ^2   V_2 ^2 + \frac{1}{6}   V_2 ^3 +
 \frac{1}{6}   V_1 ^3   V_3  +
   V_1    V_2    V_3  +
 \frac{5}{24}   V_3 ^2 + \frac{1}{4}   V_1 ^2   V_4  +
\frac{ 1}{4}   V_2    V_4 \\
&+
\frac{ 1}{8}   V_1    V_5  +\frac{ 1}{48}  V_6;\\
G_4=&\frac{1}{2}   V_1 ^2   V_2 ^3 + \frac{1}{8}   V_2 ^4 +
 \frac{1}{2}   V_1 ^3   V_2    V_3  +
 \frac{3}{2}   V_1    V_2 ^2   V_3  +
 \frac{1}{2}   V_1 ^2   V_3 ^2 +
 \frac{5}{8}   V_2    V_3 ^2\\
 & +
\frac{ 1}{24 }  V_1 ^4   V_4  +
\frac{ 3}{4}   V_1 ^2   V_2    V_4  +
\frac{ 3}{8 }  V_2 ^2   V_4  +
\frac{ 2}{3}   V_1    V_3    V_4  +
\frac{ 1}{12}   V_4 ^2 +\frac{ 1}{12 }  V_1 ^3   V_5\\
&  +
\frac{ 3}{8}   V_1    V_2    V_5  +
\frac{ 7}{48}   V_3    V_5  +
 \frac{1}{16 }  V_1 ^2   V_6  +
\frac{ 1}{16}   V_2    V_6  +
\frac{ 1}{48 }  V_1   V_7 + \frac{1}{384}  V_8;\\
\cdots&\cdots
\end{split}
\end{equation*}

\end{Example}

The operator $d_V$ also appears in an earlier work
of the second author \cite[\S 7.6]{zhou1}.
In that work,
$d_V$ is interpreted using some rules
in terms of Young diagrams,
which is similar to the Littlewood-Richardson rule.
Let us recall these rules in our case.
First by Corollary \ref{cor-chi-g,0},
we can rewrite the Feynman rule as:
\be
\label{eq-Feynman-partition}
G_k=\sum_{|\mu|=2k}\sum_{\Gamma\in \widetilde\cG_\mu}
\frac{1}{|\Aut(\Gamma)|} V_{\mu},
\ee
where the sum is over all partitions $\mu=(\mu_1,\cdots,\mu_l)$ of $2k$,
and $\widetilde\cG_\mu\subset\widetilde\cG_k$ is the set of all graphs with vertices
which have valences $\mu_1,\cdots,\mu_l$ respectively,
and
\ben
 V_{\mu}:= V_{\mu_1} V_{\mu_2}\cdots V_{\mu_l}.
\een
The partition $\mu$ can be represented in terms of a Young diagram
which has $l$ rows.
Then the action of $d_V$ on $ V_\mu$ can be interpreted as
adding a box to this Young diagram on the right
of one of the $l$ rows.
If the new diagram is not a Young diagram any more,
then we switch the rows of this diagram to make it a Young diagram.
Then we have
\ben
d_V (  V_\mu ) =\sum_{\nu\in  Y_\mu} V_\nu,
\een
where $Y_\mu$ is the set of Young diagrams obtained from $\mu$
using the above rule.
For example:
\begin{equation*}
\begin{tikzpicture}[scale=1.2]
\draw (0,0)--(0.6,0);
\draw (0,0.3)--(0.6,0.3);
\draw (0,0.6)--(0.9,0.6);
\draw (0,0.9)--(0.9,0.9);
\draw (0,0)--(0,0.9);
\draw (0.3,0)--(0.3,0.9);
\draw (0.6,0)--(0.6,0.9);
\draw (0.9,0.6)--(0.9,0.9);
\node [align=center,align=center] at (1.8,0.45) {$\mapsto$};
\draw (2.7,0)--(2.7,0.9);
\draw (3,0)--(3,0.9);
\draw (3.3,0)--(3.3,0.9);
\draw (3.6,0.6)--(3.6,0.9);
\draw (3.9,0.6)--(3.9,0.9);
\draw (2.7,0)--(3.3,0);
\draw (2.7,0.3)--(3.3,0.3);
\draw (2.7,0.6)--(3.9,0.6);
\draw (2.7,0.9)--(3.9,0.9);
\node [align=center,align=center] at (4.5,0.45) {$+$};
\draw (2.7+2.4,0)--(2.7+2.4,0.9);
\draw (3+2.4,0)--(3+2.4,0.9);
\draw (3.3+2.4,0)--(3.3+2.4,0.9);
\draw (3.6+2.4,0.3)--(3.6+2.4,0.9);
\draw (2.7+2.4,0)--(3.3+2.4,0);
\draw (2.7+2.4,0.3)--(3.6+2.4,0.3);
\draw (2.7+2.4,0.6)--(3.6+2.4,0.6);
\draw (2.7+2.4,0.9)--(3.6+2.4,0.9);
\node [align=center,align=center] at (6.6,0.45) {$+$};
\draw (2.7+2.4+2.1,0)--(2.7+2.4+2.1,0.9);
\draw (3+2.4+2.1,0)--(3+2.4+2.1,0.9);
\draw (3.3+2.4+2.1,0)--(3.3+2.4+2.1,0.9);
\draw (3.6+2.4+2.1,0.6)--(3.6+2.4+2.1,0.9);
\draw (3.6+2.4+2.1,0.3)--(3.6+2.4+2.1,0);
\draw (2.7+2.4+2.1,0)--(3.6+2.4+2.1,0);
\draw (2.7+2.4+2.1,0.3)--(3.6+2.4+2.1,0.3);
\draw (2.7+2.4+2.1,0.6)--(3.6+2.4+2.1,0.6);
\draw (2.7+2.4+2.1,0.9)--(3.6+2.4+2.1,0.9);
\node [align=center,align=center] at (2.1+0.5,-1) {$=$};
\draw (2.7+0.5,0-1.45)--(2.7+0.5,0.9-1.45);
\draw (3+0.5,0-1.45)--(3+0.5,0.9-1.45);
\draw (3.3+0.5,0-1.45)--(3.3+0.5,0.9-1.45);
\draw (3.6+0.5,0.6-1.45)--(3.6+0.5,0.9-1.45);
\draw (3.9+0.5,0.6-1.45)--(3.9+0.5,0.9-1.45);
\draw (2.7+0.5,0-1.45)--(3.3+0.5,0-1.45);
\draw (2.7+0.5,0.3-1.45)--(3.3+0.5,0.3-1.45);
\draw (2.7+0.5,0.6-1.45)--(3.9+0.5,0.6-1.45);
\draw (2.7+0.5,0.9-1.45)--(3.9+0.5,0.9-1.45);
\node [align=center,align=center] at (4.5+0.5,-1) {$+$ $2$};
\draw (2.7+2.4+0.5,0-1.45)--(2.7+2.4+0.5,0.9-1.45);
\draw (3+2.4+0.5,0-1.45)--(3+2.4+0.5,0.9-1.45);
\draw (3.3+2.4+0.5,0-1.45)--(3.3+2.4+0.5,0.9-1.45);
\draw (3.6+2.4+0.5,0.3-1.45)--(3.6+2.4+0.5,0.9-1.45);
\draw (2.7+2.4+0.5,0-1.45)--(3.3+2.4+0.5,0-1.45);
\draw (2.7+2.4+0.5,0.3-1.45)--(3.6+2.4+0.5,0.3-1.45);
\draw (2.7+2.4+0.5,0.6-1.45)--(3.6+2.4+0.5,0.6-1.45);
\draw (2.7+2.4+0.5,0.9-1.45)--(3.6+2.4+0.5,0.9-1.45);
\node [align=center,align=center] at (6.95,-1.35) {$.$};
\end{tikzpicture}
\end{equation*}

Therefore in the recursion \eqref{rec-G-2},
the term $d_V^2 G_{k-1}$ corresponds to adding two boxes to all Young diagrams
corresponding to partitions of $2(k-1)$ via the above rule;
and the term $d_V G_{r-1}\cdot d_V G_{k-r}$
is to add a box to a Young diagram corresponding to a partition of $2(r-1)$,
and add another box to a Young diagram corresponding to a partition of $2(k-r)$,
and then put one of the resulting diagrams underneath the other
and switch the rows if necessary to get a new Young diagram.

\subsection{Relationship to topological 1D gravity}
\label{sec-relation-1d}

The integral representation \eqref{eq:Gauss-Gk}
of the partition function $\widetilde{Z}$
establishes a connection to some earlier work on topological 1D gravity
by the second author \cite{zhou1}.

By Corollary \ref{cor-chi-g,0} we know that $G_k$ is of the form:
\ben
G_{k}=\sum_{|\mu|=2k}a_\mu V_{\mu_1}V_{\mu_2}\cdots V_{\mu_l},
\een
where the summations is over all partitions $\mu=(\mu_1,\mu_2,\cdots,\mu_l)$ of $2k$.
Now let us use another way to represent this partition:
\ben
\mu=(1^{m_1}2^{m_2}\cdots (2k)^{m_{2k}}),
\een
where $m_i$ is the number of $i$ appearing in the sequence $(\mu_1,\mu_2,\cdots,\mu_l)$.
Then $G_k$ can be expanded in the following way:
\ben
G_k=\sum_{|\mu|=2k}
\frac{G_\mu}{m_1!\cdots m_{2k}!}
V_1^{m_1}\cdots V_{2k}^{m_{2k}},
\een
where by \eqref{eq-Feynman-partition} the coefficients $G_\mu$ can be expressed as:
\be
\label{eq-graphsum-1d}
G_\mu=\sum_{\Gamma\in \widetilde\cG_\mu}
\frac{1}{|\Aut(\Gamma)|}\cdot \prod_{i=1}^{2k} m_i!.
\ee

Given a graph $\Gamma\in \widetilde\cG_\mu$
where $\mu=(\mu_1,\cdots,\mu_l)$ is a partition of $2k$,
by Euler's formula we know that the genus $g$ of $\Gamma$ is determined by:
\ben
l-\half(\mu_1+\cdots+\mu_l)=1-g,
\een
therefore
\be
\label{eq-selectionrule}
g=k+1-l.
\ee
Now comparing the graph sum \eqref{eq-graphsum-1d} with \cite[Proposition 4.4]{zhou1},
one easily obtain the following:

\begin{Theorem}
Given a partition $\mu=(\mu_1,\mu_2,\cdots,\mu_l)=(1^{m_1}\cdots (2k)^{m_{2k}})$ of $2k$,
we have:
\be
\begin{split}
G_\mu=&\langle
\tau_{\mu_1-1}\tau_{\mu_2-1}\cdots \tau_{\mu_{l}-1}
\rangle_{k+1-l}^{1D}\\
=&\langle
\tau_0^{m_1}\tau_1^{m_2}\cdots\tau_{n-1}^{m_n}
\rangle_{k+1-l}^{1D},
\end{split}
\ee
where the right-hand-sides of the above equation are
correlators of the topological 1D gravity.
\end{Theorem}

The coefficients $G_\mu$ will understood as the correlators
of the theory defined by the partition function \eqref{eq:Gauss-Gk},
and the equation \eqref{eq-selectionrule} is equivalent to
the selection rule for the topological 1D gravity \cite[(119)]{zhou1}.

Now we have seen that the study of $\{G_k(V_1,V_2,\cdots)\}$ are
equivalent to the topological 1D gravity
(up to a shift of the genus).
Then it follows that the results developed in the works \cite{ny, ny2, zhou1}
can also be applied to this problem.

For example,
similar to the formula \cite[(116)]{zhou1},
by expanding the formal Gaussian integral in Theorem \ref{thm:Gauss-Gk}
one knows that:
\begin{equation}
\label{eqn:Young}
G_k=\sum_{n\geq 1}\frac{(-1)^{n+1}}{n}\sum_{\sum\limits_{j=1}^{n}k_j=k}
\prod_{j=1}^n \biggl(\sum_{m_j>0}\sum_{\sum\limits_{i=1}^{m_j}l_i^{(j)}=2k_j}
\frac{(2k_j-1)!!  \cdot  V_{l_1^{(j)}}\cdots V_{l_{m_j}^{(j)}}}
{m_j!\cdot l_1^{(j)}!\cdots l_{m_j}^{(j)}!}
\biggr)
\end{equation}
for every $k\geq 1$.
However, this formula is not convenient for practical computations
since the right-hand-side are too complicated.
In the following subsections,
we will follow \cite[\S 5-\S 6]{zhou1}
and derive the Virasoro constraints
(or equivalently, the flow equations and the polymer equation)
for the partition function $\widetilde{Z}$.

\subsection{Operator formalism for computing $\widetilde{Z}$}
\label{sec:Operator-Z}

The recursion relations \eqref{rec-G-2} can be solved in an operator formalism
for $\widetilde{Z}$:

\begin{Theorem} \label{thm:Gk}
The partition function
$\widetilde Z= \exp\bigg(\sum\limits_{k\geq 0} \lambda^{2k} G_k\bigg)$
is given by the following formula:
\be
\widetilde Z= e^{\half \lambda^2 \tilde{d}_V^2} e^{G_0},
\ee
where $\tilde{d}_V$ is an operator defined by:
\be
\label{eq-def-tilde-d}
\tilde{d}_V=V_1 \frac{\pd}{\pd G_0} +
\sum_{l=1}^{\infty} V_{l+1}\frac{\pd}{\pd  V_l}.
\ee
\end{Theorem}

\begin{proof}
Denote $\hbar:=\lambda^2$,
then the quadratic recursion \eqref{rec-G-2} is equivalent to:
\ben
\frac{\pd}{\pd \hbar}\widetilde Z =
\half \tilde{d}_V^2 \widetilde Z.
\een
Thus we have:
\ben
\widetilde{Z}=
\sum_{n=0}^\infty \frac{\hbar^n}{n!}
\bigg(\frac{\pd^n \widetilde Z}{\pd \hbar^n}\bigg)
\bigg|_{\hbar=0}
=\sum_{n=0}^\infty \frac{\hbar^n}{n!}\cdot
\bigg(\big(\frac{1}{2}\tilde{d}_V^2\big)^n \widetilde{Z}\bigg)
\bigg|_{\hbar=0}.
\een
Notice that for every $k\geq 0$ we always have:
\ben
\big(\tilde{d}_V^k \widetilde{Z}\big)\big|_{\hbar=0}
=\tilde{d}_V^k (e^{G_0})
\een
since $\widetilde{Z}$ is of the form
$\widetilde Z= \exp\bigg(\sum\limits_{k\geq 0} \hbar^k G_k\bigg)$.
Therefore
\be
\widetilde{Z}=
\sum_{n=0}^\infty \frac{\hbar^n}{n!}\cdot
\big(\half \tilde{d}_V^2 \big)^n \big(e^{G_0}\big)
=\exp\big(\half \hbar \tilde{d}_V^2\big)\exp\big(G_0\big).
\ee
\end{proof}

One can also derive other solutions in operator formalism for $\widetilde{Z}$
using the methods in the study of the topological 1D gravity.
Now let us introduce the flow equations
and the polymer equation for $\widetilde{Z}$
(see \cite[\S 6.1]{zhou1}).

\begin{Theorem}
We have the flow equations:
\be\label{flow}
\frac{\pd\widetilde Z}{\pd V_k}=
\frac{1}{k!}\frac{\pd^k}{\pd V_1^k}\widetilde Z,
\qquad
k\geq 1,
\ee
and the polymer equation:
\be
\label{polymer}
\sum_{k\geq 0}\frac{V_{k+1}-\lambda^{-2}\delta_{k,1}}{k!}\cdot
\frac{\pd^{k}}{\pd V_1^{k}}\widetilde Z
=0.
\ee
\end{Theorem}

These equations can be proved by directly taking partial derivatives
to the partition function \eqref{eq:Gauss-Gk}
and here we omit the details.
A straightforward consequence of the flow equations is the following:

\begin{Corollary}
We have the following expression for $\widetilde{Z}$
in operator formalism:
\be \label{Z-V1}
\widetilde Z=\exp\biggl(
\sum_{k=2}^{\infty}\frac{1}{k!}V_k\frac{d^k}{dV_1 ^k}
\biggr)\exp\biggl(
V_0+\half \lambda^2 V_1^2\biggr).
\ee
\end{Corollary}

\begin{proof}
Take $V_2=V_3=\cdots=0$ in the partition function \eqref{eq:Gauss-Gk},
we have
\ben
\widetilde Z(V_0,V_1)=
\frac{1}{\sqrt{2\pi\lambda^2}}\int
\exp\biggl(
-\half \lambda^{-2} x^2+V_0+V_1 x
\biggr)dx
=\exp\biggl(
V_0+\half\lambda^2 V_1^2\biggr).
\een
Then the conclusion follows from \eqref{flow}.
\end{proof}

We now expand the exponentials on the right-hand side of \eqref{Z-V1} to get:
\begin{equation*}
\begin{split}
\frac{\widetilde Z}{e^{V_0}}  = & \exp\biggl(
\sum_{k=2}^{\infty}\frac{1}{k!}V_k\frac{d^k}{dV_1 ^k}
\biggr)\exp\biggl(\half \lambda^2 V_1^2\biggr) \\
 = & \sum_{n \geq 0} \sum_{\sum\limits_{k\geq 2} km_k = n}
 \prod_{k\geq 2} \frac{V_k^{m_k}}{m_k! (k!)^{m_k}}
\cdot \frac{d^n}{d V_1^n} \sum_{m_1\geq 0}
\frac{\lambda^{2m_1}}{2^{m_1}m_1!} V_1^{2m_1} \\
 = & \sum_{n \geq 0} \sum_{\sum\limits_{k\geq 2} km_k = n}
 \prod_{k\geq 2} \frac{V_k^{m_k}}{m_k! (k!)^{m_k}}
\bigg(  \sum_{m_1\geq 0} \frac{(2m_1)\cdots (2m_1-n+1)
\lambda^{2m_1}}{2^{m_1}m_1!} V_1^{2m_1-n}\bigg).
\end{split}
\end{equation*}
Using selection rule for correlators
one can check that this matches with \eqref{eqn:Young}.

\subsection{Operator formalism for computing $G_k(z)$}
\label{sec:Operator-G}

In last subsection we have given an operator formulation for computing the partition
function $\widetilde{Z}$.
In this subsection we present an operator formalism for $G_k(z)$.
The basic tools we need are the Virasoro constraints for
the partition function $\widetilde Z$.
We will see that the Virasoro constraints in this case enable us
to obtain $G_l$ from $G_{l-1}$ directly by applying an operator $\widehat A$.

Combining the flow equations \eqref{flow} and polymer equation \eqref{polymer},
we can get the puncture equation:
\ben
\biggl(V_1+\sum_{k\geq 1}(V_{k+1}-\lambda^{-2}\delta_{k,1})\frac{\pd}{\pd V_k}
\biggr)\widetilde Z=0.
\een
Applying $\big(\frac{\pd}{\pd V_1}\big)^{m+1}$
and using the flow equations again,
we have:
\be
\biggl((m+1)\frac{\pd^m}{\pd V_1^m}
+\sum_{k\geq 0}\frac{1}{k!}(V_{k+1}-\lambda^{-2}\delta_{k,1})
\frac{\pd^{m+k+1}}{\pd V_1^{m+k+1}}\biggr)\widetilde Z=0,
\ee
or equivalently,
\be
\biggl((m+1)!\frac{\pd}{\pd V_m}
+\sum_{k\geq 0}\frac{(m+k+1)!}{k!}(V_{k+1}-\lambda^{-2}\delta_{k,1})
\frac{\pd}{\pd V_{m+k+1}}\biggr)\widetilde Z=0.
\ee
This gives us the Virasoro constraints (see \cite[Theorem 6.1]{zhou1}):

\begin{Theorem}
For every $m\geq -1$, we have $L_m\widetilde Z=0$, where
\ben
&& L_{-1}=V_1+\sum_{k\geq 1}(V_{k+1}-\lambda^{-2}\delta_{k,1})\frac{\pd}{\pd V_k};\\
&& L_0=1+\sum_{k\geq 0}(k+1)(V_{k+1}-\lambda^{-2}\delta_{k,1})
\frac{\pd}{\pd V_{k+1}};\\
&& L_m=(m+1)!\frac{\pd}{\pd V_m}
+\sum_{k\geq 0}\frac{(m+k+1)!}{k!}(V_{k+1}-\lambda^{-2}\delta_{k,1})
\frac{\pd}{\pd V_{m+k+1}},\quad m\geq 1.
\een
These operators satisfy the Virasoro commutation relation:
\ben
[L_m,L_n]=(m-n)L_{m+n}, \qquad \forall m,n\geq -1.
\een
\end{Theorem}
This theorem can be checked by a direct computation.

We can also write the Virasoro constraints in terms of the free energy
$\log\widetilde{Z}=\sum\limits_{k}\lambda^{2k}G_k$.
By doing this, we obtain:

\begin{Proposition}
For every $l\geq 2$ and $m\geq -1$, we have:
\be
(m+2)\frac{\pd}{\pd V_{m+2}}G_l=
\sum_{k\geq 0}\binom{m+k+1}{k}V_{k+1}\frac{\pd}{\pd V_{m+k+1}}G_{l-1}
+\frac{\pd}{\pd V_m}G_{l-1}.
\ee
In particular, for $m=-1$, we have
\be
\frac{\pd}{\pd V_1}G_l= d_V G_{l-1}.
\ee
Here we use that convention $\frac{\pd}{\pd V_{-1}} := 0$.
\end{Proposition}

Therefore,
for every $l\geq 1$ and $m\geq 1$, we have
\be\label{virasoro}
m\frac{\pd}{\pd V_m}G_l=\biggl(
\sum_{k\geq 0}\binom{m+k-1}{k}V_{k+1}\frac{\pd}{\pd V_{m+k-1}}
+\frac{\pd}{\pd V_{m-2}}\biggr)G_{l-1}.
\ee

\begin{Theorem} \label{thm:Gk-2}
Define an operator
\be
\label{eq-def-tilde-A}
\widehat A:=\sum_{m\geq 1}V_m\biggl(
\sum_{k\geq 0}\binom{m+k-1}{k}V_{k+1}\frac{\pd}{\pd V_{m+k-1}}
+\frac{\pd}{\pd V_{m-2}}\biggr),
\ee
where we use that convention $\frac{\pd}{\pd V_{-1}} := 0$.
Then we have:
\be \label{rec-G-opr}
G_k=\frac{1}{2^k\cdot k!}\widehat{A}^k G_0, \qquad k\geq 1,
\ee
and the free energy of $\widetilde{Z}$ is
\ben
\log\widetilde Z=\sum_{k\geq 0}\lambda^{2k}G_k
= e^{\half \lambda^2\widehat{A}} (G_0).
\een
\end{Theorem}

\begin{proof}
By the homogeneity condition given in Corollary \ref{cor-chi-g,0}
and the relation \eqref{virasoro},
we have:
\ben
2k\cdot G_k=
\sum_{m\geq 1}m V_m\frac{\pd}{\pd V_m}G_k=
\widehat A G_{k-1}.
\een
Then the conclusion is clear.
\end{proof}

Notice that
although the operator $\widehat A$ in the above theorem is a summation of
an infinite numbers of terms,
one only needs to evaluate a finite number of terms while
computing $G_l$ by applying $\widehat A$ to $G_{l-1}$
since $G_{l-1}$ is a polynomial in $\{V_k\}_{1\leq k\leq 2l-2}$.

\begin{Remark}
This operator formalism is inspired by
the work \cite{al} of Alexandrov on a similar result
for the Witten-Kontsevich tau function.
See also \cite{zhou2} for the result in the case of $r$-spin curves.
\end{Remark}

\subsection{The orbifold Euler characteristics of $\Mbar_{g,0}$}
\label{sec-orb-genus0}

Now we have derived three types of recursions
\eqref{rec-G}, \eqref{rec-G-2} and \eqref{rec-G-opr}
for the generating series $G_k(z)$ of the coefficients $\{a_{g,0}^k\}$.
From \eqref{rec-G-2} and \eqref{rec-G-opr} we easily see that $G_k$ are polynomials in $V_l$,
while from \eqref{rec-G} we can not see this polynomiality.
Now let us summarize the formulas for
the orbifold Euler characteristics $\chi(\Mbar_{g,0})$.

From the definition of $G_k(z)$ and $\{a_{g,0}^k\}$,
one can see:
\be
\chi(\Mbar_{g,0}) = \biggl[\sum_{k=0}^{3g-3} G_k(z) \biggr]_{z^{2-2g}}
= \biggl[\sum_{k=0}^{\infty} G_k(z) \biggr]_{z^{2-2g}},
\ee
where $[\cdot]_{z^{2-2g}}$ means the coefficient of $z^{2-2g}$.
Now by taking $\lambda=1$ in Theorem \ref{thm:Gk} and \ref{thm:Gk-2},
we obtain the following formulas for $\chi(\Mbar_{g,0})$:

\begin{Theorem}
We have:
\be
\sum_{g=2}^\infty
\chi(\Mbar_{g,0})z^{2-2g}
=\log \biggl( e^{\half \tilde{d}_V^2} e^{G_0(z)} \biggr)
=e^{\half \widehat{A}}G_0(z),
\ee
or equivalently,
\be
\chi(\Mbar_{g,0})=
\bigg[
\log \biggl( e^{\half \tilde{d}_V^2} e^{G_0(z)} \biggr)
\bigg]_{z^{2-2g}}
=\bigg[
e^{\half \widehat{A}}G_0(z)
\bigg]_{z^{2-2g}},
\ee
where $\tilde{d}_V$ and $\widehat{A}$ are given by
\eqref{eq-def-tilde-d} and \eqref{eq-def-tilde-A} respectively.
\end{Theorem}

\section{Results for $\chi_{g,n}(t,\kappa)$ and $\chi(\Mbar_{g,n})$}
\label{sec5}

In \S \ref{sec:Linear} we have showed that given initial values
$\{a_{0,3}^k\}$, $\{a_{1,1}^k\}$ and $\{a_{g,0}^k\}$ ($g\geq 2$),
all the coefficients $\{a_{g,n}^k\}$ are computed by the
linear recursion \eqref{eq-linear-general};
and the structures of the initial values
$\{a_{0,3}^k\}$, $\{a_{1,1}^k\}$ and $\{a_{g,0}^k\}$ ($g\geq 2$)
have been discussed in \S \ref{sec4}.
Now in the present section,
let us solve the linear recursion \eqref{eq-linear-general} directly
and give explicit formulas for the solutions.
This completes our calculations of $\{a_{g,n}^k\}$
and thus of $\chi_{g,n}(t,\kappa)$ and $\chi(\Mbar_{g,n})$.
We will first solve the genus zero case,
and it turns out that the results in higher genera exhibit similar patterns.

\subsection{Explicit expressions for generating series of $a_{0,n}^k$ and $\chi(\Mbar_{0,n})$}
\label{sol-genus0}

In this subsection,
we present an explicit expression for
the generating series of $\chi(\Mbar_{0,n})$ $(n\geq 3)$
by solving the linear recursion \eqref{eq-linear-general}
of the coefficients $\{a_{0,n}^k\}$ at genus zero.

Take $g=0$.
Recall that the refined orbifold Euler characteristic $\chi_{0,n}(t,\kappa)$
is of the form (Proposition \ref{prop-chi-homog}):
\ben
\chi_{0,n}(t,\kappa)=
t^{2-n}\cdot \widetilde\chi_{0,n}(\kappa)
=t^{2-n}\cdot\sum_{i=0}^{n-3}a_{0,n}^i \kappa^i,
\qquad n\geq 3,
\een
and the orbifold Euler characteristic of $\Mbar_{0,n}$ is given by:
\ben
\chi(\Mbar_{0,n})=n!\cdot \chi_{0,n}(1,1)=
n!\cdot\sum_{i=0}^{n-3}a_{0,n}^i.
\een

Now let us consider the following generating series
of the coefficients $\{a_{0,n}^i\}$:
\be
A_i(x):=\sum_{n=3}^{\infty}a_{0,n}^i\cdot x^n,
\qquad i\geq 0,
\ee
where we formally denote $a_{0,n}^i:=0$ for $i>n-3$.
Then it is not hard to see that
the generating series of the orbifold Euler characteristics of $\Mbar_{0,n}$
is given by:
\be
\sum_{n=3}^\infty \chi(\Mbar_{0,n})\cdot\frac{x^n}{n!}=
\sum_{k=0}^{\infty}A_k(x),
\ee
and we have
\be
\sum_{i=0}^{n-3}a_{0,n}^i
=
\biggl[\sum_{i=0}^{n-3}A_i(x)
\biggr]_n,
\ee
where $[\cdot]_n$ means taking the coefficient of $x^n$.

Our main result in this subsection is the following:

\begin{Theorem}\label{genus-0-expl}
For $n\geq 3$, we have:
\be
\chi(\Mbar_{0,n})=
n!\cdot
\biggl[\sum_{k=0}^{n-3}A_k(x)
\biggr]_n,
\ee
where $[\cdot]_n$ means the coefficient of $x^n$.
The functions $A_k(x)$ are given by:
\begin{equation*}
\begin{split}
&A_0=\frac{1}{2}(1+x)^2\log(1+x)-\frac{1}{2}x-\frac{3}{4}x^2, \\
&A_1 = \half + (1+x) (\log(1+x)-1) + \half (1+x)^2 (\log(1+x)-1)^2,
\end{split}
\end{equation*}
and for $k\geq 2$,
\begin{equation*}
A_k=\sum_{m=2-k}^{2}\frac{(x+1)^m}{(m+2)!}\sum_{l=0}^{k-m-1}
\frac{\big(\log(x+1)-1\big)^l}{l!}
e_{l+m}(1-m,2-m,\cdots,k-m-1),
\end{equation*}
where $e_l$ are the elementary symmetric polynomials
\be
\label{eq-def-ele-sym}
e_l(x_1,x_2,\cdots, x_n)=\sum_{1\leq j_1<\cdots<j_l\leq n}x_{j_1}\cdots x_{j_l}
\ee
for $l\geq 0$,
and we use the convention $e_{l}:=0$ for $l<0$.
\end{Theorem}

Before giving a proof to the above theorem,
let us first do some explicit calculations of $A_k$
for small $k$.
From the Harer-Zagier formula \eqref{eq-harer-zagier} one knows that
$A_0(x)$ is given by:
\be
\label{eq-initial-A0}
\begin{split}
A_0(x)=&
\sum_{n=3}^\infty \frac{1}{n!}\chi(\cM_{0,n})\cdot x^n\\
=&\sum_{n=3}^\infty
(-1)^{n+1}B_0\cdot (n-3)!\cdot \frac{x^n}{n!}\\
=&\frac{1}{2}(1+x)^2\log(1+x)-\frac{1}{2}x-\frac{3}{4}x^2.
\end{split}
\ee

Now let us derive a recursion relation for $\{A_k(x)\}$.
The linear recursion \eqref{genus0-linear} can be converted into the following:
\be
\label{rec-genus0-x}
\frac{d}{dx} A_k = 2 A_k - x \frac{d}{dx} A_k
+ (x\frac{d}{dx} + k-1) A_{k-1}.
\ee
This is an ordinary differential equation of first order,
thus a unique solution $A_k(x)$ is determined by this equation
together with the initial value $A_k(0)=0$.
Solving this ODE simply gives us:
\be
A_k(x)=(1+x)^2\cdot\int_0^x\biggl(
x\frac{d}{dx}A_{k-1}+(k-1)A_{k-1}
\biggr)\frac{dx}{(1+x)^3}.
\ee
Then all $A_k(x)$ are recursively computed using this recursion
and the initial value $A_0(x)$ given by \eqref{eq-initial-A0}.

In what follows, we will make a tricky change of variable:
\be
\label{eq-change-x-s}
x=e^{s+1}-1,
\ee
and the recursion \eqref{rec-genus0-x} in this new variable $s$ is:
\be\label{eq-rec-s}
\frac{d}{ds}A_k = 2A_k + \big((1-e^{-s-1}) \frac{d}{ds} + (k-1) \big)A_{k-1},
\ee
or after solving the ODE,
\be
\label{eq-rec-A-s}
A_k(s)=e^{2(s+1)}\cdot\int_{-1}^s\biggl(
\big(1-e^{-(s+1)}\big)\frac{d}{ds}A_{k-1}+(k-1)A_{k-1}
\biggr)e^{-2(s+1)}ds.
\ee
Here are examples of first a few $A_k$ computed using
\eqref{eq-initial-A0} and \eqref{eq-rec-A-s}:
\begin{equation*}
\begin{split}
A_0=&-\frac{1}{4}+e^{s+1}+e^{2(s+1)}(-\frac{1}{4}+\frac{1}{2}s),\\
A_1=& \half + e^{s+1}s + \half e^{2(s+1)} s^2,\\
A_2=&\half (s+1) + e^{s+1}(s^2+s) + e^{2(s+1)}(\half s^2+\half s^3),\\
A_3=&\frac{1}{6}e^{-(s+1)}+(\frac{1}{2}+\frac{3}{2}s+\frac{1}{2}s^2)
+e^{s+1}(s+\frac{5}{2}s^2+s^3)\\
&+e^{2(s+1)}(\frac{1}{2}s^2+\frac{7}{6}s^3+ \frac{1}{2}s^4),\\
A_4=&-\frac{1}{24}e^{-2(s+1)}+e^{-(s+1)}(\frac{1}{2}+\frac{1}{3}s)+
(\frac{1}{2}+3s+\frac{11}{4}s^2+\frac{1}{2}s^3)\\
&+e^{s+1}(s+\frac{9}{2}s^2+\frac{13}{3}s^3+s^4)
+e^{2(s+1)}(\frac{1}{2}s^2+2s^3+\frac{47}{24}s^4+\frac{1}{2}s^5).
\end{split}
\end{equation*}

Now apply the change of variables \eqref{eq-change-x-s}
to the statement in Theorem \ref{genus-0-expl}.
In order to prove that theorem,
we only need to prove the following:

\begin{Proposition}
$A_k$ is of the form:
\be
\label{genus0-sol}
A_k = e^{2(s+1)} a_{k,2} +e^{s+1} a_{k,1} + \cdots + e^{-(k-2)(s+1)} a_{k,-(k-2)},
\quad k\geq 1,
\ee
where $a_{k,j}$ are polynomials in $s$:
\ben
a_{k,-m}=\frac{1}{(m+2)!}
\sum_{l=0}^{k-m-1}\frac{s^l}{l!}e_{l+m}(-m+1,-m+2,\cdots,k-m-1),\quad k\geq 2.
\een
\end{Proposition}

\begin{proof}
By the uniqueness of the solution,
we only need to check that such an $A_k(s)$ satisfies the equation \eqref{eq-rec-s},
and the initial value is $A_k(s=-1)=0$.

First let us check the equation \eqref{eq-rec-s}.
It is equivalent to the following sequence of differential equations
in $a_{k,j}$:
\be
\label{equation-genus0}
\begin{cases}
a_{k,2}'=a_{k-1,2}'+(k+1)a_{k-1,2};\\
a_{k,1}'-a_{k,1}=
a_{k-1,1}'+k a_{k-1,1}-2a_{k-1,2}-a_{k-1,2}';\\
\qquad \cdots\cdots\\
a_{k,-k+3}'-(k-1)a_{k,-k+3}=
a_{k-1,-k+3}'+2 a_{k-1,j}\\
\qquad\qquad\qquad\qquad\qquad\qquad+(k-2)a_{k-1,-k+2}-a_{k-1,-k+2}';\\
a_{k,-k+2}'-ka_{k,-k+2}=(k-3)a_{k-1,-k+3}-a_{k-1,-k+3}'.
\end{cases}
\ee
Here the notation $a_{k,j}'$ means $\frac{d}{d s} a_{k,j}$.
These equations can be checked directly case by case.
For example, the first one holds since:
\ben
&&a_{k-1,2}'+(k+1)a_{k-1,2}\\
&=&\sum_{l=2}^{k}\frac{s^{l-1}}{(l-1)!}e_{l-2}(3,\cdots,k)+
(k+1)\sum_{l=2}^{k}\frac{s^{l}}{l!}e_{l-2}(3,\cdots,k)\\
&=&\sum_{l=2}^{k-1}\frac{s^{l}}{l!}\biggl[e_{l-1}(3,\cdots,k)+
(k+1)e_{l-2}(3,\cdots,k)\biggr]
+\frac{s^k}{k!}e_{k-2}(3,\cdots,k)+s\\
&=&\sum_{l=2}^{k-1}\frac{s^{l}}{l!}\cdot e_{l-1}(3,\cdots,k+1)
+\frac{s^k}{k!}e_{k-2}(3,\cdots,k)+s\\
&=&\sum_{l=2}^{k+1}\frac{s^{l-1}}{(l-1)!}e_{l-2}(3,\cdots,k+1)
=a_{k,2}'.
\een
Here we have used a combinatorial identity:
\ben
e_{m+1}(u_1,\cdots,u_n)+
u\cdot e_m(u_1,\cdots,u_n)=
e_{m+1}(u,u_1,\cdots,u_n).
\een
This identity follows directly from the definition of
the elementary symmetric polynomials \eqref{eq-def-ele-sym},
and we will use this identity quite often throughout the proofs in this section.

Now let us check the second equation in \eqref{equation-genus0}.
This equation holds because:
\begin{equation*}
\begin{split}
&a_{k-1,1}'+k a_{k-1,1}-2a_{k-1,2}-a_{k-1,2}'\\
=&\sum_{l=1}^{k-1}\frac{s^{l-1}}{(l-1)!}e_{l-1}(2,3,\cdots,k-1)
+k\sum_{l=1}^{k-1}\frac{s^l}{l!}e_{l-1}(2,3,\cdots,k-1)\\
&-2\sum_{l=2}^{k}\frac{s^l}{l!}e_{l-2}(3,4,\cdots,k)
-\sum_{l=2}^{k}\frac{s^{l-1}}{(l-1)!}e_{l-2}(3,4,\cdots,k)\\
=&\sum_{l=1}^{k-2}\frac{s^{l}}{l!}\biggl[
e_{l}(2,\cdots,k-1)+ke_{l-1}(2,\cdots,k-1)\biggr]
+\frac{ks^{k-1}}{(k-1)!}e_{k-2}(2,\cdots,k-1)+1\\
&-\frac{2s^k}{k!}e_{k-2}(3,\cdots,k)-s
-\sum_{l=2}^{k-1}\frac{s^{l}}{l!}\biggl[
2e_{l-2}(3,4,\cdots,k)+e_{l-1}(3,4,\cdots,k)\biggr]\\
=&\sum_{l=1}^{k-2}\frac{s^{l}}{l!}\cdot
e_{l}(2,\cdots,k)
+\frac{ks^{k-1}}{(k-1)!}e_{k-2}(2,\cdots,k-1)+1\\
&-\sum_{l=2}^{k-1}\frac{s^{l}}{l!}\cdot
e_{l-1}(2,3,\cdots,k)-
\frac{2s^k}{k!}e_{k-2}(3,\cdots,k)-s\\
=&\sum_{l=0}^{k-1}\frac{s^{l}}{l!}\cdot e_{l}(2,\cdots,k)
-\sum_{l=1}^{k}\frac{s^{l}}{l!}\cdot e_{l-1}(2,3,\cdots,k)
=a_{k,1}'-a_{k,1}.
\end{split}
\end{equation*}
The rest of the equations in \eqref{equation-genus0}
all hold for the same reason.
Therefore $A_k(s)$ defined by \eqref{genus0-sol} is indeed a solution
of the recursive equation \eqref{eq-rec-s}.

Now we only need to check $A_k(s=-1)=0$, i.e.,
\ben
&&\sum_{m=1}^{k-2}\biggl[\frac{1}{(m+2)!}
\sum_{l=0}^{k-m-2}\frac{(-1)^l}{l!}e_{l+m}(-m+1,-m+2,\cdots,k-m-1)\biggr]\\
&&+\frac{1}{2}\cdot\sum_{l=0}^{k-1}\frac{(-1)^l}{l!}e_{l}(1,2,\cdots,k-1)
+\sum_{l=1}^{k}\frac{(-1)^l}{l!}e_{l-1}(2,3,\cdots,k)\\
&&+\sum_{l=2}^{k+1}\frac{(-1)^l}{l!}e_{l-2}(3,4,\cdots,k+1)
=0.
\een
This is equivalent to
\ben
&&\sum_{l=1}^{k-2}\sum_{m=1}^{l} \biggl[\frac{1}{(m+2)!}
 \frac{(-1)^{l-m}}{(l-m)!}e_l(-m+1,-m+2,\cdots,k-m-1)\biggr]\\
&&+\frac{1}{2}\cdot\sum_{l=0}^{k-1}\frac{(-1)^l}{l!}e_{l}(1,2,\cdots,k-1)
-\sum_{l=0}^{k-1}\frac{(-1)^{l}}{(l+1)!}e_{l}(2,3,\cdots,k)\\
&&+\sum_{l=0}^{k-1}\frac{(-1)^l}{(l+2)!}e_{l}(3,4,\cdots,k+1)
=0.
\een
Since we have
\ben
\frac{1}{2}\cdot\frac{(-1)^0}{0!}- \frac{(-1)^{0}}{1!} +
\frac{(-1)^0}{2!} = 0
\een
for $l=0$, and
\ben
&& \frac{(-1)^{k-1}}{(k+1)!}e_{k-1}(3,4,\cdots,k+1)-
\frac{(-1)^{k-1}}{k!}e_{k-1}(2,3,\cdots,k) \\
& + & \frac{1}{2}\cdot\frac{(-1)^{k-1}}{(k-1)!}e_{k-1}(1,2,\cdots,k-1)\\
& = & \frac{(-1)^{k-1}}{(k+1)!}\frac{(k+1)!}{2!}-
\frac{(-1)^{k-1}}{k!}\cdot k!
+ \frac{1}{2}\cdot\frac{(-1)^{k-1}}{(k-1)!} \cdot (k-1)!\\
& = & 0
\een
for $l=k-1$, it now suffices to show that
\begin{equation*}
\begin{split}
&\frac{(-1)^l}{(l+2)!}e_{l}(3,4,\cdots,k+1)-
\frac{(-1)^{l}}{(l+1)!}e_{l}(2,3,\cdots,k)+
\frac{1}{2}\cdot\frac{(-1)^l}{l!}e_{l}(1,2,\cdots,k-1)\\
&+\sum_{m=1}^{l}  \frac{1}{(m+2)!}
 \frac{(-1)^{l-m}}{(l-m)!}e_l(-m+1,-m+2,\cdots,k-m-1)
=0
\end{split}
\end{equation*}
for $l=1, \dots, k-2$. Or equivalently,
\be\label{genus0-claim}
\sum_{j=0}^{l+2} (-1)^j \binom{l+2}{j}
e_{l}(1+(-j+2),2+(-j+2),\cdots,k-1+(-j+2))
=0.
\ee

Now let us prove \eqref{genus0-claim}.
Notice that the integers $e_{l}(1,2,\cdots,k-1)$
are the Stirling numbers of the first kind $\brac{k}{k-l}$.
We have:
\begin{equation*}
\begin{split}
&e_{l}(1+(-j+2),2+(-j+2),\cdots,k-1+(-j+2))\\
=&\sum_{1\leq a_1\leq\cdots\leq a_l\leq k-1}
\big(a_1+(-j+2)\big)\cdots\big(a_l+(-j+2)\big)\\
=&\sum_{1\leq a_1\leq\cdots\leq a_l\leq k-1}
\biggl(\sum_{m=0}^l e_m(a_1,\cdots,a_l)(-j+2)^{l-m}\biggr)\\
=&\sum_{m=0}^l \binom{k-1-m}{l-m}e_m(1,2,\cdots,k-1)(-j+2)^{l-m}\\
=&\sum_{m=0}^l \binom{k-1-m}{l-m}\brac{k}{k-m}(-j+2)^{l-m},
\end{split}
\end{equation*}
thus equation \eqref{genus0-claim} can be rewritten as
\be\label{genus0-claim2}
\begin{split}
0&=\sum_{j=0}^{l+2}(-1)^j\binom{l+2}{j}
\sum_{m=0}^l\binom{k-1-m}{l-m}\brac{k}{k-m}(-j+2)^{l-m}\\
&=\sum_{m=0}^l\biggl(
\sum_{j=0}^{l+2}(-1)^j\binom{l+2}{j}(-j+2)^{l-m}
\biggr)\cdot\binom{k-1-m}{l-m}\brac{k}{k-m}.
\end{split}
\ee

Now it suffices to prove \eqref{genus0-claim2}.
Applying the operator $(-x\frac{d}{dx})^{l-m}$ to the identity
\ben
x^{-2}(1-x)^{l+2}=\sum_{j=0}^{l+2}(-1)^j\binom{l+2}{j}x^{j-2}
\een
and then taking $x=1$, we may get:
\ben
\sum_{j=0}^{l+2}(-1)^j\binom{l+2}{j}(-j+2)^{l-m}=0.
\een
This proves \eqref{genus0-claim2},
thus $A(s=-1)=0$ indeed holds,
and we have finished the proof.
\end{proof}

\subsection{Explicit expressions for generating series of
$\{a_{1,n}^k\}$ and $\chi(\Mbar_{1,n})$}

In this subsection we present explicit expressions
for the generating series of
$\{a_{1,n}^k\}$ and $\chi(1,n)$
by solving the linear recursion \eqref{eq-linear-general} at genus one.
The result is similar to the case of genus zero.

The refined orbifold Euler characteristic $\chi_{1,n}(t,\kappa)$
is given by:
\ben
\chi_{1,n}(t,\kappa)=
t^{-n}\cdot \widetilde\chi_{1,n}(\kappa)
=t^{-n}\cdot\sum_{i=0}^n a_{1,n}^i\kappa^i,
\een
and the orbifold Euler characteristic $\chi(\Mbar_{1,n})$
is given by:
\ben
\chi(\Mbar_{1,n})=n! \cdot \chi_{1,n}(1,1)
=n!\cdot\sum_{k=0}^n a_{1,n}^k.
\een

Now for $g=1$, define the following generating series of $\{a_{1,n}^k\}$:
\ben
B_k(x):=\sum_{n=1}^\infty a_{1,n}^k x^n,
\een
then the linear recursion \eqref{genus1-linear}
gives us the following recursion for $B_k(x)$:
\begin{equation*}
\begin{cases}
\frac{d}{dx}(B_k-\delta_{k,1}\cdot\frac{x}{2})+x\frac{d}{dx}B_k
=x\frac{d}{dx}B_{k-1}+(k-1)B_{k-1},\\
B_k(0)=0,
\end{cases}
\end{equation*}
and by \eqref{eq-harer-zagier} $B_0(x)$ is given by:
\begin{equation*}
B_0(x)=
\sum_{n=1}^\infty \frac{1}{n!}\chi(\cM_{1,n})\cdot x^n
=-\frac{1}{12} \log (1+x).
\end{equation*}

Recall that all the coefficients $\{a_{1,n}^k\}$ are determined by
\ben
\widetilde\chi_{1,1}(\kappa)=-\frac{1}{12}+\frac{1}{2}\kappa,
\een
i.e.,
by two numbers $a_{1,1}^0=-\frac{1}{12}$ and $a_{1,1}^1=\frac{1}{2}$ using
\eqref{genus1-linear}
Now since the recursion is linear,
here we can separate this initial data $\widetilde\chi_{1,1}$ into two parts
in the following way.
First,
we replace by the initial data $\widetilde\chi_{1,1}$ by
$1=1+0\cdot \kappa$ and run the recursion,
and the generating series $B_k(x)$ will be replaced by new generating series $C_k(x)$;
and then we replace by the initial data $\widetilde\chi_{1,1}$ by
$\kappa=0+1\cdot \kappa$,
and obtain new generating series $D_k(x)$ similary.
The it is clear that $B_k=-\frac{1}{12}C_k+\frac{1}{2}D_k$,
where $C_k(x)$ and $D_k(x)$ are determined by:
\begin{equation}
\label{eq-C-D-conditions}
\begin{split}
&\begin{cases}
\frac{d}{dx}C_k+x\frac{d}{dx}C_k
=x\frac{d}{dx}C_{k-1}+(k-1)C_{k-1},\\
C_0(x)=\log(1+x);  \qquad C_k(0)=0;\\
\end{cases}\\
&\begin{cases}
\frac{d}{dx}D_k+x\frac{d}{dx}D_k
=x\frac{d}{dx}D_{k-1}+(k-1)D_{k-1},\\
D_0(x)=0,\qquad D_1(x)=\log(1+x); \qquad D_k(0)=0.
\end{cases}
\end{split}
\end{equation}
Now it is not hard to see that
the generating series of the orbifold Euler characteristics of $\Mbar_{1,n}$
is given by:
\be
\sum_{n=1}^\infty
\chi(\Mbar_{1,n})\cdot \frac{x^n}{n!}
=\sum_{k=0}^{\infty}\biggl(
-\frac{1}{12}C_k(x)
+\half D_k(x)\biggr).
\ee

Our main result in this subsection is the following:

\begin{Theorem}
For every $n\geq 1$, we have:
\ben
\chi(\Mbar_{1,n})=n!\cdot\biggl[\sum_{k=0}^{n}\biggl(
-\frac{1}{12}C_k(x)
+\half D_k(x)\biggr)\biggr]_n,
\een
where $[\cdot]_n$ means the coefficient of $x^n$.
Let $s:=\log(x+1)-1$, then
the explicit formulas for $C_k$ are given by:
\begin{equation*}
\begin{split}
&C_0=s+1;\\
&C_k=c_{k,-k}e^{-k(s+1)}+c_{k,-k+1}e^{(-k+1)(s+1)}+\cdots+c_{k,-1}e^{-(s+1)}+c_{k,0},
\quad k\geq 1,
\end{split}
\end{equation*}
where
\begin{equation*}
\begin{split}
&c_{k,0}=\sum\limits_{l=1}^{k}\frac{s^l}{l!}e_{l-1}(1,2,\cdots,k-1),
\quad k\geq 1;\\
&c_{k,-m}=\frac{1}{m!}\sum\limits_{l=0}^{k-m-1}
\frac{s^l}{l!}e_{l}(-m+1,-m+2,\cdots,k-m-1),
 \quad m>0, k\geq m+1.
\end{split}
\end{equation*}
And the explicit formulas for $D_k$ are given by:
\begin{equation*}
\begin{split}
&D_0=0,\\
&D_1=s+1;\\
&D_k=d_{k,-k+1}e^{(-k+1)(s+1)}+d_{k,-k+2}e^{(-k+2)(s+1)}+\cdots\\
&\qquad\qquad +d_{k,-1}e^{-(s+1)}+d_{k,0},
\quad k\geq 2,
\end{split}
\end{equation*}
where
\begin{equation*}
\begin{split}
&d_{k,0}=\frac{1}{k}+ks+
\sum\limits_{l=2}^{k}\biggl(\frac{s^l}{l!}\sum\limits_{j=1}^{k-l+1}
j^2\cdot e_{l-2}(j+1,j+2,\cdots,k-1)\biggr),\\
&d_{k,-1}=(k-1)+
\sum\limits_{l=1}^{k-2}\biggl(\frac{s^l}{l!}\sum\limits_{j=0}^{k-l}
j^2\cdot e_{l-1}(j+1,j+2,\cdots,k-2)\biggr),\\
&d_{k,-m}=(-1)^{m+1}\frac{1}{m}\cdot(k-m)\\
&\qquad\qquad+\frac{(-1)^{m+1}}{m}\sum\limits_{l=0}^{k-m-1}\frac{s^l}{l!}\biggl[
\sum\limits_{h=0}^{m-1}(-1)^h e_h(1,\frac{1}{2},\frac{1}{3},\cdots,\frac{1}{m-1})\\
&\qquad\qquad\times\sum\limits_{j=1}^{k-l-m-h}j^2 e_{l-1+h}(j+1,\cdots,k-m-1)
\biggr],
\qquad m\geq 2.
\end{split}
\end{equation*}
\end{Theorem}

These functions $d_{k,-m}$ can be rewritten in the following way:
\begin{equation*}
d_{k,-m}=(-1)^{m+1}\frac{1}{m}(k-m)+
\frac{(-1)^{m+1}}{m}\tilde d_{n,-m}(s),\quad m\geq 2,
\end{equation*}
where the functions $\tilde d_{n,-m}(s)$ are given by:
\begin{equation*}
\begin{split}
&\tilde d_{k,-2}=\sum\limits_{l=0}^{k-3}\frac{s^l}{l!}\biggl(
\sum\limits_{j=1}^{k-l-2}j^2e_{l-1}(j+1,\cdots,k-3)
-\sum\limits_{j=1}^{k-l-3}j^2e_{l}(j+1,\cdots,k-3)
\biggr),\\
&\tilde d_{k,-m}=\tilde d_{k-1,-m+1}-\frac{1}{m-1}\frac{d}{ds}\tilde d_{k-1,-m+1},
\qquad m\geq 3.
\end{split}
\end{equation*}
For example,
\begin{equation*}
\begin{split}
&\tilde d_{k,-3}=
\sum_{l=0}^{k-4}\frac{s^l}{l!}\biggl[\biggl(
\sum_{j=1}^{k-l-3}j^2e_{l-1}(j+1,\cdots,k-4)-\sum_{j=1}^{k-l-4}j^2e_{l}(j+1,\cdots,k-4)
\biggr)\\
&\qquad-\frac{1}{2}\biggl(
\sum_{j=1}^{k-l-4}j^2e_{l}(j+1,\cdots,k-4)-\sum_{j=1}^{k-l-5}j^2e_{l+1}(j+1,\cdots,k-4)
\biggr)\biggr].
\end{split}
\end{equation*}

\begin{equation*}
\begin{split}
&\tilde d_{k,-4}=\sum_{l=0}^{k-5}
\frac{s^l}{l!}\biggl\{
\biggl[\biggl(
\sum_{j=1}^{k-l-4}j^2e_{l-1}(j+1,\cdots,k-5)-\sum_{j=1}^{k-l-5}j^2e_{l}(j+1,\cdots,k-5)
\biggr)\\
&\qquad-\frac{1}{2}\biggl(
\sum_{j=1}^{k-l-5}j^2e_{l}(j+1,\cdots,k-5)-\sum_{j=1}^{k-l-6}j^2e_{l+1}(j+1,\cdots,k-5)
\biggr)\biggr]
\\
&\quad-\frac{1}{3}\biggl[\biggl(
\sum_{j=1}^{k-l-5}j^2e_{l}(j+1,\cdots,k-5)-\sum_{j=1}^{k-l-6}j^2e_{l+1}(j+1,\cdots,k-5)
\biggr)\\
&\qquad-\frac{1}{2}\biggl(
\sum_{j=1}^{k-l-6}j^2e_{l+1}(j+1,\cdots,k-5)-\sum_{j=1}^{k-l-7}j^2e_{l+2}(j+1,\cdots,k-5)
\biggr)\biggr]\biggr\}.
\end{split}
\end{equation*}

We will omit the proof of this theorem for genus one,
since it can be proved using the same method as the case of genus zero
given in the previous subsection.
One only needs to check that $C_k$ and $D_k$ given in this theorem
satisfy the recursions and initial conditions in \eqref{eq-C-D-conditions}.
See also the next subsection
where we will present a similar proof in a general case for $g\geq 2$.

\subsection{Solutions of the linear recursion in general case $g\geq 2$}

Similar to the cases of genus zero and genus one,
in the general case $g\geq 2$ one can also solve the linear recursion
\eqref{eq-linear-general} explicitly to compute the generating series
of $\chi(\Mbar_{g,n})$.
Let us do this in the present subsection.

For $g \geq 2$,
recall that $\chi_{g,n}(t,\kappa)=t^{2-2g-n}\cdot \widetilde\chi_{g,n}(\kappa)$
and
\ben
\widetilde\chi_{g,0}(\kappa)=a_{g,0}^0+a_{g,0}^1\kappa
\cdots+a_{g,0}^{3g-3}\kappa^{3g-3}.
\een
The coefficients $a_{g,0}^{i}$ ($0\leq i\leq 3g-3$) are the initial data
for the linear recursion,
and the structures of these initial data have been discussed in \S \ref{sec4}.
Similar to the case $g=1$,
we decompose $A_{g,k}(x):=\sum\limits_{n=0}^{\infty}a_{g,n}^k x^n$
into the following summation:
\ben
A_{g,k}(x)=a_{g,0}^0A_{g,k}^0(x)+a_{g,0}^1A_{g,k}^1(x)
\cdots+a_{g,0}^{3g-3}A_{g,k}^{3g-3}(x),
\een
where the sequence $\{A_k^p(x)\}_{k\geq 0}$ is the solution of the linear recursion
if we replace the initial data $\widetilde\chi_{g,0}(\kappa)$ by $\kappa^p$
($0\leq p\leq 3g-3$).
Then from \eqref{eq-linear-general} we know that
these sequences are determined by:
\be
\begin{cases}
A_{g,0}^p(x)=\cdots=A_{g,p-1}^{p}(x)=0;\\
A_{g,p}^p(x)=\sum\limits_{n=0}^\infty (-1)^n \cdot
\frac{(2g-3+n)!}{n!\cdot(2g-3)!} \cdot x^n
=(1+x)^{2-2g};\\
A_{g,k}^p(0)=0,\qquad k>p;\\
\frac{d}{dx}A_{g,k}^p+x\frac{d}{dx}A_{g,k}^p+(2g-2)A_{g,k}^p=
x\frac{d}{dx}A_{g,k-1}^p+(k-1)A_{g,k-1}^p.
\end{cases}
\ee
Recall that the orbifold Euler characteristic of $\Mbar_{g,n}$ is given by
\ben
\chi(\Mbar_{g,n})=n!\cdot \chi_{g,n}(1,1)=
n!\cdot \sum_{p=0}^{3g-3+n} a_{g,n}^{p},
\een
therefore the generating series of $\chi(\Mbar_{g,n})$ is given by:
\be
\sum_{n=0}^\infty \chi(\Mbar_{g,n})\cdot \frac{x^n}{n!}=
\sum_{p=0}^{3g-3} a_{g,0}^p
\sum_{k=0}^\infty A_{g,k}^p (x).
\ee

Our main result in this subsection is the following:

\begin{Theorem}
The orbifold Euler characteristic $\chi(\Mbar_{g,n})$
is given by:
\ben
\chi(\Mbar_{g,n})=n!\cdot\sum_{k=0}^{3g-3+n}a_{g,n}^k=n!\cdot
\biggl[\sum_{k=0}^{3g-3+n}\sum_{p=0}^{3g-3}a_{g,0}^p A_{g,k}^p(x)\biggr]_n,
\een
where $[\cdot]_n$ means the coefficient of $x^n$.
Let $x=e^{s+1}-1$, then
$A_{g,k}^p$ ($g\geq 2$) are given by:
\begin{equation*}
\begin{split}
&A_{g,0}^p=\cdots=A_{g,p-1}^{p}=0;\\
&A_{g,p}^p=e^{(2-2g)(s+1)};\\
&A_{g,k}^p=a_{g,k,-k-2g+2+p}^p e^{(-k-2g+2+p)(s+1)}
+\cdots+a_{g,k,-2g+2}^p e^{(-2g+2)(s+1)},
\quad k>p,
\end{split}
\end{equation*}
where for $p=0$ we have:
\begin{equation*}
\begin{split}
&a_{g,k,-2g+2}^0=(2-2g)\sum\limits_{l=1}^{k}\frac{s^l}{l!}
e_{l-1}(-2g+3,-2g+4,\cdots,-2g+k+1),\\
&a_{g,k,-2g+2-m}^0=\frac{2-2g}{m!}
\sum\limits_{l=0}^{k-m}\frac{s^l}{l!}e_{l+m-1}(-2g+3-m,\cdots,
-2g+k+1-m), \quad m\geq 1.
\end{split}
\end{equation*}
For $p=1$, we have:
\begin{equation*}
a_{g,k,-2g+2-m}^1=\frac{1}{m!}\sum_{l=0}^{k-m-1}\frac{s^l}{l!}
e_{l+m}(-2g+3-m,\cdots,-2g+k+1-m).
\end{equation*}
For $p=2$, we have:
\begin{equation*}
\begin{split}
&a_{g,k,-2g+2}^2=k-1+\sum\limits_{l=1}^{k-2}\frac{s^l}{l!}\biggl(
\sum\limits_{j=-2g+4}^{k-l-2g+2}j(j+2g-3)e_{l-1}(j+1,\cdots,k-2g+1)
\biggr),\\
&a_{g,k,-2g+2-m}^2=(-1)^m\binom{2g-3+m}{m}\cdot(k-1-m)\\
&\qquad\qquad
+\frac{(-1)^m}{m!}\cdot\sum\limits_{l=0}^{k-2-m}\frac{s^l}{l!}
\biggl[\sum\limits_{h=0}^{m}(-1)^h
e_{m-h}(2g-2,\cdots,2g-3+m)\times\\
&\qquad\qquad
\sum\limits_{j=-2g+4}^{k-l-2g+2-m-h}j(j+2g-3)
e_{l-1+h}(j+1,\cdots,k-2g-m+1)
\biggr],
\quad m\geq 1.
\end{split}
\end{equation*}
And for $p\geq 3$, we have
\be\label{sol-general1}
\begin{split}
&a_{g,k,-2g+2-m}^p=\frac{(-1)^m}{m!}\sum_{l=0}^{k-p-m}\frac{s^l}{l!}\biggl[
\sum_{h=0}^{m}(-1)^h\cdot e_{m-h}(2g-2,\cdots,2g-3+m)\\
&\quad\qquad\times\biggl(
\sum_{j=-2g+p+1}^{k-l-2g+1-m-h}\binom{j+2g-3}{p-2}
e_{l+h}(j+1,\cdots,k-2g-m+1)\biggr)\biggr],
\end{split}
\ee
or equivalently,
\be\label{sol-general2}
\begin{split}
&a_{g,k,-2g+2-m}^p=\frac{1}{m!}\sum_{l=0}^{k-p-m}\frac{s^l}{l!}\biggl[
 \sum_{j=-2g+p+1}^{-2g+k+1-m-l}
\binom{j+2g-3}{p-2}\\
&\qquad\times e_{m+l} (-2g+3-m,\cdots,-2g+2;j+1,\cdots,-2g+k+1-m)\biggr].
\end{split}
\ee

\end{Theorem}

\begin{proof}
We will only prove the case $p\geq 3$.
The proofs for other cases are all similar.

The equivalence of \eqref{sol-general1} and \eqref{sol-general2} is clear.
By the uniqueness of solutions to first order ODE's,
we only need to check that such $A_{g,k}^p$ satisfy
\be\label{equation-general}
\frac{d}{dx}A_{g,k}^p+x\frac{d}{dx}A_{g,k}^p+(2g-2)A_{g,k}^p=
x\frac{d}{dx}A_{g,k-1}^p+(k-1)A_{g,k-1}^p
\ee
and the initial condition $A_{g,k}^p(s=-1)=0$ for $k>p$.

First let us check the equation \eqref{equation-general}. It is equivalent to
\be\label{equation-general-s}
\frac{d}{ds}A_{g,k}^p+(2g-2)A_{g,k}^p=
\big(1-e^{-(s+1)}\big)\frac{d}{ds}A_{g,k-1}^p+(k-1)A_{g,k-1}^p,
\ee
or equivalently,
\be
\begin{cases}
\frac{d}{ds}a_{g,k,-2g+2}^{p}=
\frac{d}{ds}a_{g,k-1,-2g+2}^p+(k-2g+1)a_{g,k-1,-2g+2}^p,\\
\frac{d}{ds}a_{g,k,-2g+1}^{p}-a_{g,k,-2g+1}^{p}=
(2g-2)a_{g,k-1,-2g+2}^{p}-\frac{d}{ds}a_{g,k-1,-2g+2}^{p}\\
\qquad\qquad\qquad\qquad\qquad
+\frac{d}{ds}a_{g,k-1,-2g+1}^p+(k-2g)a_{g,k-1,-2g+1}^p,\\
\quad\cdots\cdots
\end{cases}
\ee
Similar to the equations \eqref{equation-genus0} in the case $g=0$,
these equations for $g\geq 2$ can also be checked case by case
using the expression \eqref{sol-general1},
For example, the first equation holds because:
\ben
&&\frac{d}{ds}a_{g,k-1,-2g+2}^p+(k-2g+1)a_{g,k-1,-2g+2}^p\\
&=&\sum_{l=1}^{k-1-p}\frac{s^{l-1}}{(l-1)!}
\sum_{j=-2g+p+1}^{k-l-2g}\binom{j+2g-3}{p-2}e_{l}(j+1,\cdots,k-2g)\\
&&+(k-2g+1)\sum_{l=0}^{k-1-p}\frac{s^l}{l!}
\sum_{j=-2g+p+1}^{k-l-2g}\binom{j+2g-3}{p-2}e_{l}(j+1,\cdots,k-2g)\\
&=&\sum_{l=0}^{k-2-p}\biggl[\frac{s^{l}}{l!}
\sum_{j=-2g+p+1}^{k-l-2g}\binom{j+2g-3}{p-2}\biggl(e_{l+1}(j+1,\cdots,k-2g)\\
&&\qquad\qquad\qquad+(k-2g+1)e_l(j+1,\cdots,k-2g)
\biggr)\biggr]\\
&&+\frac{s^{k-1-p}}{(k-1-p)!}\cdot(k-2g+1)e_{k-1-p}(-2g+p+2,\cdots,k-2g)\\
&=&\sum_{l=0}^{k-1-p}\biggl[\frac{s^{l}}{l!}
\sum_{j=-2g+p+1}^{k-l-2g}\binom{j+2g-3}{p-2}e_{l+1}(j+1,\cdots,k-2g+1)\biggr]\\
&=&\frac{d}{ds}a_{g,k,-2g+2}^{p}.
\een
Now it suffices to prove the initial value conditions
$A_{g,k}^p(s=-1)=0$ for $k>p$, i.e.,
to prove:
\be
\begin{split}
&\sum_{m=0}^{k-p}\frac{1}{m!}\sum_{l=0}^{k-p-m}\frac{(-1)^l}{l!}\biggl[
 \sum_{j=-2g+p+1}^{-2g+k+1-m-l}
\binom{j+2g-3}{p-2}\\
&\times e_{m+l} (-2g+3-m,\cdots,-2g+2;j+1,\cdots,-2g+k+1-m)\biggr]=0.
\end{split}
\ee
This is equivalent to say:
\be\label{proof-general-eq}
\begin{split}
&\sum_{r=0}^{k-p}\frac{(-1)^r}{r!}
\sum_{m=0}^r(-1)^m\cdot \binom{r}{m}\biggl[
\sum_{j=-2g+p+1}^{-2g+k+1-r}
\binom{j+2g-3}{p-2}\\
&\qquad\times e_r (-2g+3-m,\cdots,-2g+2;j+1,\cdots,-2g+k+1-m)\biggr]=0.\\
\end{split}
\ee

To prove the equation \eqref{proof-general-eq},
we first claim the following combinatorial identity:
\be\label{proof-general-eqclaim}
\begin{split}
&\sum_{m=0}^r(-1)^m\cdot \binom{r}{m}\biggl[
\sum_{j=-2g+p+1}^{-2g+k+1-r}
\binom{j+2g-3}{p-2}\\
&\qquad\times e_r (-2g+3-m,\cdots,-2g+2;j+1,\cdots,-2g+k+1-m)\biggr]\\
&=\frac{(k-1)!}{(k-1-r)!}\biggl(
\binom{p-2}{p-2}+\binom{p-1}{p-2}+\cdots+\binom{k-2-r}{p-2}
\biggr),
\end{split}
\ee
for $k>p\geq 3$ and $0\leq r\leq k-p$.
This combinatorial identity can be proved by induction on $r$.
First notice that the case $r=0$ is trivial for every $p\geq 3$ and $k>p$.
Now let us assume that \eqref{proof-general-eqclaim} holds for $r$,
and consider the case $r+1$.
Using the property
$\binom{r+1}{m}=\binom{r}{m-1}+\binom{r}{m}$,
we have:
\be\label{proof-general-induction}
\begin{split}
&\sum_{m=0}^{r+1}(-1)^m\cdot \binom{r+1}{m}\biggl[
\sum_{j=-2g+p+1}^{-2g+k+1-(r+1)}
\binom{j+2g-3}{p-2}\\
&\qquad\times e_{r+1} (-2g+3-m,\cdots,-2g+2;j+1,\cdots,-2g+k+1-m)\biggr]\\
&=\sum_{m=0}^{r}(-1)^{m+1}\binom{r}{m}\biggl\{
\sum_{j=-2g+p+1}^{-2g+k+1-(r+1)}
\binom{j+2g-3}{p-2}\times\\
&\quad\biggl[ e_{r+1} (-2g+3-(m+1),\cdots,-2g+2;j+1,\cdots,-2g+k+1-(m+1))\\
&\qquad- e_{r+1} (-2g+3-m,\cdots,-2g+2;j+1,\cdots,-2g+k+1-m)\biggr]
\biggr\}.
\end{split}
\ee
Notice here
\ben
&&e_{r+1} (-2g+3-(m+1),\cdots,-2g+2;j+1,\cdots,-2g+k+1-(m+1))\\
&&-e_{r+1} (-2g+3-m,\cdots,-2g+2;j+1,\cdots,-2g+k+1-m)\\
&=&\big[\big(-2g+3-(m+1)\big)-\big(-2g+k+1-m\big)\big]\times\\
&&e_r (-2g+3-m,\cdots,-2g+2;j+1,\cdots,-2g+k+1-(m+1))\\
&=&(1-k)e_r (-2g+3-m,\cdots,-2g+2;j+1,\cdots,-2g+k-m),
\een
thus \eqref{proof-general-induction} equals to:
\ben
&&(1-k)\cdot\sum_{m=0}^{r}(-1)^{m+1}\binom{r}{m}\biggl\{
\sum_{j=-2g+p+1}^{-2g+k-r}
\binom{j+2g-3}{p-2}\\
&&\qquad\times e_r (-2g+3-m,\cdots,-2g+2;j+1,\cdots,-2g+k-m)\biggr]
\biggr\}.
\een
By the induction hypothesis for $(k-1,p,r)$, this becomes:
\ben
&&(1-k)(-1)\cdot\frac{\big((k-1)-1\big)!}{\big((k-1)-1-r\big)!}\biggl(
\binom{p-2}{p-2}+\cdots+\binom{(k-1)-2-r}{p-2}
\biggr)\\
&=&\frac{(k-1)!}{\big(k-1-(r+1)\big)!}\biggl(
\binom{p-2}{p-2}+\cdots+\binom{k-2-(r+1)}{p-2}
\biggr),
\een
which proves the claim \eqref{proof-general-eqclaim}.

Now applying \eqref{proof-general-eqclaim} to the left-hand-side of
\eqref{proof-general-eq},
we see that in order to prove \eqref{proof-general-eq}
we only need to show:
\be
\sum_{r=0}^{k-p}(-1)^r\binom{k-1}{r}\biggl(
\binom{p-2}{p-2}+\binom{p-1}{p-2}+\cdots+\binom{k-2-r}{p-2}
\biggr)=0
\ee
for $p\geq 3$ and $k>p$.
Using the identity
$\binom{k-1}{r}=\binom{k-2}{r}+\binom{k-2}{r-1}$,
the left-hand-side of this equation becomes:
\ben
&&\sum_{r=0}^{k-p}(-1)^r\binom{k-1}{r}\biggl(
\binom{p-2}{p-2}+\binom{p-1}{p-2}+\cdots+\binom{k-2-r}{p-2}
\biggr)\\
&=&\sum_{r=0}^{k-p}(-1)^r\binom{k-2}{r}\biggl(
\binom{p-2}{p-2}+\binom{p-1}{p-2}+\cdots+\binom{k-2-r}{p-2}
\biggr)\\
&&+\sum_{r=0}^{k-p-1}(-1)^{r+1}\binom{k-2}{r}\biggl(
\binom{p-2}{p-2}+\binom{p-1}{p-2}+\cdots+\binom{k-3-r}{p-2}
\biggr)\\
&=&\sum_{r=0}^{k-p}(-1)^r\binom{k-2}{r}\binom{k-2-r}{p-2}.
\een
Again using $\binom{k-2}{r}=\binom{k-3}{r}+\binom{k-3}{r-1}$,
we have:
\ben
&&\sum_{r=0}^{k-p}(-1)^r\binom{k-2}{r}\binom{k-2-r}{p-2}\\
&=&\sum_{r=0}^{k-p}(-1)^r\binom{k-3}{r}\binom{k-2-r}{p-2}
+\sum_{r=0}^{k-p-1}(-1)^{r+1}\binom{k-3}{r}\binom{k-3-r}{p-2}\\
&=&(-1)^{k-p}\cdot\binom{k-3}{k-p}\binom{p-2}{p-2}+\sum_{r=0}^{k-p-1}(-1)^r
\binom{k-3}{r}\binom{k-3-r}{p-3}\\
&=&\sum_{r=0}^{k-p}(-1)^r
\binom{k-3}{r}\binom{k-3-r}{p-3}.
\een
Inductively, we get:
\ben
&&\sum_{r=0}^{k-p}(-1)^r\binom{k-2}{r}\binom{k-2-r}{p-2}
=\sum_{r=0}^{k-p}(-1)^r\binom{k-3}{r}\binom{k-3-r}{p-3}\\
&=&\sum_{r=0}^{k-p}(-1)^r\binom{k-4}{r}\binom{k-4-r}{p-4}
=\cdots\cdots\\
&=&\sum_{r=0}^{k-p}(-1)^r\binom{k-p}{r}\binom{k-p-r}{p-p}
=\sum_{r=0}^{k-p}(-1)^r\binom{k-p}{r}=0,
\een
thus \eqref{proof-general-eq} indeed holds.
This completes the proof.
\end{proof}

\section{Generating Series of Orbifold Euler Characteristics of $\Mbar_{g,n}$
and $\cM_{g,n}$ as Tau-Functions of the KP Hierarchy }

\label{sec:KP}

In the \S \ref{sec4} we have seen that $\chi(\Mbar_{g,0})$ are related to the topological 1D gravity.
Now in this section let us relate them to the KP hierarchy via the topological 1D gravity,
and generalize this result to $\chi(\Mbar_{g,n})$ for $2g-2+n>0$.
See  \cite{djm, sa} for an introduction to the KP hierarchy.

\subsection{Topological 1D gravity and KP hierarchy}

First let us recall a basic result concerning the topological 1D gravity.

The partition function of the topological 1D gravity
(with $\lambda=1$) is defined to be
(\cite[(93)]{zhou1}):
\be
Z^{1D} : =  \frac{1}{\sqrt{2\pi} }
\int  dx \exp  \biggl( -\half x^2 + \sum_{n \geq 1} t_{n-1}
\frac{x^n}{n!}  \biggr),
\ee
where $t_n$ are the coupling constants.
The relation between the correlators
(i.e., coefficients of each term $t_{i_1}^{n_1}\cdots t_{i_k}^{n_k}$ in $\log Z^{1D}$)
and the coefficients $\{a_{g,n}^k\}$ of the refined orbifold Euler characteristics
has been discussed in \S \ref{sec-relation-1d}.

It has been shown by Nishigaki and Yoneya that:
\begin{Theorem}
[\cite{ny2}]
\label{thm-1d-KP}
$Z^{1D}$ is a tau-function of the KP hierarchy
with respect to the time variables $(T_1,T_2,\cdots)$
where:
\be
T_n = \frac{t_{n-1}}{n!}, \qquad n\geq 1.
\ee
\end{Theorem}

\subsection{$\chi(\Mbar_{g,0})$ and KP hierarchy}

Now the connection of the partition function $\widetilde{Z}$ defined by \eqref{eq:Gauss-Gk}
with the topological 1D gravity leads to a connection of $\chi(\Mbar_{g,0})$ to KP hierarchy.

Recall that although in \S \ref{sec4}
we have `pretended' that $V_k$ are some formal variables
in order to derive recursion relations for
the expressions of $G_k=G_k(V_1,V_2,\cdots)$,
they are actually generating series of $\chi(\cM_{g,n})$
and have definite expressions (see \eqref{eq-def-vertexV}).
Now if we take $\lambda = 1$ in \eqref{eq:Gauss-Gk},
then we get:
\be
\begin{split}
& \frac{1}{\sqrt{2\pi}}
\int\exp\biggl[\biggl(-\half x^2+\sum_{n\geq 0}
V_n(z) \cdot \frac{x^n}{n!}
\biggr)\biggr]dx \\
= & \exp\biggl( \sum_{k \geq 0} G_k(z) \biggr)\\
= & \exp \biggl( \sum_{g \geq 2} \chi(\Mbar_{g,0}) z^{2-2g}\biggr),
\end{split}
\ee
in other words,
\be
\begin{split}
&\exp \biggl( \sum_{g \geq 2} \big(\chi(\Mbar_{g,0})-\chi(\cM_{g,0})\big) z^{2-2g}\biggr)\\
=&\frac{1}{\sqrt{2\pi}}
\int\exp\biggl[\biggl(-\half x^2+\sum_{n\geq 1}
V_n(z) \cdot \frac{x^n}{n!}
\biggr)\biggr]dx.
\end{split}
\ee
So by Theorem \ref{thm-1d-KP} we have obtained our main result in this section:
\begin{Theorem}
The generating series $\sum\limits_{g \geq 2} \big(\chi(\Mbar_{g,0})-\chi(\cM_{g,0}) \big) z^{2-2g}$
is the logarithm of the tau-function $Z^{1D}$ of the KP hierarchy,
evaluated at time:
\be
T_n = \frac{1}{n!}V_n(z), \qquad n\geq 1,
\ee
where $V_n(z)$ are the generating series of the orbifold Euler characteristics
of $\cM_{g,n}$:
\begin{equation}
\label{eq-def-Vn}
\begin{split}
&V_n(z):=\sum_{g=1}^\infty \chi(\cM_{g,n})z^{2-2g-n},\qquad n=1,2;\\
&V_n(z):=\sum_{g=0}^\infty \chi(\cM_{g,n})z^{2-2g-n},\qquad n\geq 3,
\end{split}
\end{equation}
whose explicit formulas are given in Lemma \ref{lem-barnes-V}.
\end{Theorem}

\subsection{Generalization to $\chi(\Mbar_{g,n})$}

The above result can be generalized to a generating series of all $\chi(\Mbar_{g,n})$ with $2g-2+n>0$.
First let us recall the following formula observed by Bini and Harer \cite[(11)]{bh}:
\begin{Lemma}
[\cite{bh}]
Let $y,z$ be two formal variables,
then:
\be
\label{eq-generating-integral}
\begin{split}
&\exp\bigg(
\sum_{2g-2+n>0} \frac{1}{n!}
\chi(\Mbar_{g,n}) y^n z^{2-2g}\bigg)\\
=&
\frac{1}{\sqrt{2\pi}}
\int\exp\biggl(-\half (x-yz)^2 +\sum_{n\geq 0}
V_n(z) \cdot \frac{x^n}{n!}
\biggr)dx.
\end{split}
\ee
\end{Lemma}
\begin{proof}
Let us give a brief proof of this equation by analyzing the Feynman graph expansion.
The formal integral in the right-hand side can be expanded as:
\ben
&&\frac{1}{\sqrt{2\pi}}
\int\exp\biggl(-\half (x-yz)^2+\sum_{n\geq 0}
V_n(z) \cdot \frac{x^n}{n!}
\biggr)dx\\
&=&
\frac{e^{-\half y^2 z^2}}{\sqrt{2\pi}}
\int\exp\biggl(-\half x^2 +yz\cdot \frac{x}{1!} +\sum_{2g-2+n>0}
\chi(\cM_{g,n})z^{2-2g-n} \cdot \frac{x^n}{n!}
\biggr)dx\\
&=& e^{-\half y^2 z^2}\cdot
\exp\bigg( \sum_{\Gamma} \frac{\tilde w_\Gamma}{|\Aut(\Gamma)|}\bigg).
\een
Here the graphs $\Gamma$ are connected graphs with no external edges,
and each vertex of $\Gamma$ can be one of the following:
\begin{itemize}
\item[1)]
A stable vertex,
i.e.,
vertex of genus $g$ and valence $n$ with $2g-2+n>0$;
\item[2)]
A vertex of genus $0$ and valence $1$.
\end{itemize}
Let the weight $w_v$ of a stable vertex of genus $g$ and valence $n$ be $\chi(\cM_{g,n})z^{2-2g-n}$,
and the weight $w_v$ of a vertex of genus $0$ and valence $1$ be simply $yz$,
then the weight of a graph $\tilde w_\Gamma$ in the above formula is:
\be
\tilde w_\Gamma = \prod_{v:\text{ vertex}} w_v.
\ee
Now notice that a vertex $v$ of genus $0$ and valence $1$ can be regarded as an external edge in such graphs,
thus we can rewrite the above graph sum as follows:
\begin{equation*}
\sum_{\Gamma} \frac{\tilde w_\Gamma}{|\Aut(\Gamma)|}
=\frac{ \tilde w_{\Gamma_0}}{2} + \sum_{\substack{\Gamma:\text{ connected}\\ \text{stable graph}}}
\frac{\tilde w_\Gamma}{|\Aut(\Gamma)|}\\
=\frac{ y^2 z^2}{2} + \sum_{\substack{\Gamma:\text{ connected}\\ \text{stable graph}}}
 \frac{\tilde w_\Gamma}{|\Aut(\Gamma)|},
\end{equation*}
where $\Gamma_0$ is the graph consisting of two vertices of valence one,
with an internal edge connecting them.
Moreover,
using \eqref{eq-bini-harer} one has:
\begin{equation*}
\begin{split}
\sum_{\Gamma \in \cG_{g,n}^c}
\tilde w_\Gamma =&
\sum_{\Gamma \in \cG_{g,n}^c} z^{\sum_v (2-2g_v-\val(v))} \cdot (yz)^n\cdot \frac{1}{n!}\chi(\Mbar_{g,n})\\
=& \frac{1}{n!} \cdot y^n z^{2-2g} \chi(\Mbar_{g,n})
\end{split}
\end{equation*}
for every $2g-2+n>0$,
where the second equality holds by Euler's formula:
\begin{equation*}
1-(g-\sum_v g_v) = |V(\Gamma)| - |E(\Gamma)|
= |V(\Gamma)| -\half(\sum_v \val(v) -n)
= \frac{n}{2}+ \sum_v (1-\frac{\val(v)}{2}).
\end{equation*}
Thus the conclusion is proved.
\end{proof}

Now again by Theorem \ref{thm-1d-KP}, we have:
\begin{Theorem}
\label{thm-KP-gn}
The generating series
\ben
\sum_{2g-2+n>0} \frac{y^n z^{2-2g}}{n!} \cdot \chi(\Mbar_{g,n})
-\widetilde V_0(y,z)
\een
of the orbifold Euler characteristics of $\Mbar_{g,n}$
is the logarithm of the tau-function $Z^{1D}$ of the KP hierarchy,
evaluated at time:
\be
T_n = \frac{1}{n!}\widetilde V_n(y,z),
\qquad n\geq 1.
\ee
where $\widetilde V_n(y,z)$ are the following generating series of $\chi(\cM_{g,n})$:
\begin{equation}
\label{eq-def-Vn}
\begin{split}
&\widetilde V_0(y,z):=-\half y^2 z^2 + \sum_{g=2}^\infty \chi(\cM_{g,0})z^{2-2g};\\
&\widetilde V_1(y,z):=yz+\sum_{g=1}^\infty \chi(\cM_{g,1})z^{1-2g};\\
&\widetilde V_2(y,z):=\sum_{g=1}^\infty \chi(\cM_{g,2})z^{-2g};\\
&\widetilde V_n(y,z):=\sum_{g=0}^\infty \chi(\cM_{g,n})z^{2-2g-n},\qquad n\geq 3.
\end{split}
\end{equation}
\end{Theorem}

\section{Open-Closed Duality for the Orbifold Euler Characteristics of $\cM_{g,n}$ and $\Mbar_{g,n}$}
\label{sec-duality}

In this section we discuss a duality between the orbifold characteristic
of the moduli spaces  $\cM_{g,n}$ and $\Mbar_{g,n}$.
The main results are proved in \cite{wz4}.

In the previous work \cite{wz2},
the authors have introduced the notion of Fourier-like transforms for stable graphs,
which are a family of linear transformations on the infinite-dimensional vector space
spanned by all stable graphs.
Given a number $\epsilon$ (in $\bR$ or $\bC$, etc.) and a stable graph $\Gamma$,
one can construct a `stable graph $\Gamma^\epsilon$ of type $\epsilon$'
whose underlying graph is $\Gamma$.
Roughly speaking,
this $\Gamma^\epsilon$ stands for a linear combination of stable graph in the usual sense,
obtained by suitably gluing some `vertices of type $\epsilon$' together,
where a vertex of type $\epsilon$ of genus $g$ and valence $n$
is defined to be the linear combination $n!\cdot\wcF_{g,n}$ (see \eqref{abs-n-pt}).
The linear map $\Phi_\epsilon$ on the space spanned by all stable graphs
which takes $\Gamma$ to $\Gamma^\epsilon$ for every $\Gamma$ is called a Fourier-like transform.
See \cite[\S 6]{wz2} for this construction,
and we will not describe the details here.

Now fix a propagator $\kappa$.
When a Feynman rule of the form \eqref{Feynman-ordinary} is assigned to the stable graphs,
the Fourier-like transforms will be realized by a transformation
which takes the input data $\{F_{g,n}\}_{2g-2+n>0}$ to the output data
$\{\wF_{g,n}\}_{2g-2+n>0}$ of the realization of abstract QFT
where the propagator is taken to be $\epsilon\kappa$:
\ben
\wF_{g,n}:=\sum_{\Gamma\in\cG_{g,n}^c}
\frac{(\epsilon\kappa)^{|E(\Gamma)|}}{|\Aut(\Gamma)|}
\prod_{v\in V(\Gamma)}F_{g_v,\val_v} ,
\qquad
2g-2+n>0.
\een
We have interpreted this procedure as a transformation on the space of `field theories',
see \cite[\S 6]{wz2} for details.

One of the main results of that work is the following duality theorem:
\begin{Theorem}
[\cite{wz2}]
We have:
\ben
\Phi_{\epsilon}\circ\Phi_{-\epsilon}
=\Id.
\een
\end{Theorem}
As a corollary,
taking $\kappa=1$ and $\epsilon=1$,
we know that the inverse of the transformation $\{F_{g,n}\} \mapsto \{\tF_{g,n}\}$
given by the graph sum formula
\ben
\tF_{g,n}:= n! \cdot \sum_{\Gamma\in\cG_{g,n}^c}
\frac{1}{|\Aut(\Gamma)|}
\prod_{v\in V(\Gamma)}F_{g_v,\val_v}
\een
is simply (\cite[\S 4]{wz2}):
\ben
F_{g,n}= n!\cdot \sum_{\Gamma\in\cG_{g,n}^c}
\frac{(-1)^{|E(\Gamma)|}}{|\Aut(\Gamma)|}
\prod_{v\in V(\Gamma)}\tF_{g_v,\val_v}.
\een

Applying this duality theorem to Theorem \ref{bini-harer},
we obtain the following graph sum formula
which inverses \eqref{eq-bini-harer}:

\begin{Theorem}
[\cite{wz4}]
Assume $2g-2+n>0$,
then the orbifold Euler characteristic of $\cM_{g,n}/S_n$ is
given by the following graph sum formula:
\be
\label{eq-inverse-formula}
\chi(\cM_{g,n})=n!\cdot \sum_{\Gamma\in \cG^c_{g,n}}
\frac{(-1)^{|E(\Gamma)|}}{|\Aut(\Gamma)|}\prod_{v\in V(\Gamma)}\chi(\Mbar_{g_v,\val_v}),
\ee
where $g_v$ is the genus of a vertex $v$,
$\val_v$ is the valence of $v$,
and $|E(\Gamma)|$ is the number of internal edges of $\Gamma$.
\end{Theorem}

Now comparing this theorem with Theorem \ref{bini-harer},
we see that the graph sum formulas that
represents $\chi(\Mbar_{g,n})$ in terms of $\chi(\cM_{g,n})$
and represents $\chi(\cM_{g,n})$ in terms of $\chi(\Mbar_{g,n})$
is almost the same,
and the only difference is an additional factor $(-1)$ in the propagator.
This duality between the two types of orbifold Euler characteristics
is a new example of the open-closed duality.

Similar to the integral formula \eqref{eq-generating-integral},
we can derive the following formula from the above graph sum formula:
\begin{Corollary}
We have:
\be
\begin{split}
&\exp\bigg(
\sum_{2g-2+n>0} \frac{1}{n!}
\chi(\cM_{g,n}) y^n z^{2-2g}\bigg)\\
=&
\frac{1}{\sqrt{2\pi}}
\int\exp\biggl(\half (x-yz)^2 +\sum_{2g-2+n>0}
\frac{1}{n!}\chi(\Mbar_{g,n}) x^n z^{2-2g-n}
\biggr)dx.
\end{split}
\ee
\end{Corollary}
Notice that the quadratic term in $x$ in the exponential on the right-hand side
is $\frac{1}{2}x^2$.
This tells us that we need to understand this integral formally
as the above summation over graphs
where the weight of an internal edge is $(-1)$.
Now denote $\tilde x = -i x$,
then the above formal integral formula becomes:
\be
\begin{split}
&\exp\bigg(
\sum_{2g-2+n>0} \frac{1}{n!}
\chi(\cM_{g,n}) y^n z^{2-2g}\bigg)\\
=&
\frac{i}{\sqrt{2\pi}}
\int\exp\biggl(-\half (\tilde x+iyz)^2 +\sum_{2g-2+n>0}
\frac{1}{n!}\chi(\Mbar_{g,n}) (i\tilde x)^n z^{2-2g-n}
\biggr)d \tilde x,
\end{split}
\ee
thus by Theorem \ref{thm-1d-KP} one has the following dual version of Theorem \ref{thm-KP-gn}:

\begin{Theorem}
The generating series
\ben
-i\cdot \exp\bigg(
\frac{y^n z^{2-2g}}{n!} \chi(\cM_{g,n})
-\sum_{g=2}^\infty \chi(\Mbar_{g,0})z^{2-2g}
\bigg)
\een
is the tau-function $Z^{1D}$ of the KP hierarchy
evaluated at time:
\begin{equation*}
T_n = \frac{1}{n!}\bigg(
\frac{\delta_{n,0}}{2}\cdot y^2z^2 -\delta_{n,1}\cdot iyz
+\sum_{\substack{g\geq 0\\g>1-\frac{n}{2}}} \chi(\Mbar_{g,n})\cdot i^n z^{2-2g-n}
\bigg),
\qquad n\geq 1.
\end{equation*}

\end{Theorem}

\begin{Remark}
In \cite{wz4},
we have interpreted the open-closed duality \eqref{eq-inverse-formula}
as an analogue of the M\"obius inversion formula
(see Rota \cite{ro} for an introduction).

Recall that given a locally finite partially-ordered set $P$,
the zeta function $\zeta(x,y)$ on $P$ is defined by:
\begin{equation*}
\zeta (x,y):= \begin{cases}
1, & \text{if $x\leq y$};\\
0, & \text{otherwise},
\end{cases}
\end{equation*}
and the M\"obius function $\mu$ is defined to be the inverse of $\zeta$ in the incidence algebra.
Here $\mu$ is an integer-valued function.
Then given a real-valued function $f$ on $P$,
let $g$ be defined by:
\begin{equation*}
g(x):= \sum_{y\leq x} f(y)= \sum_{y\in P} f(y) \zeta(y,z) ,
\end{equation*}
then one has the following M\"obius inversion formula
(\cite{ro}):
\be
f(x) = \sum_{y\leq x} g(y) \mu(y,x).
\ee

In the work \cite{wz4},
we presented a similar construction to interpret \eqref{eq-inverse-formula}
in the following way.
First we introduce a partial ordering on the set of stable graphs
using the edge-contraction procedures.
Next,
we modify the zeta function such that the information about the orders of automorphism groups
of stable graphs are encoded in it.
We denote by $\tilde\zeta$ this generalized zeta function,
and by $\tilde\mu$ its inverse in the incidence algebra.
Then $\tilde\zeta$ and $\tilde\mu$ are rational-valued functions on the set of graphs.
Moreover,
we have an inversion formula for a pair of functions $(f,g)$
which are related to each other via $(\tilde\zeta,\tilde\mu)$.
In \cite{wz4}, we prove that this generalized M\"obius inverse formula
becomes the open-closed duality formula \eqref{eq-inverse-formula}
if one takes $f(\Gamma) := w_\Gamma$ where
\begin{equation*}
w_\Gamma = \prod_{v\in V(\Gamma)} \chi(\cM_{g_v,\val_v}).
\end{equation*}
\end{Remark}

\section{Concluding Remarks}

In this work,
we have studied the problem of computing the orbifold Euler characteristics
using two (mathematical) formalisms inspired by quantum physics.
We have obtained the following results:
\begin{itemize}
\item[1)]
We have introduced the refined orbifold Euler characteristics of $\Mbar_{g,n}$,
and derive a quadratic recursion and a linear recursion
using the formalism of abstract QFT for stable graphs developed in \cite{wz},
which enables us to obtain the numerical data completely.
\item[2)]
We have studied the structures of the refined orbifold Euler characteristics of $\Mbar_{g,0}$,
and showed that this problem is equivalent to the topological 1D gravity.
This enables us to solve $\chi(\Mbar_{g,0})$ using techniques developed in \cite{zhou1}
(such as the Virasoro constraints).
Moreover,
this method leads to a relation between $\chi(\Mbar_{g,n})$ and the KP hierarchy.
\item[3)]
We have solved the linear recursion and given the explicit formulas for the solutions.
This gives a way to represent $\chi(\Mbar_{g,n})$ using the coefficients of
the refined orbifold Euler characteristics of $\Mbar_{g,0}$ for fixed $g$.
\item[4)]
We have described a version of open-closed duality that represents the orbifold Euler characteristics
of $\cM_{g,n}$ and $\Mbar_{g,n}$ in terms of each other
via inversion formulas.

\end{itemize}

There are some more unexpected applications of physics ideas on this geometric problem.
In fact,
in \cite{zhou3, zhou4} the second author has developed another formalism called
the emergent geometry of KP hierarchy,
which allows one to construct some geometric structures emerging from a tau-function $\tau$
and derive some algorithms to compute the $n$-point functions associated to $\tau$.
Now due to the results in \S \ref{sec:KP},
this formalism will enable us to study such a problem in algebraic geometry
using techniques from integrable systems.
This is what the Witten Conjecture/Kontsevich Theorem \cite{kon1, wit1}
has inspired us to do.
We will report the applications of emergent geometry of KP hierarchy
to the computations of $\chi(\Mbar_{g,n})$ in a subsequent work.

\vspace{.2in}
{\bf Acknowledgements}.
The authors thank Professor Di Yang for helpful discussions.
The second author is partly supported by NSFC grant 11661131005 and 11890662.

\begin{appendices}

\section{Tables of Notations}

There are a lot of notations appearing in this work.
Here we make tables of some notations that have appeared in
several different subsections,
and list out the locations of their definitions to avoid confusion.

\begin{table}[H]
\caption{Notations in the abstract QFT and its realizations}
\label{tab:1}
\begin{tabular}{lll}
\hline\noalign{\smallskip}
Notation &Meaning &Location \\
\noalign{\smallskip}\hline\noalign{\smallskip}
$\cG_{g,n}^c$ & the set of connected stable graphs of genus $g$
& \S \ref{sec-pre-strat}\\
& with $n$ external edges \\
$\cG_{g,n}$ & the set of stable graphs of genus $g$
& \S \ref{sec-pre-strat}\\
& with $n$ external edges\\
$V(\Gamma)$ & the set of vertices of the graph $\Gamma$
&\S \ref{sec-pre-strat}\\
$E(\Gamma)$ & the set of internal edges of the graph $\Gamma$
&\S \ref{sec-pre-strat}\\
$E^{ext}(\Gamma)$ & the set of external edges of the graph $\Gamma$
&\S \ref{sec-pre-strat}\\
$\Aut(\Gamma)$
& group of automorphisms of the graph $\Gamma$
& \S \ref{sec-pre-absrec}\\
$g_v$ & genus of the vertex $v$ & \S \ref{sec2}\\
$\val_v$ & valence of the vertex $v$ & \S \ref{sec2}\\
$\wcF_{g}$ & abstract free energy of genus $g$
& \eqref{abs-fe}\\
$\wcF_{g,n}$ & abstract $n$-point function of genus $g$
& \eqref{abs-n-pt}\\
$K$ & edge-cutting operator
&\S \ref{sec-pre-absrec}\\
$\cD,\pd,\gamma$ & edge-adding operators
&\S \ref{sec-pre-absrec}\\
$F_{g,n}(t)$ & weight of vertices in a realization
& \eqref{Feynman-ordinary-v}\\
$\kappa$ & weight of an internal edge in a realization
& \eqref{Feynman-ordinary-e}\\
$\wF_g(t,\kappa)$ & realization of $\wcF_g$
& \eqref{eq-pre-realization-fe}\\
$\wF_{g,n}(t,\kappa)$ & realization of $\wcF_{g,n}$
& \eqref{eq-pre-realization-npt}\\
$\widehat Z (t,\kappa)$ & formal Gaussian integral representation
& \eqref{eq-pre-partition}\\
$\pd_\kappa$ & realization of $K$ (partial derivative w.r.t $\kappa$)
&\eqref{eq-pre-realizationK}\\
$\hpd$ & realization of $\pd$
& \S \ref{sec-pre-realization-rec}\\
$\hat{D}$ & realization of $\cD$
& \eqref{eq-pre-realizationD} \\
\noalign{\smallskip}\hline
\end{tabular}
\end{table}

\begin{table}[H]
\caption{Notations in computations of $\chi(\Mbar_{g,n})$}
\label{tab:2}
\begin{tabular}{lll}
\hline\noalign{\smallskip}
Notation &Meaning &Location \\
\noalign{\smallskip}\hline\noalign{\smallskip}
$F_{g,n}^{orb}(t)$ & weight of vertices (using Harer-Zagier formula)
& \eqref{eq-realization-vertices}\\
$\kappa$ & weight of an internal edge (a formal variable)
& \eqref{eq-realization-kappa}\\
$\chi_{g,n}(t,\kappa)$ & refined orbifold Euler characteristic
& \eqref{chi-g,n}\\
$\widetilde\chi_{g,n}(\kappa)$ &
$\widetilde\chi_{g,n}(\kappa):=\chi_{g,n}(1,\kappa)$
&\eqref{eq-def-tildechi}\\
$\widehat{Z}^{orb}(t,\kappa)$
& partition function for $\chi_{g,0}(t,\kappa)$
&\eqref{gaussian-chi(t,k)} \\
$d$ & realization of $\pd$ in this case
& \eqref{eq-realization-chi-pd}\\
$D$ & realization of $\cD$ in this case
&\eqref{eq-realization-D}\\
$\widetilde{D}$ & modification of $D$
& \eqref{eq-operator-tildeD} \\
$a_{g,n}^i$ & coefficients of the polynomial $\widetilde\chi_{g,n}(\kappa)$
& \eqref{eq-def-agni} \\
$G_k(z)$ & generating series for $a_{g,0}^k$
& \eqref{eq-def-generating-G}\\
$V_n(z)$ & generating series of $\chi(\cM_{g,n})$
& \eqref{eq-def-vertexV}\\
$\widetilde Z$ & partition function for $G_k(z)$
& \eqref{eq:Gauss-Gk} \\
$d_V$ & an operator to compute $G_k(V_1,V_2,\cdots)$
& \eqref{eq-def-dV}\\
$\tilde d_V$ & a modification of $d_V$
& \eqref{eq-def-tilde-d}\\
$\widehat A$ & operator to generate $G_k$ from $G_{k-1}$
& \eqref{eq-def-tilde-A} \\
$e_k$
&elementary symmetric functions
&\eqref{eq-def-ele-sym} \\
\noalign{\smallskip}\hline
\end{tabular}
\end{table}

\end{appendices}

\end{document}